\documentclass[10pt,centertags]{amsart}
\usepackage{amsfonts}
\usepackage{}
\usepackage{amsmath,amstext,amsthm,a4,amssymb,amscd}
\usepackage[mathscr]{eucal}
\usepackage{mathrsfs}
\usepackage{epsf}
\usepackage{tikz}

\usetikzlibrary{matrix,arrows,decorations.pathmorphing}
\numberwithin{equation}{section}

\newcommand{\field}[1]{\mathbb{#1}}
\newcommand{\Z}{\field{Z}}
\newcommand{\R}{\field{R}}
\newcommand{\C}{\field{C}}
\newcommand{\N}{\field{N}}

 \def\cC{\mathscr{C}}

 \def\cR{\mathscr{R}}
 
 \def\cE{\mathscr{E}}
\def\mC{\mathcal{C}}
\def\mA{\mathcal{A}}
\def\mB{\mathcal{B}}
\def\mD{\mathcal{D}}
\def\mE{\mathcal{E}}

\def\mG{\mathcal{G}}

\def\mR{\mathcal{R}}

\def\mU{\mathcal{U}}

\newcommand\mS{\mathcal{S}}

\def\Re{{\rm Re}}
\def\Im{{\rm Im}}

\def\la{\langle}
\def\ra{\rangle}

\DeclareMathOperator{\End}{End}

\DeclareMathOperator{\tr}{Tr}

\DeclareMathOperator{\ind}{ind}

\DeclareMathOperator{\ch}{ch}

\newtheorem{thm}{Theorem}[section]
\newtheorem{lemma}[thm]{Lemma}
\newtheorem{prop}[thm]{Proposition}

\theoremstyle{definition}
\newtheorem{rem}[thm]{Remark}
\theoremstyle{definition}
\newtheorem{defn}[thm]{Definition}
\newtheorem{assump}[thm]{Assumption}
\newcommand{\be}{\begin{eqnarray}}
\newcommand{\ee}{\end{eqnarray}}

\newcommand{\var}{\varepsilon}
\numberwithin{equation}{section}
\numberwithin{thm}{section}

\newcommand{\comment}[1]{}

\begin{document}

\title{Functoriality of equivariant eta forms}

\author{Bo Liu}

\address{Institut f\"ur Mathematik, Humboldt-Universit\"at zu Berlin,
Rudower Chaussee 25£¬  Johann von Neumann - Haus
12489 Berlin }
\email{boliumath@gmail.com}

\begin{abstract}
In this paper, we define the equivariant eta form of Bismut-Cheeger
for a compact Lie group
and establish a formula about the functoriality of equivariant eta forms  with respect to the composition of two submersions,
which is motivated by constructing the geometric model of equivariant differential K-theory.
\end{abstract}
\maketitle

{\bf Keywords:} Equivariant eta form; index theory and fixed point theory; Chern-Simons form.

{\bf 2010 Mathematics Subject Classification: } 58J20, 19K56, 58J28, 58J35.

\tableofcontents

\setcounter{section}{-1}

\section{Introduction} \label{s00}

 In order to find a well-defined index for a first order elliptic differential operator
over a compact manifold with nonempty boundary, Atiyah-Patodi-Singer \cite{MR0397797} introduced a boundary condition which is
particularly significant for applications. In this situation, an invariant of a first order self-adjoint
operator called the eta invariant, $\eta$, enters into the index formula. Formally, the eta invariant is equal to the number
of positive eigenvalues of the self-adjoint operator minus the number of negative eigenvalues.

Extending the work of Bismut-Freed \cite{MR861886}, which is a rigorous proof of Witten's holonomy theorem \cite{MR804460},
Bismut and Cheeger \cite{MR966608} studied the adiabatic limit for a fibration of closed Spin manifolds and found that
under the invertible assumption of the Dirac family along the fibers, the
adiabatic limit of the eta invariant of a Dirac operator on the total space is expressible in terms of a canonically
constructed differential form, $\tilde{\eta}$, on the base space.
Later, Dai \cite{MR1088332} extended this result to the case when the kernel of the Dirac family forms a vector bundle
over the base manifold.

This eta form of Bismut-Cheeger, $\tilde{\eta}$, is the higher degree version of the eta invariant $\eta$, i.e.,
it is exactly the
boundary correction term in the family index theorem for manifolds with boundary \cite{MR1042214,MR1052337,MR1472895}.
When the base space is a point, the eta form of Bismut-Cheeger is just the eta invariant of Atiyah-Patodi-Singer.
On the other hand, by \cite{MR2273508,MR966608,MR1088332},
when the dimension of the fibers are even,
the eta form serves as
a canonically constructed transgression  between the Chern character of the family index
and Bismut's explicit local index representative \cite{Bismut1985} of it.
We can also see it later by taking $g=1$ in (\ref{i18}).

Recently,
in the study of differential $K$-theory, the Bismut-Cheeger eta form naturally appears in the geometric model constructed by
Bunke and Schick \cite{Bunke2009} as a key ingredient. Moreover, the results in \cite{Bunke2009} are highly relied on the properties
of the eta form. In particular, the well-defined property of the push-forward map is based on
a formula about the functoriality of eta forms proved by Bunke and Ma \cite{MR2072502}, which is a family version of
\cite{MR966608}.
In \cite{Bunkea}, Bunke and Schick extend their geometric model to the orbifold case. It can also be regarded
as a geometric model for the equivariant differential $K$-theory for a finite group. Thus the equivariant
eta form appears naturally here and this motivates us to understand systematically the equivariant eta form.

In this paper, we define first the equivariant eta form
 when the fibration admits a fiberwise compact Lie group action
and establish a formula about the functoriality of equivariant eta forms which extends \cite[Theorem 5.11]{MR2072502} and \cite{MR966608}
to our case.
Note that Bunke-Ma in \cite{MR2072502} worked for the eta form associated to flat vector bundles,
and many analytic arguments are only sketched. Here we work on the equivariant situation,
thus we need to combine the equivariant local index technique to the different functional analysis technique in
analytic localization developed by Bismut and his collaborators
\cite{MR1305280,MR1316553,Bismut1997,MR1188532,MR2097553,MR1800127,MR1942300}.
We take this opportunity to give also the details of the analytic arguments missed in Bunke-Ma \cite{MR2072502}.

Let $\pi:W\rightarrow S$ be a smooth submersion of smooth manifolds
with closed oriented fiber $Z$, with $\dim Z=n$. Let $TZ=TW/S$ be the relative tangent bundle
to the fibers $Z$ with Riemannian metric $g^{TZ}$ and $T^HW$ be a horizontal subbundle of $TW$, such that $TW=T^HW\oplus TZ$.
Let $\nabla^{TZ}$ be the Euclidean connection on $TZ$ defined in (\ref{e01014}).
We assume that $TZ$ has a $\mathrm{Spin}^c$ structure. Let $L_Z$ be the complex line bundle associated to the
$\mathrm{Spin}^c$ structure of $TZ$ with a Hermitian metric $h^{L_Z}$ and a Hermitian connection $\nabla^{L_Z}$
(see \cite[Appendix D]{MR1031992}).

Let $G$ be a compact Lie group which acts fiberwisely on $W$ and as identity on $S$. We assume that the action
of $G$ preserves the $\mathrm{Spin}^c$ structure of $TZ$ and all metrics and connections are $G$-invariant.
Let $(E, h^E)$ be a $G$-equivariant Hermitian vector bundle over $W$ with a $G$-invariant Hermitian connection $\nabla^E$.
Let $D^Z$ be the fiberwise Dirac operator defined in (\ref{e01029}) and $B_t$ be the Bismut superconnection defined
in (\ref{e01041}).
For $\alpha\in \Omega^i(S)$, the differential form on $S$ with degree $i$, set
\begin{align}
\psi_S(\alpha)=\left\{
  \begin{array}{ll}
    \left(\frac{1}{2\pi\sqrt{-1}}\right)^{\frac{i}{2}}\cdot \alpha, & \hbox{if $i$ is even;} \\
    \frac{1}{\sqrt{\pi}}\left(\frac{1}{2\pi\sqrt{-1}}\right)^{\frac{i-1}{2}}\cdot \alpha, & \hbox{if $i$ is odd.}
  \end{array}
\right.
\end{align}
We define now the equivariant eta form (cf. (\ref{e01104}) and Definition \ref{e01083}).
\begin{defn}
Assume that $\dim \ker D^Z$ is locally constant on $S$.
For any $g\in G$, the equivariant eta form of Bismut-Cheeger is defined by
\begin{multline}\label{i17}
\tilde{\eta}_g(T^HW, g^{TZ}, h^{L_Z}, h^{E}, \nabla^{L_Z}, \nabla^{E})
\\
:=
\left\{
  \begin{aligned}%{ll}
   &\int_0^{\infty} \frac{1}{2\sqrt{-1}\sqrt{\pi}}\psi_S \tr_s\left[g\frac{\partial B_t}{\partial t}\exp(-B_t^2)\right] dt
   \in \Omega^{\mathrm{odd}}(S),
   & \hbox{if $n$ is even;} \\
   & \int_0^{\infty}\frac{1}{\sqrt{\pi}}\psi_S \tr^{\mathrm{even}}\left[g\frac{\partial B_t}{\partial t}\exp(-B_t^2)\right] dt
   \in \Omega^{\mathrm{even}}(S),
   & \hbox{if $n$ is odd.}
  \end{aligned}
\right.
\end{multline}
\end{defn}
The regularities of the integral in the right hand side of (\ref{i17}) are proved in Section \ref{s0104}.
Let $W^g$ be the fixed point set of $g$ on $W$. Then $W^g$ is a submanifold of $W$ and the restriction of $\pi$ on $W^g$ gives a fibration
$\pi:W^g\rightarrow S$ with fiber $Z^g$.
From Proposition \ref{e01072}, the fiber $Z^g$ is naturally oriented.
Furthermore, it verifies the following transgression.
\begin{multline}\label{i18}
d^S\tilde{\eta}_g(T^HW,g^{TZ}, h^{L_Z}, h^{E}, \nabla^{L_Z}, \nabla^E)
\\
=\left\{
  \begin{aligned}%{ll}
    &\int_{Z^g}\widehat{\mathrm{A}}_g(TZ,\nabla^{TZ})\wedge
\ch_g(L_Z^{1/2}, \nabla^{L_Z^{1/2}})\wedge \ch_g(E, \nabla^E)\\
  &\quad\quad\quad\quad\quad\quad\quad\quad\quad\quad-\ch_g(\ker D^Z, \nabla^{\ker D^Z}), & \hbox{if $n$ is even;} \\
    &\int_{Z^g}\widehat{\mathrm{A}}_g(TZ,\nabla^{TZ})\wedge
\ch_g(L_Z^{1/2}, \nabla^{L_Z^{1/2}})\wedge \ch_g(E, \nabla^E),& \hbox{if $n$ is odd.}
  \end{aligned}
\right.
\end{multline}
For the definition of characteristic forms in (\ref{i18}), see (\ref{e01051}), (\ref{e01132}) and (\ref{e01138}).

By (\ref{i17}), the equivariant eta form depends on the geometric data ($T^HW, g^{TZ},$ $h^{L_Z},$ $h^E,$$\nabla^{L_Z}, \nabla^E$).
When the geometric data vary, we have the anomaly formula for the equivariant eta forms.

\begin{thm}\label{i01}
Assume that there exists a smooth path connecting $(T^HW,g^{TZ},h^{L_Z}, h^E,$ $\nabla^{L_Z}, \nabla^E)$
and $(T^{'H}W,g^{'TZ},h^{'L_Z}, h^{'E},\nabla^{'L_Z}, \nabla^{'E})$ such that the dimension of the kernel of the Dirac family is
locally constant
(see Assumption \ref{e01146}).

i) When $n$ is odd, modulo exact forms on $S$, we have
\begin{multline}
\tilde{\eta}_g(T^{'H}W,g^{'TZ},h^{'L_Z}, h^{'E},\nabla^{'L_Z}, \nabla^{'E})-\tilde{\eta}_g(T^HW,g^{TZ},h^{L_Z}, h^E,\nabla^{L_Z}, \nabla^E)
\\
=\int_{Z^g}\widetilde{\widehat{\mathrm{A}}}_g(TZ, \nabla^{TZ}, \nabla^{'TZ})\wedge \ch_g(L_Z^{1/2}, \nabla^{L_Z^{1/2}})
\wedge \ch_g(E, \nabla^{E})
\\
+\int_{Z^g}\widehat{\mathrm{A}}_g(TZ, \nabla^{'TZ})\wedge \widetilde{\ch}_g(L_Z^{1/2}, \nabla^{L_Z^{1/2}},\nabla^{'L_Z^{1/2}})
\wedge \ch_g(E, \nabla^{E})
\\
+\int_{Z^g}\widehat{\mathrm{A}}_g(TZ, \nabla^{'TZ})\wedge \ch_g(L_Z^{1/2}, \nabla^{'L_Z^{1/2}})
\wedge \widetilde{\ch}_g(E, \nabla^{E},\nabla^{'E}).
\end{multline}

ii) When $n$ is even,
modulo exact forms on $S$, we have
\begin{multline}
\tilde{\eta}_g(T^{'H}W,g^{'TZ},h^{'L_Z}, h^{'E},\nabla^{'L_Z}, \nabla^{'E})-\tilde{\eta}_g(T^HW,g^{TZ},h^{L_Z}, h^E,\nabla^{L_Z}, \nabla^E)
\\
=\int_{Z^g}\widetilde{\widehat{\mathrm{A}}}_g(TZ, \nabla^{TZ}, \nabla^{'TZ})\wedge \ch_g(L_Z^{1/2}, \nabla^{L_Z^{1/2}})
\wedge \ch_g(E, \nabla^{E})
\\
+\int_{Z^g}\widehat{\mathrm{A}}_g(TZ, \nabla^{'TZ})\wedge \widetilde{\ch}_g(L_Z^{1/2}, \nabla^{L_Z^{1/2}},\nabla^{'L_Z^{1/2}})
\wedge \ch_g(E, \nabla^{E})
\\
+\int_{Z^g}\widehat{\mathrm{A}}_g(TZ, \nabla^{'TZ})\wedge \ch_g(L_Z^{1/2}, \nabla^{'L_Z^{1/2}})
\wedge \widetilde{\ch}_g(E, \nabla^{E},\nabla^{'E})
\\
-\widetilde{\ch}_g(\ker D^Z, \nabla^{\ker D^Z},\nabla^{'\ker D^Z}).
\end{multline}
\end{thm}
For the definitions of the Chern-Simons forms $\widetilde{\widehat{\mathrm{A}}}_g(TZ, \nabla^{TZ}, \nabla^{'TZ})$,
$\widetilde{\ch}_g(L_Z^{1/2}$, $\nabla^{L_Z^{1/2}},$ $\nabla^{'L_Z^{1/2}})$ and
$\widetilde{\ch}_g(\ker D^Z, \nabla^{\ker D^Z},\nabla^{'\ker D^Z})$ used here, see (\ref{e01127}).

For the reminder of this introduction, we shall consider the composition of two submersions.

Let $W$, $V$, $S$ be smooth manifolds. Let $\pi_1: W\rightarrow V$, $\pi_2: V\rightarrow S$ be smooth submersions with closed oriented
fiber $X$, $Y$. Then $\pi_3=\pi_2\circ \pi_1: W\rightarrow S$ is a smooth submersion with closed oriented fiber $Z$.
We have the diagram of fibrations:

\begin{center}\label{i3}
\begin{tikzpicture}[>=angle 90]
\matrix(a)[matrix of math nodes,
row sep=2em, column sep=2.5em,
text height=1.5ex, text depth=0.25ex]
{X&Z&W\\
&Y&V&S.\\};
\path[->](a-1-1) edge (a-1-2);
\path[->](a-1-2) edge node[left]{$\pi_1$} (a-2-2);
\path[->](a-1-2) edge (a-1-3);
\path[->](a-2-2) edge (a-2-3);
\path[->](a-1-3) edge node[left]{$\footnotesize{\pi_1}$} (a-2-3);
\path[->](a-2-3) edge node[above]{$\footnotesize{\pi_2\ \ }$} (a-2-4);
\path[->](a-1-3) edge node[above]{$\footnotesize{\ \pi_3}$} (a-2-4);
\end{tikzpicture}
\end{center}

Let $TX$, $TY$, $TZ$ be the relative tangent bundles.
We assume that $TX$ and $TY$ have the $\mathrm{Spin}^c$ structures with complex line bundles $L_X$ and $L_Y$ respectively.
Then $TZ$ have a $\mathrm{Spin}^c$ structure with a complex line bundle $L_Z$.
We take the geometric data ($T_1^HW, g^{TX}, h^{L_X}, \nabla^{L_X}$),
($T_2^HV, g^{TY}, h^{L_Y}, \nabla^{L_Y}$) and
($T_3^HW$, $g^{TZ}$, $h^{L_Z}$, $\nabla^{L_Z}$) with respect to
submersions  $\pi_1$, $\pi_2$ and $\pi_3$ respectively.
Let $\,^0\nabla^{TZ}$, $\,^0\nabla^{L_Z}$ be the connections on $TZ$, $L_Z$ defined in (\ref{e02114}), (\ref{e02005}).

Let $G$ be a compact Lie group which acts on $W$ such that for any $g\in G$,
$g\cdot\pi_1=\pi_1\cdot g$ and $\pi_3\cdot g=\pi_3$.
We assume that the action of $G$ preserves the $\mathrm{Spin}^c$ structures of $TX$, $TY$, $TZ$
and all metrics and connections are $G$-invariant.
Let $(E, h^E)$ be an equivariant Hermitian vector bundle over $W$ with equivariant Hermitian connection $\nabla^E$.
For any $g\in G$, let $T_1^H(W|_{V^g})=T_1^HW|_{V^g}\cap T(W|_{V^g})$ be the horizontal subbundle of $T(W|_{V^g})$.

The purpose of this paper is to establish the following result, which we state as Theorem \ref{e02084}.
\begin{thm}\label{i20}
If Assumption \ref{e02116} and \ref{e02119} hold, for any $g\in G$,
we have the following identity in $\Omega^*(S)/d^S\Omega^*(S)$,
\begin{multline}\label{i21}
\tilde{\eta}_{g}(T_3^HW,g^{TZ},h^{L_Z}, h^E,\nabla^{L_Z},\nabla^E)=\tilde{\eta}_g(T_2^HV,g^{TY},h^{L_Y},
h^{\ker D^X},\nabla^{L_Y},\nabla^{\ker D^X})
\\
+\int_{Y^g}{\rm \widehat{A}}_g(TY,\nabla^{TY})\wedge\ch_g(L_Y^{1/2}, \nabla^{L_Y^{1/2}})\wedge
\tilde{\eta}_g(T_1^H(W|_{V^g}),g^{TX},h^{L_X}, h^E, \nabla^{L_X}, \nabla^E)
\\
-\int_{Z^g}{\rm \widetilde{{\widehat{A}}}}_g(TZ,\nabla^{TZ},\,^0\nabla^{TZ})\wedge\ch_g(L_Z^{1/2}, \nabla^{L_Z^{1/2}})\wedge \ch_g(E,\nabla^E)
\\
-\int_{Z^g}{\rm {\widehat{A}}}_g(TZ,\,^0\nabla^{TZ})\wedge\widetilde{\ch}_g(L_Z^{1/2}, \nabla^{L_Z^{1/2}}, \,^0\nabla^{L_Z^{1/2}})\wedge \ch_g(E,\nabla^E).
\end{multline}
\end{thm}

Note that if $\ker D^Z$ is not locally constant, we can also construct an equivariant eta form when $\ind (D^Z)=0\in K_G^*(S)$
using the spectral section technique \cite{MR1472895}.
The functoriality of equivariant eta forms in this case is almost the same as Theorem \ref{i20}. We will
construct the equivariant differential $K$-theory and the push-forward map by equivariant eta forms with equivariant spectral section
in a comparative paper \cite{liu}
as applications of the results in this paper.

This paper is organized as follows.

In Section 1, we define the equivariant eta form
and prove the anomaly formula Theorem \ref{i01}.
In Section 2, we state our main result Theorem \ref{i20}.
In Section 3, we use some intermediate results, whose proofs are delayed to Section 4-8, to prove Theorem \ref{i20}.
Section 4-8 are devoted to the proofs of the intermediate results stated in Section 3.

To simplify the notations, we use the Einstein summation convention in this paper.

In the whole paper, we use the superconnection formalism of Quillen \cite{MR790678}. If $A$ is a $\Z_2$-graded algebra, and if
$a,b\in A$, then we will note $[a,b]$ as the supercommutator of $a$, $b$. If $B$ is another $\Z_2$-graded algebra,
we will note $A\widehat{\otimes}B$ as the $\Z_2$-graded tensor product.
If $A$, $B$ are not $\Z_2$-graded, sometimes, we also denote $A\widehat{\otimes}B$ by considering the whole algebra
as the even part.

For a trace class operator $P$ acting on a space $E$, if $E=E_+\oplus E_-$ is a $\Z_2$-graded space,
we denote by
\begin{align}\label{i22}
\tr_s[P]=\tr|_{E_+}[P]-\tr|_{E_-}[P].
\end{align}
If $\tr[P]$ takes value in differential forms, we denote by $\tr^{\mathrm{odd/even}}[P]$ the part of $\tr[P]$
which takes value in odd or even forms. We denote by
\begin{align}\label{i16}
\widetilde{\tr}[P]=
\left\{
  \begin{array}{ll}
    \tr_s[P], & \hbox{if $E$ is $\Z_2$-graded;} \\
    \tr^{\mathrm{odd}}[P], & \hbox{if $E$ is not $\Z_2$-graded.}
  \end{array}
\right.
\end{align}

For a vector bundle $\pi: W\rightarrow S$, we will often use the integration of the differential forms along the fiber $Z$ in this paper.
Since the fibers may be odd dimensional,
we must make precise our sign conventions.
If $\alpha$ is a differential form on $W$ which in local coordinates is given by
\begin{align}\label{e01135}
\alpha=dy^{p_1}\wedge \cdots\wedge dy^{p_q}\wedge \beta(x)dx^1\wedge\cdots\wedge dx^n,
\end{align}
%where $\{dx^i\}$ and $\{dy^p\}$ are the local frames of $T^*Z$ and $T^*S$,
we set
\begin{align}\label{e01136}
\int_Z\alpha=dy^{p_1}\wedge \cdots\wedge dy^{p_q} \int_Z\beta(x)dx^1\wedge\cdots\wedge dx^n.
\end{align}

\section{Equivariant eta form}\label{s01}

The purpose of this section is to  define the equivariant eta form and prove the anomaly formula.
In Section 1.1, we recall elementary results on Clifford algebras of arbitrary dimension.
In Section 1.2, we describe the geometry of fibration and introduce the Bismut superconnection and Bismut's
Lichnerowicz formula (cf. \cite{MR2273508, Bismut1985}). In Section 1.3, we explain the equivariant family local
index theorem. In Section 1.4, we define the
equivariant eta form when the dimension of the kernel of Dirac operators is locally constant.
In Section 1.5, we prove the anomaly formula.
In this section, we follow mainly from \cite{MR966608}.

\subsection{Clifford algebras}\label{s0101}

Let $C(V^n)$ denote the complex Clifford algebra of the real inner product space, $V^n$. Related to an orthonormal basis, $\{e_i\}$,
$C(V^n)$ is defined by the relations
\begin{align}\label{e01001}
e_ie_j+e_je_i=-2\delta_{ij}.
\end{align}
To avoid ambiguity, we denote by $c(e_i)$ the element of $C(V^n)$ corresponding to $e_i$.
We consider the group $\mathrm{Spin}_n^c$ as a multiplicative subgroup of the group of units of $C(V^n)$.
For the definition and the properties of the group $\mathrm{Spin}_n^c$, see \cite[Appendix D]{MR1031992}.

As a vector space,
\begin{align}\label{e01152}
C(V^n)\simeq \Lambda (V^n).
\end{align}
The Clifford multiplication on $\Lambda (V^n)$
is exterior multiplication minus interior multiplication. The elements
$c(e_I)=c(e_{i_1})\cdots c(e_{i_j})$, $I=\{i_1, \cdots, i_j\}\subset \{1, \cdots, n\}$, $i_1<\cdots<i_j$,
form a basis for $C(V^n)$. Put $|I|=j$. The subspace $C_0(V^n)$,  $C_1(V^n)$ spanned by those $c(e_I)$ with
$|I|$ even (resp. odd) give $C(V^n)$ the structure of a $\Z_2$-graded algebra.

For $n=2k$, even, up to isomorphism, $C(V^n)$ has a unique irreducible module, $\mS_n$, which has dimension $2^k$
and is $\Z_2$-graded. In fact, $C(V^{2k})\simeq \End(\mS_{2k})$. If $V$ is oriented, the element
\begin{align}\label{e01004}
\tau=
    (\sqrt{-1})^kc(e_1)\cdots c(e_{2k})%, & \hbox{$n=2k$}
\end{align}
is independent of the choice $\{e_i\}$ and satisfies
\begin{align}\label{e01006}
\tau^2=1.
\end{align}
Set $\mS_{\pm,n}=\{s\in \mS_n : \tau s=\pm s\}$.
We write $\tr_s[\,\cdot\,]$ for the supertrace of $C(V^{2k})$ on $\mS_n$ defined as (\ref{i22}).

If $n=2k-1$ is odd, $C(V^n)$ has two inequivalent irreducible modules, each of dimension $2^{k-1}$. For arbitrary $n$,
\begin{align}\label{e01007}
c(e_j)\rightarrow c(e_j)c(e_{n+1})
\end{align}
defines an isomorphism, $C(V^n)\simeq C_0(V^n\oplus \R)$. Thus, for $n$ odd, we can regard $\mS_{\pm,n+1}$ for $V^n\oplus \R$
as (inequivalent) modules over $C(V^n)$.
However, they are equivalent when restricted to $\mathrm{Spin}_n^c$.
For $V^{2k-1}$ oriented, the notation $\tr[\,\cdot\,]$ refers to the representation
$\mS_{+,2k}$.

By \cite[Lemma 1.22]{MR1042214}, if $n=2k$ is even, then
\begin{align}\label{e01008}
\tr_s[c(e_I)]=
\left\{
  \begin{array}{ll}
    (-\sqrt{-1})^{k}2^k, & \hbox{if $I=\{1,\cdots, 2k\}$;} \\
    0, & \hbox{if $I\neq\{1,\cdots, 2k\}$.}
  \end{array}
\right.
\end{align}
If $n=2k-1$ is odd and $|I|\geq 1$,
\begin{align}\label{e01090}
\tr[c(e_I)]=
\left\{
  \begin{array}{ll}
    (-\sqrt{-1})^{k}2^{k-1}, & \hbox{if $I=\{1,\cdots, 2k-1\}$;} \\
    0, & \hbox{if $I\neq\{1,\cdots, 2k-1\}$.}
  \end{array}
\right.
\end{align}

By (\ref{e01008}) and (\ref{e01090}), for $n$ odd, the trace $\tr$ behaves on the odd elements of $C(V^n)$
in exactly the same way as the supertrace $\tr_s$ on the even elements of $C(V^n)$ for $n$ even,
i.e. we must saturate all the elements $c(e_1), \cdots, c(e_n)$ to get a non-zero trace or supertrace.
It will be of utmost importance in the computations of the local index in Section 6. We set
\begin{align}\label{e01086}
\widetilde{c}_{V^n}=
\left\{
  \begin{array}{ll}
    \tr_s[c(e_1)\cdots c(e_n)], & \hbox{if $n$ is even;} \\
    \tr[c(e_1)\cdots c(e_n)], & \hbox{if $n$ is odd.}
  \end{array}
\right.
\end{align}

Let $W^m$ be another real inner product space with orthonormal basis $\{f_p\}$. Then as Clifford algebras,
\begin{align}\label{e01087}
C(V^n\oplus W^m)\simeq C(V^n)\widehat{\otimes}C(W^m).
\end{align}
By (\ref{e01008}), (\ref{e01090}) and (\ref{e01086}), we have
\begin{align}\label{e01154}
\widetilde{c}_{V^n\oplus W^m}=\left\{
  \begin{array}{ll}
    2\sqrt{-1}\,\widetilde{c}_{V^n}\cdot \widetilde{c}_{W^m}, & \hbox{if $n$, $m$ are odd;} \\
    \widetilde{c}_{V^n}\cdot \widetilde{c}_{W^m}, & \hbox{if others.}
  \end{array}
\right.
\end{align}

Finally, we note the effect of scaling the inner product $\la\cdot, \cdot\ra$ on $V$.
Fix any inner product, $\la\cdot, \cdot\ra$ and let $C_t(V)$ be the Clifford algebra associated to $t^{-1}\la\cdot, \cdot\ra$.
Then the map $t^{1/2}V\rightarrow V$ provides a natural isomorphism $C_t(V)\simeq C(V)$.
It also provides a natural isomorphism between the orthonormal frames $\{t^{1/2}e_i\}$ for $t^{-1}\la\cdot, \cdot\ra$
and $\{e_i\}$ for $\la\cdot, \cdot\ra$. Thus, the spinor $\mS$ for $\la\cdot, \cdot\ra$ is also an irreducible module for
$C_t(V)$ via the above isomorphism.
In the sequel, if $Z$ is a Riemannian $\mathrm{Spin}^c$ manifold, we will always assume that the space of spinors has been chosen
independent of the scaling parameter of the metric.

\subsection{Bismut superconnection and Lichnerowicz formula}\label{s0102}

Let $\pi:W\rightarrow S$ be a smooth submersion of smooth manifolds with closed oriented fiber $Z$, with $\dim Z=n$.
Let $TZ=TW/S$ be the relative tangent bundle to the fibers $Z$.

Let $T^HW$ be a horizontal subbundle of $TW$ such that
\begin{align}\label{e01010}
TW=T^HW\oplus TZ.
\end{align}
The splitting (\ref{e01010}) gives an identification
\begin{align}\label{e01011}
T^HW\cong \pi^*TS.
\end{align}
Let $P^{TZ}$ be the projection
\begin{align}\label{e01012}
P^{TZ}:TW=T^HW\oplus TZ \rightarrow TZ.
\end{align}

Let $g^{TZ}$, $g^{TS}$ be Riemannian metrics on $TZ$, $TS$. We equip $TW=T^HW\oplus TZ$ with the Riemannian metric
\begin{align}\label{e01013}
g^{TW}=\pi^*g^{TS}\oplus g^{TZ}.
\end{align}

Let $\nabla^{TW}$, $\nabla^{TS}$ be the Levi-Civita connections on $(W, g^{TW})$, $(S, g^{TS})$. Set
\begin{align}\label{e01014}
\nabla^{TZ}=P^{TZ}\nabla^{TW}P^{TZ}.
\end{align}
Then $\nabla^{TZ}$ is a Euclidean connection on $TZ$.
Let $\,^0\nabla^{TW}$ be the connection on $TW=T^HW\oplus TZ$ defined by
\begin{align}\label{e01015}
\,^0\nabla^{TW}=\pi^*\nabla^{TS}\oplus\nabla^{TZ}.
\end{align}
Then $\,^0\nabla^{TW}$ preserves the metric $g^{TW}$ in (\ref{e01013}).
Set
\begin{align}\label{e01017}
S=\nabla^{TW}-\,^0\nabla^{TW}.
\end{align}
Then $S$ is a $1$-form on $W$ with values in antisymmetric elements of $\End(TW)$.
Let $T$ be the torsion of $\,^0\nabla^{TW}$.
By \cite[Theorem 1.9]{Bismut1985}, we know that
$\nabla^{TZ}$, the torsion tensor $T$ and the $(3,0)$ tensor $\la S(\cdot)\cdot,\cdot\ra$ only depend on $(T^HW,g^{TZ})$,
where $\la \cdot, \cdot \ra=g^{TW}(\cdot, \cdot)$.

Let $C(TZ)$ be the Clifford algebra bundle of $(TZ, g^{TZ})$, whose fiber at $x\in W$ is the Clifford algebra
$C(T_xZ)$ of the Euclidean space $(T_xZ, g^{T_xZ})$.
We make the assumption that $TZ$ has a Spin$^c$ structure. Then there exists a complex line bundle $L_Z$
over $W$ such that $\omega_2(TZ)= c_1(L_Z) \mod (2)$.
Let $\mS(TZ,L_Z)$ be the fundamental
complex spinor bundle for $(TZ,L_Z)$, which
has a smooth action of $C(TZ)$ (cf. \cite[Appendix D.9]{MR1031992}).
Locally, the spinor $\mS(TZ,L_Z)$ may be written
as
\begin{align}\label{e01071}
\mS(TZ,L_Z) = \mS(TZ)\otimes L_Z^{1/2},
\end{align}
where $\mS(TZ)$ is the fundamental spinor bundle for the (possibly non-existent) spin
structure on $TZ$, and $L_Z^{1/2}$ is the (possibly non-existent) square root of $L_Z$.
Let $h^{L_Z}$ be the Hermitian metric on $L_Z$ and $\nabla^{L_Z}$ be the Hermitian connection on $(L_Z, h^{L_Z})$.
Let $h^{\mS_Z}$ be the Hermitian metric on $\mS(TZ,L_Z)$ induced by $g^{TZ}$ and $h^{L_Z}$
and $\nabla^{\mS_Z}$ be the connection on $\mS(TZ,L_Z)$ induced by $\nabla^{TZ}$ and $\nabla^{L_Z}$.
Then $\nabla^{\mS_Z}$ is a Hermitian connection on ($\mS(TZ,L_Z), h^{\mS_Z}$).
Moreover, it is a Clifford connection associated to $\nabla^{TZ}$, i.e., for any $U\in TW$, $V\in \cC^{\infty}(W,TZ)$,
\begin{align}\label{e01024}
\left[\nabla_U^{\mS_Z}, c(V)\right]=c\left(\nabla^{TZ}_UV\right).
\end{align}
If $n=\dim Z$ is even, the spinor $\mS(TZ, L_Z)$ is $\Z_2$-graded and the action of $TZ$ exchanges the $\Z_2$-grading.
Let $(E, h^E)$ be a Hermitian vector bundle over $W$, and $\nabla^E$ a Hermitian connection on $(E, h^E)$. Set
\begin{align}
\nabla^{\mS_Z\otimes E}=\nabla^{\mS_Z}\otimes 1+1\otimes \nabla^E.
\end{align}
Then $\nabla^{\mS_Z\otimes E}$ is a Hermitian connection on $(\mS(TZ,L_Z)\otimes E, h^{\mS_Z}\otimes h^E)$.

Let $\{e_i\}$, $\{f_p\}$ be local orthonormal frames of $TZ$, $TS$
and $\{e^i\}$, $\{f^p\}$ be the dual. Let $D^{Z}$ be the fiberwise Dirac operator
\begin{align}\label{e01029}
D^Z=c(e_i)\nabla_{e_i}^{\mS_Z\otimes E}.
\end{align}

For $b\in S$, let $\cE_{Z,b}$ be the set of smooth sections over $Z_b$ of $\mS(TZ, L_Z)\otimes E$. As in \cite{Bismut1985},
we will regard $\cE_Z$ as
 an infinite dimensional fiber bundle over $S$.

Let $dv_Z$ be the Riemannian volume element in the fiber $Z$. For any $b\in S$, $s_1, s_2\in \cE_{Z,b}$,
we can define the scalar product
\begin{align}\label{e01025}
\la s_1, s_2 \ra_{0}=\int_{Z_b}\la s_1(x), s_2(x) \ra dv_Z.
\end{align}
This scalar product could be naturally extended on $\Lambda(T^*S)\widehat{\otimes}\cE_Z$. We still denote it by $\la\cdot,\cdot\ra_0$.

If $U\in TS$, let $U^H\in T^HW$ be its horizontal lift in $T^HW$ so that $\pi_*U^H=U$.
For any $U\in TS$, $s\in \cC^{\infty}(S,\cE_Z)=\cC^{\infty}(W,\mS(TZ, L_Z)\otimes E)$,
we set
\begin{align}\label{e01026}
\nabla_{U}^{\cE_Z}s=\nabla_{U^H}^{\mS_Z\otimes E}s.
\end{align}
Then $\nabla^{\cE_Z}$ is a connection on $\cE_Z$, but need not preserve the scalar product $\la\cdot,\cdot\ra_0$ in (\ref{e01025}).
By \cite[Proposition 1.4]{MR853982}, for $U\in TS$, the connection
\begin{align}\label{e01028}
\nabla_U^{\cE_Z,u}:=\nabla_U^{\cE_Z}-\frac{1}{2}\la S(e_i)e_i, U^H \ra
\end{align}
preserves the scalar product $\la\cdot, \cdot\ra_{0}$.

If $U_1,U_2\in \cC^{\infty}(S,TS)$, by \cite[(1.30)]{Bismut1985}, we have
\begin{align}\label{e01131}
T(U_1,U_2)=-P^{TZ}[U_1^H,U_2^H].
\end{align}
We denote by
\begin{align}\label{e01030}
c(T)=\frac{1}{2}\,c\left(T(f_p^H, f_q^H)\right)f^p\wedge f^q\wedge.
\end{align}
By \cite[(3.18)]{Bismut1985}, the Bismut superconnection
\begin{align}\label{e01031}
B:\cC^{\infty}(S,\Lambda(T^*S)\widehat{\otimes}\cE_Z)\rightarrow\cC^{\infty}(S, \Lambda(T^*S)\widehat{\otimes}\cE_Z)
\end{align}
is defined by
\begin{align}\label{e01032}
B=D^Z+\nabla^{\cE_Z,u}-\frac{1}{4}c(T).
\end{align}
In fact, the Bismut superconnection only depends on the quadruple $(T^HW, g^{TZ}, \nabla^{L_Z}, \nabla^E)$.

In the sequel, if $A(U)$ is any $0$-order operator depending linearly on $U\in TW$, we define the operator
\begin{align}\label{e01036}
\left(\nabla_{e_i}^{\mS_Z\otimes E}+A(e_i)\right)^2
\end{align}
as follows: if $\{e_i(x)\}_{i=1}^n$ is any (locally defined) smooth orthonormal frame of $TZ$,
then
\begin{multline}\label{e01037}
\left(\nabla_{e_i}^{\mS_Z\otimes E}+A(e_i)\right)^2
\\
:=\sum_{i=1}^n\left(\nabla_{e_i(x)}^{\mS_Z\otimes E}+A(e_i(x))\right)^2
-\nabla^{\mS_Z\otimes E}_{\sum_{i=1}^n\nabla_{e_i}^{TZ}e_i}-A\left(\sum_{i=1}^n\nabla_{e_i}^{TZ}e_i\right).
\end{multline}

Let $R^{TZ}$, $R^{L_Z}$, $R^E$ and $R^{\mS_Z\otimes E}$ be the curvatures of $\nabla^{TZ}$, $\nabla^{L_Z}$, $\nabla^E$
and $\nabla^{\mS_Z\otimes E}$ respectively.
By (\ref{e01071}), we have
\begin{align}\label{e01038}
R^{\mS_Z\otimes E}=\frac{1}{4}\la R^{TZ}e_i, e_j\ra c(e_i)c(e_j)+\frac{1}{2} R^{L_Z}+R^E.
\end{align}

For $t>0$, we denote $\delta_t$ the operator  on $\Lambda^i (T^*S)\widehat{\otimes} \cE_Z$ by multiplying differential forms
by $t^{-i/2}$. Set
\begin{align}\label{e01041}
B_t:=\sqrt{t}\,\delta_t\circ B\circ \delta_t^{-1}.
\end{align}
Then from (\ref{e01032}) and (\ref{e01041}), we get
\begin{align}\label{e01042}
B_t=\sqrt{t}D^X+\nabla^{\cE_Z,u}-\frac{1}{4\sqrt{t}}c(T).
\end{align}
Let $K^Z$ be the scalar curvature of the fibers $(TZ, g^{TZ})$. We have the Bismut's Lichnerowicz formula
(see \cite[Theorem 10.17]{MR2273508}, \cite[Theorem 3.5]{Bismut1985}),
\begin{multline}\label{e01043}
B_t^2=-\left(\sqrt{t}\nabla_{e_i}^{\mS_Z\otimes E}+\frac{1}{2}\la S(e_i)e_j, f_p^H\ra c(e_j)f^p\wedge
+\frac{1}{4\sqrt{t}}\la S(e_i)f_p^H, f_q^H\ra f^p\wedge f^q\wedge\right)^2
\\
+\frac{t}{4}K^Z+\frac{t}{2}\left(\frac{1}{2}R^{L_Z}+R^E\right)(e_i,e_j)c(e_i)c(e_j)
+\sqrt{t}\left(\frac{1}{2}R^{L_Z}+R^E\right)(e_i,f_p^H)c(e_i)f^p\wedge
\\
+\frac{1}{2}\left(\frac{1}{2}R^{L_Z}+R^E\right)(f_p^H,f_q^H)f^p\wedge f^q\wedge.
\end{multline}
In particular, $B_t^2$ is a $2$-order elliptic differential operator along the fiber $Z$.
Let $\exp(-B_t^2)$ be the family of heat operators associated to the fiberwise elliptic operator $B_t^2$ in (\ref{e01043}).
From \cite[Theorem 9.50]{MR2273508}, we know that $\exp(-B_t^2)$ is a smooth family of smoothing operators.

\subsection{Compact Lie group action and equivariant family local index theorem}\label{s0103}

Let $G$ be a compact Lie group which acts on $W$ such that for any $g\in G$, $\pi\circ g=\pi$.
So it acts trivially on $S$.
We assume that the action of $G$ preserves the splitting (\ref{e01010}), the $\mathrm{Spin}^c$ structure of $TZ$
and $g^{TZ}$, $h^{L_Z}$, $\nabla^{L_Z}$ are $G$-invariant.
We assume that $E$ is a $G$-equivariant complex vector bundle and $h^E$, $\nabla^E$ are $G$-invariant.
So the action of $G$ commutes with the Bismut superconnection $B$ in (\ref{e01032}).

Take $g\in G$ and set
\begin{align}\label{e01045}
W^g=\{x\in W: gx=x\}.
\end{align}
Then $W^g$ is a submanifold of $W$ and $\pi:W^g\rightarrow S$ is a fiber bundle with closed fiber $Z^g$.
Let $N$ denote the normal bundle of $W^g$ in $W$, then $N=TZ/TZ^g$.
Since $g$ preserves the $\mathrm{Spin}^c$ structure, it preserves the orientation of $TZ$.
So the normal bundle $N$ is even dimensional.
We denote the differential of $g$ by $dg$
which gives a bundle isometry $dg: N\rightarrow N$. Since $g$ lies in a compact abelian Lie group,
we know that there is an orthonormal decomposition of smooth vector bundles on $W^g$
\begin{align}\label{e01046}
N=N(\pi)\oplus\oplus_{0<\theta<\pi}N(\theta),
\end{align}
where $dg|_{N(\pi)}=-\mathrm{id}$ and for each $\theta$, $0<\theta<\pi$, $N(\theta)$ is a complex vector bundle
on which $dg$ acts by multiplication by $e^{\sqrt{-1}\theta}$, and $\dim N(\pi)$ is even.
By the following proposition, $Z^g$ and $N$ are all naturally oriented. This proposition is a modification of \cite[Theorem 6.14]{MR2273508}.

\begin{prop}\label{e01072}
Let $Z$ be a closed oriented manifold and $G$ be a compact Lie group. If $TZ$ has a $G$-equivariant $\mathrm{Spin}^c$ structure,
then for each $g\in G$,
$Z^g$ is naturally oriented.
\end{prop}
\begin{proof}
We fix a connected component of $Z^g$ and assume that the dimension of the normal bundle $N$ of this connected component is $2k$.
By (\ref{e01046}), on $N$, the matrix of $g$ has diagonal blocks
\begin{align}\label{e01150}
\left(
  \begin{array}{cc}
    \cos(\theta_j) & -\sin(\theta_j) \\
    \sin(\theta_j) & \cos(\theta_j) \\
  \end{array}
\right),\quad j=1,2,\cdots,k,\quad 0<\theta_j\leq \pi.
\end{align}
By the definition of the $\mathrm{Spin}^c$ group, the action of $g$ on the spinor is given by
\begin{align}\label{e01151}
g=\alpha\cdot\prod_{j=1}^{k}(\cos(\theta_j/2)+\sin(\theta_j/2)c(e_{2j-1})c(e_{2j})),
\end{align}
where $\alpha\in \C$, $|\alpha|=1$.
Let $\sigma : C(N)\rightarrow \Lambda(N)$ be the isomorphism in (\ref{e01152}).
For $\beta\in \Lambda (N)$, let $[\beta]_{2k}$ denote the degree $2k$ part of $\beta$.
Since $\alpha$ and $\theta_j$ are locally constant on $Z^g$, the term
\begin{align}
\alpha^{-1}[\sigma(g)]_{2k}=\left(\prod_{j=1}^{k}\sin(\theta_j/2)\right)e^{1}\wedge \cdots\wedge e^{2k}
\end{align}
gives a non-zero section of $\Lambda^{2k}(N)$. Then it gives a canonical orientation of $N$.
The canonical orientation of $Z^g$ can be obtained by the orientations of $Z$ and $N$.

The proof of Proposition \ref{e01072} is complete.
\end{proof}

Since $g^{TZ}$ is $G$-invariant, the connection $\nabla^{TZ}$ preserves the decomposition of smooth vector bundles on $W^g$
\begin{align}\label{e01047}
TZ|_{W^g}=TZ^g\oplus\oplus_{0<\theta\leq \pi}N(\theta).
\end{align}
Let $\nabla^{TZ^g}$, $\nabla^N$ and $\nabla^{N(\theta)}$ be the corresponding induced connections on $TZ^g$, $N$ and $N(\theta)$,
and let $R^{TZ^g}$, $R^N$ and $R^{N(\theta)}$ be the corresponding curvatures. Here we consider $N(\theta)$ as a real vector bundle.
We have the decompositions on $W^g$:
\begin{align}\label{e01048}
\nabla^{TZ}|_{W^g}=\nabla^{TZ^g}\oplus\nabla^N, \quad \nabla^N=\oplus_{0<\theta\leq \pi}\nabla^{N(\theta)},
\end{align}
and
\begin{align}\label{e01049}
R^{TZ}|_{W^g}=R^{TZ^g}\oplus R^N, \quad R^N=\oplus_{0<\theta\leq \pi}R^{N(\theta)}.
\end{align}

For $0<\theta\leq \pi$, we write
\begin{align}\label{e01050}
\widehat{\mathrm{A}}_{\theta}\left(N(\theta), \nabla^{N(\theta)}\right)=
\left((\sqrt{-1})^{\frac{1}{2}\dim_{\R} N(\theta)}\mathrm{det}^{\frac{1}{2}}\left(1-g
\exp\left(\frac{\sqrt{-1}}{2\pi}R^{N(\theta)}\right)\right)\right)^{-1}.
\end{align}
Set
\begin{align}\label{e01051}
\begin{split}
&\widehat{\mathrm{A}}\left(TZ^g, \nabla^{TZ^g}\right)=
\mathrm{det}^{\frac{1}{2}}\left(\frac{\frac{\sqrt{-1}}{4\pi}R^{TZ^g}}{\sinh \left(\frac{\sqrt{-1}}{4\pi}R^{TZ^g}\right)}\right),
\\
&\widehat{\mathrm{A}}_g(TZ,\nabla^{TZ})=\widehat{\mathrm{A}}\left(TZ^g, \nabla^{TZ^g}\right)\cdot
\prod_{0<\theta\leq \pi}\widehat{\mathrm{A}}_{\theta}\left(N(\theta), \nabla^{N(\theta)}\right)\in \Omega^{4*}(W^g,\C).
\end{split}
\end{align}
Note that for any Euclidean connection
$\nabla$ on $(TZ, g^{TZ})$, we can also define the characteristic form $\widehat{\mathrm{A}}_g(TZ,\nabla)$ as in (\ref{e01051}).
Let $\widehat{\mathrm{A}}_g(TZ)\in H^{4*}(W^g, \C)$ denote the cohomology class of $\widehat{\mathrm{A}}_g(TZ,\nabla)$.
If $E$ is $\Z_2$-graded, we assume that the $G$-action and $\nabla^E$ preserve the $\Z_2$-grading.
Set
\begin{align}\label{e01132}
\ch_g(E, \nabla^E)=
\left\{
  \begin{aligned}%{ll}
    &\tr\left[g\exp\left(\frac{\sqrt{-1}}{2\pi}R^E|_{W^g}\right)\right], & \hbox{if $E$ is not $\Z_2$-graded;} \\
    &\tr_s\left[g\exp\left(\frac{\sqrt{-1}}{2\pi}R^E|_{W^g}\right)\right], & \hbox{if $E$ is $\Z_2$-graded.}
  \end{aligned}
\right.
\end{align}
Let $\ch_g(E)\in H^{2*}(W^g,\C)$ denote the cohomology class of $\ch_g(E, \nabla^E)$.
By Chern-Weil theory \cite{MR1864735}, the classes $\widehat{\mathrm{A}}_g(TZ)$ and $\ch_g(E)$ are independent of $\nabla$
and $\nabla^{E}$. Furthermore, if $S$ is compact, the equivariant Chern character in (\ref{e01132}) descends to a ring homomorphism
\begin{align}\label{e01133}
\ch_g: K_G^0(W^g)\rightarrow H^{2*}(W^g, \C),
\end{align}
where $K_G^0(W^g)$ is the equivalent $K^0$ group of $W^g$.

Assume that $n$ is even. If $S$ is compact, the index bundle $\ind (D^Z)$ is an element of $K_G^0(S)$.
Under the equivariant Chern character map (\ref{e01133}), for any $g\in G$, we have
\begin{align}\label{e01054}
\ch_g(\ind(D^Z))\in H^{2*}(S,\C).
\end{align}
Since the fiber is even-dimensional, the spinor $\mS(TZ,L_Z)$ is $\Z_2$-graded, i.e., $\mS(TZ,L_Z)=\mS_+(TZ,L_Z)\oplus \mS_-(TZ,L_Z)$.
Note that if $\dim \ker D^Z$ is locally constant,
\begin{align}\label{e01201}
\ind(D^Z)=\ker D^Z_+-\ker D^Z_-\in K_G^0(S),
\end{align}
where $D^Z_{\pm}$ is the restriction of $D^Z$ on $\mS_{\pm}(TZ, L_Z)\otimes E$.

Let $\cE_{Z,\pm}$ be the set of smooth sections of $\mS_{\pm}(TZ,L_Z)\otimes E$ over $W$. Then $\cE_Z=\cE_{Z,+}\oplus\cE_{Z,-}$ is a $\Z_2$-graded
infinite dimensional vector bundle over $S$
and $\Lambda(T^*S)\widehat{\otimes}\End(\cE_Z)$
is also $\Z_2$-graded. We extend $\tr$, $\tr_s$ to the trace class element $A\in \Lambda(T^*S)\widehat{\otimes}\End(\cE_Z)$,
which take values in $\Lambda(T^*S)$. We use the convention that if $\omega\in \Lambda(T^*S)$,
\begin{align}\label{e01056}
\tr[\omega A]=\omega\tr[A],\quad \tr_s[\omega A]=\omega\tr_s[A].
\end{align}

Let $i:S\rightarrow S^1\times S$ be a $G$-equivariant inclusion map. It is well known that if the $G$-action on $S^1$ is trivial,
\begin{align}\label{d124}
K_G^1(S)\simeq \ker\left(i^*:K_G^0(S^1\times S)\rightarrow K_G^0(S)\right).
\end{align}
By (\ref{d124}), for $x\in K_G^1(S)$, we can regard $x$ as an element $x'$ in $K_G^0(S^1\times S)$. The odd equivariant Chern character map
\begin{align}\label{e01134}
\ch_g: K_G^1(S)\rightarrow H^{\mathrm{odd}}(S, \C)
\end{align}
is defined by 
\begin{align}\label{e01066}
\ch_g(x)=
\left[\int_{S^1}\ch_g(x')\right]\in H^{\mathrm{odd}}(S,\C).
\end{align}
Here we use the sign convention (\ref{e01136}) in this integration.

If $n$ is odd,  the fibrewise Dirac operator $D^Z$ is a family of equivariant self-adjoint Fredholm operators.
Set
\begin{align}\label{e01101}
D_{\theta}^Z=\left\{
  \begin{array}{ll}
    I \cos \theta +\sqrt{-1}D^Z \sin \theta, & \hbox{if $0\leq \theta \leq \pi$;} \\
    (\cos\theta+\sqrt{-1}\sin\theta)I, & \hbox{if $\pi\leq \theta \leq 2\pi$}
  \end{array}
\right.
\end{align}
(see \cite[(3.3)]{MR0397799}).
If $S$ is compact,
then $\ind(\{D_{\theta}^Z\})\in K_G^0(S^1\times S)$. Since the restriction of $D_{\theta}^Z$ to $\{0\}\times S$ is trivial, so
it can be regarded as an element of $K_G^1(S)$. From \cite{MR0397799} and \cite{MR0234452}, the definition of the index of
$D^Z$ is
\begin{align}\label{e01065}
\ind(D^Z):=\ind(\{D_{\theta}^{Z}\})\in K_G^1(S).
\end{align}
%Note that as an analogue of (\ref{e01133}), for any $g\in G$, there is a homomorphism

%defined by the suspension.
%In our case,

% The constant $(2\pi\sqrt{-1})^{-1}$ here
%is chosen to normalize the constant in Theorem \ref{e01061}.

When the fiber is odd dimensional, the spinor $\mS(TZ,L_Z)$ is not $\Z_2$-graded.
For a trace class element $A\in \Lambda(T^*S)\otimes\End(\cE_Z)$,
we also use the convention as in (\ref{e01056}) that if $\omega\in \Lambda(T^*S)$,
\begin{align}\label{e01137}
\tr[\omega A]=\omega\tr[A].
\end{align}
It is compatible with the sign convention in (\ref{e01136}).

For $\alpha\in \Omega^i(S)$, set
\begin{align}\label{e01059}
\psi_S(\alpha)=\left\{
  \begin{array}{ll}
    \left(\frac{1}{2\pi\sqrt{-1}}\right)^{\frac{i}{2}}\cdot \alpha, & \hbox{if $i$ is even;} \\
    \frac{1}{\sqrt{\pi}}\left(\frac{1}{2\pi\sqrt{-1}}\right)^{\frac{i-1}{2}}\cdot \alpha, & \hbox{if $i$ is odd.}
  \end{array}
\right.
\end{align}
Comparing with (\ref{e01132}), for the locally defined line bundle $L_Z^{1/2}$, we write
\begin{align}\label{e01138}
\ch_g(L_Z^{1/2}, \nabla^{L_Z^{1/2}}):=g\cdot\exp\left(\frac{\sqrt{-1}}{4\pi}R^{L_Z}|_{W^g}\right)\in \Omega^{2*}(W^g, \C)
\end{align}
and $\ch_g(L_Z^{1/2})\in H^{2*}(W^g, \C)$ as the corresponding cohomology class.
Denote by $\pi_*:H^*(W^g, \C)\rightarrow H^*(S, \C)$ the integration along the fiber $Z^g$ with the sign convention (\ref{e01136}).
Recall that the trace operator $\widetilde{\tr}$ is defined in (\ref{i16}).
We give the equivariant family local index theorem as follows.

\begin{thm}\label{e01061}
For any $t>0$ and $g\in G$, the differential form $\psi_S\widetilde{\tr}[g\exp(-B_t^2)]\in \Omega^{*}(S)$ is closed and its cohomology class
is independent of $t$. As $t\rightarrow 0$,
\begin{align}\label{e01062}
\lim_{t\rightarrow 0}\psi_S\widetilde{\tr}[g\exp(-B_t^2)]=\int_{Z^g}\widehat{\mathrm{A}}_g(TZ,\nabla^{TZ})\wedge
\ch_g(L_Z^{1/2}, \nabla^{L_Z^{1/2}})\wedge
\ch_g(E, \nabla^{E}).
\end{align}
If $S$ is compact,  the differential form $\psi_S\widetilde{\tr}[g\exp(-B_t^2)]$
represents $\ch_g(\ind(D^Z))$ in (\ref{e01054}) or (\ref{e01066}).
In $H^*(S,\C)$,
\begin{align}\label{e01063}
\ch_g(\ind(D^Z))=\pi_*\left\{\widehat{\mathrm{A}}_g(TZ)\ch_g(L_Z^{1/2})\ch_g(E)\right\}.
\end{align}
\end{thm}
\begin{proof}
If $n$ is even, the proof is the same as that of \cite[Theorem 1.1]{MR1756105}. If $n$ is odd,
the proof follows from \cite[Theorem 2.10]{MR861886} and the even case.
\end{proof}

\subsection{Equivariant eta form}\label{s0104}

In this subsection, we define the equivariant eta form when $\dim \ker D^Z$ is locally constant.
We will proceed as the proof of \cite[Theorem 10.32]{MR2273508} as follows.

Let $\widehat{S}=\R_+\times S$ and $\mathrm{pr}:\widehat{S}\rightarrow S$ be the projection.
We consider the bundle $\widehat{\pi}:\widehat{W}:=\R_+\times W\rightarrow\widehat{S}$ together with the canonical projection
$\mathrm{Pr}:\widehat{W}\rightarrow W$. Set $T^H\widehat{W}=T(\R_+)\oplus\mathrm{Pr}^*(T^HW)$.
Then $T^H\widehat{W}$ is a horizontal subbundle of $T\widehat{W}$ as in (\ref{e01010}).
We fix the vertical metric $\widehat{g}^{TZ}$ which restricts to $t^{-1}g^{TZ}$
over $\{t\}\times W$. Let $\widehat{C}(TZ)$ be the Clifford algebra bundle associated to $\widehat{g}^{TZ}$.
Then $\widehat{\mS}(TZ, \mathrm{Pr}^*L_{Z}):=\mathrm{Pr}^*\mS(TZ,L_Z)$ is the spinor of $\widehat{C}(TZ)$ by the assumption in
the end of Section \ref{s0101}.
Let $h^{\widehat{L}_Z}=\mathrm{Pr}^*h^{L_Z}$ and $\nabla^{\widehat{L}_Z}=\mathrm{Pr}^*\nabla^{L_Z}$.
Let $\widehat{E}=\mathrm{Pr}^*E$, $h^{\widehat{E}}=\mathrm{Pr}^*h^{E}$ and $\nabla^{\widehat{E}}=\mathrm{Pr}^*\nabla^{E}$.
We naturally extend the $G$-actions to this case such that the $G$-action is identity on $\R_+\times S$.
We will mark the objects associated to ($T^H\widehat{W}, \widehat{g}^{TZ}, h^{\widehat{L}_Z},
h^{\widehat{E}}, \nabla^{\widehat{L}_Z}, \nabla^{\widehat{E}}$) by $\ \widehat{}\ $.

For $t\in \R_+$, the fiberwise Dirac operator $D^{\widehat{Z}}$ on $\{t\}\times Z$ is $t^{1/2}D^Z$.
By (\ref{e01028}), $\nabla^{\widehat{\cE_Z},u}=\nabla^{\cE_Z,u}-\frac{n}{4t}dt$.
Since $B_t$ in (\ref{e01042}) is just the Bismut superconnection associated to ($T^HW, t^{-1}g^{TZ}, \nabla^{L_Z}, \nabla^E$),
from (\ref{e01032}) and (\ref{e01042}), the Bismut superconnection associated to ($T^H\widehat{W}, \widehat{g}^{TZ}, \nabla^{\widehat{L}_Z},
\nabla^{\widehat{E}}$)
is
\begin{align}\label{e01102}
\widehat{B}|_{(t,b)}=B_t+dt\wedge\frac{\partial}{\partial t}-\frac{n}{4t}dt,
\end{align}
for $(t,b)\in \widehat{S}$.
Note that the extended $G$-action commutes with the Bismut superconnection $\widehat{B}$.

If $\alpha\in \Lambda(T^*(\R_+\times S))$, we can expand $\alpha$ in the form
\begin{align}\label{e01003}
\alpha=dt\wedge \alpha_0+\alpha_1,\quad \alpha_0, \alpha_1\in \Lambda (T^*S).
\end{align}
Set
\begin{align}\label{e01005}
[\alpha]^{dt}=\alpha_0.
\end{align}

For any $g\in G$, set
\begin{align}\label{e01103}
\psi_S\widetilde{\tr}[g\exp(-\widehat{B}^2)]=dt\wedge \gamma(t)+r(t).
\end{align}
Then from Duhamel's principle, (\ref{e01059}) and (\ref{e01102}), we have
\begin{align}\label{e01104}
\gamma(t)=\left\{\psi_S\widetilde{\tr}[g\exp(-\widehat{B}^2)]\right\}^{dt}=
\left\{
  \begin{aligned}%{ll}
    -\frac{1}{2\sqrt{-1}\sqrt{\pi}}\psi_S \tr_s\left[g\frac{\partial B_t}{\partial t}\exp(-B_t^2)\right],\quad &
    \hbox{if $n$ is even;} \\
    -\frac{1}{\sqrt{\pi}}\psi_S \tr^{\mathrm{even}}\left[g\frac{\partial B_t}{\partial t}\exp(-B_t^2)\right],\quad \quad
    & \hbox{if $n$ is odd}
  \end{aligned}
\right.
\end{align}
and
\begin{align}\label{e01115}
r(t)=\psi_S\widetilde{\tr}[g\exp(-B_t^2)].
\end{align}

%So
%\begin{align}\label{e01009}
%\gamma(t)=.
%\end{align}
For $u\in (0,+\infty)$, set $\widehat{B}_u=\sqrt{u}\delta_u\widehat{B}\delta_{u}^{-1}$. Similarly as in (\ref{e01103}),
we decompose
\begin{align}\label{e01116}
\psi_S\widetilde{\tr}[g\exp(-\widehat{B}_u^2)]=dt\wedge \gamma(u,t)+r(u,t).
\end{align}
Take $t=1$. Then
\begin{align}\label{e01117}
\left.\frac{\partial B_{ut}}{\partial t}\right|_{t=1}=u\frac{\partial B_{u}}{\partial u}.
\end{align}
So from (\ref{e01104}), (\ref{e01115}) and (\ref{e01117}), we have
\begin{align}\label{e01118}
\gamma(u,1)=u\gamma(u), \quad r(u,1)=r(u).
\end{align}

From the asymptotic expansion of the heat kernel, when $u\rightarrow 0$, there exist $a_i(t)\in \Lambda( T^*(\R_+\times S))$,
$i\in \N$, such that
\begin{align}\label{e01119}
\psi_S\widetilde{\tr}[g\exp(-\widehat{B}_u^2)]\sim \sum_{i=0}^{+\infty}a_i(t)u^{i/2}.
\end{align}
By Theorem \ref{e01061}, $r(0,t)$ exists and $a_0(t)=r(0,t)$.
Take $t=1$ in (\ref{e01119}). By Theorem \ref{e01061}, (\ref{e01115}) and (\ref{e01118}),
we have
\begin{align}\label{e01140}
r(0)=\int_{Z^g}\widehat{\mathrm{A}}_g(TZ,\nabla^{TZ})\wedge
\ch_g(L_Z^{1/2}, \nabla^{L_Z^{1/2}})\wedge
\ch_g(E, \nabla^{E}).
\end{align}
From (\ref{e01116}) and (\ref{e01118}), we have
\begin{align}\label{e01120}
dt\wedge u\gamma(u)+r(u)-r(0)\sim \sum_{i=1}^{+\infty}a_i(1)u^{i/2},
\end{align}
that is, when $u\rightarrow 0$,
\begin{align}\label{e01123}
\gamma(u)=O(u^{-1/2}).
\end{align}

Assume that $\dim \ker D^Z$ is locally constant, then $\ker D^Z$ forms a vector bundle over $S$.
Let $P^{\ker D^{Z}}:\cE_Z\rightarrow \ker D^Z$ be the orthogonal projection
with respect to the scalar product in (\ref{e01025}).
Let
\begin{align}\label{e01141}
\nabla^{\ker D^Z}=P^{\ker D^{Z}}\nabla^{\cE,u}P^{\ker D^{Z}}
\end{align}
be a connection on the vector bundle $\ker D^Z$.
For $b\in S$, $t\in (0, +\infty)$, $\ker (t^{1/2}D^Z_b)=\ker D^Z_b$. So $\ker D^{\widehat{Z}}$
forms a vector bundle over $\R_+\times S$.
As in (\ref{e01141}), we can define the connection $\nabla^{\ker D^{\widehat{Z}}}$
on the vector bundle $\ker D^{\widehat{Z}}$.
If $n$ is even, $\ker D^Z$ and $\ker D^{\widehat{Z}}$ are $\Z_2$-graded.
Since the curvature of $\nabla^{\widehat{\cE},u}$ is trivial along $\R_+$, the equivariant Chern character
$\ch_g(\ker D^{\widehat{Z}}, \nabla^{\ker D^{\widehat{Z}}})$ does not involve $dt$.

From \cite[Theorem 9.19]{MR2273508}, which is also valid in odd dimensional fiber case, we know that when $u\rightarrow +\infty$,
\begin{align}\label{e01121}
\psi_S\widetilde{\tr}[g\exp(-\widehat{B}_u^2)]=
\left\{
  \begin{aligned}%{ll}
    &\ch_g(\ker D^{\widehat{Z}}, \nabla^{\ker D^{\widehat{Z}}})
    +O(u^{-1/2}), & \hbox{if $n$ is even;} \\
    &O(u^{-1/2}),\quad\quad\quad\quad & \hbox{if $n$ is odd,}
  \end{aligned}
\right.
\end{align}
and
\begin{align}\label{e01139}
r(\infty):=\lim_{u\rightarrow \infty}r(u,1)=
\left\{
  \begin{aligned}%{ll}
    &\ch_g(\ker D^{Z}, \nabla^{\ker D^Z}), & \hbox{if $n$ is even;} \\
    &0,& \hbox{if $n$ is odd.}
  \end{aligned}
\right.
\end{align}
Take $t=1$ in (\ref{e01121}).
From (\ref{e01116}), (\ref{e01118}) and (\ref{e01139}) we have
\begin{align}\label{e01122}
dt\wedge u\gamma(u)+r(u)-r(\infty)=O(u^{-1/2}).
\end{align}
By (\ref{e01115}), (\ref{e01121}) and (\ref{e01122}), when $u\rightarrow +\infty$,
\begin{align}\label{e01124}
\gamma(u)=O(u^{-3/2}).
\end{align}

\begin{defn}\label{e01083}
Assume that $\dim\ker D^Z$ is locally constant on $S$.
For any $g\in G$, the equivariant eta form of Bismut-Cheeger
$\tilde{\eta}_g(T^HW, g^{TZ}, h^{L_Z}, h^E, \nabla^{L_Z},\nabla^E)\in\Omega^*(S)$ is defined  by
\begin{align}\label{e01084}
\tilde{\eta}_g(T^HW, g^{TZ}, h^{L_Z}, h^E, \nabla^{L}, \nabla^E):=-\int_0^\infty \gamma(t)dt.
\end{align}
Note that by (\ref{e01123}) and (\ref{e01124}), the integral on the right hand side of (\ref{e01084}) is convergent.
\end{defn}

When $g=1$, $TZ$ is Spin, this equivariant eta form is just the
usual eta form of Bismut-Cheeger defined in \cite{MR966608} and \cite{MR1088332}.
Note that the equivariant eta form here was also defined in \cite{Wang} when $TZ$ is Spin and $n$ is odd.

From \cite{Bismut1985}, we know that $\widetilde{\tr}[g\exp(-B^2)]$ is a closed differential form. So
\begin{align}\label{e01002}
\left(dt\wedge\frac{\partial}{\partial t}+d^{S}\right)\psi_S\widetilde{\tr}[g\exp(-\widehat{B}^2)]=0, \quad
d^S\psi_S\widetilde{\tr}[g\exp(-B_t^2)]=0.
\end{align}
By (\ref{e01115}), (\ref{e01103}) and (\ref{e01002}), we have
\begin{align}\label{e01125}
d^S\gamma(t)=\frac{\partial r(t)}{\partial t}.
\end{align}
Then from (\ref{e01115}), (\ref{e01140}), (\ref{e01125}) and Definition \ref{e01083}, we have
\begin{multline}\label{e01085}
d^S\tilde{\eta}_g(T^HW,g^{TZ}, h^{L_Z},h^E, \nabla^{L_Z}, \nabla^E)=-\int_0^{+\infty}\frac{\partial r(t)}{\partial t}dt=r(0)-r(\infty)
\\
=\left\{
  \begin{aligned}%{ll}
    &\int_{Z^g}\widehat{\mathrm{A}}_g(TZ,\nabla^{TZ})\wedge
\ch_g(L_Z^{1/2}, \nabla^{L_Z^{1/2}})\wedge
\ch_g(E, \nabla^{E})\\
  &\quad\quad\quad\quad\quad\quad\quad\quad\quad\quad-\ch_g(\ker D^Z, \nabla^{\ker D^Z}), & \hbox{if $n$ is even;} \\
    &\int_{Z^g}\widehat{\mathrm{A}}_g(TZ,\nabla^{TZ})\wedge
\ch_g(L_Z^{1/2}, \nabla^{L_Z^{1/2}})\wedge
\ch_g(E, \nabla^{E}),& \hbox{if $n$ is odd.}
  \end{aligned}
\right.
\end{multline}

\begin{rem}\label{e01067}
If we fix the vertical metric $\widehat{g}^{TZ}$ which restricts to $t^{-2}g^{TZ}$ over $\{t\}\times W$ in the beginning of this subsection,
as in (\ref{e01102}), we have
\begin{align}\label{e01068}
\widehat{B}'|_{(t,b)}=B_{t^2}+dt\wedge\frac{\partial}{\partial t}-\frac{n}{2t}dt,
\end{align}
and
\begin{align}\label{e01069}
\begin{split}
\gamma'(t)&=\left\{\psi_S\widetilde{\tr}[g\exp(-\widehat{B}'^2)]\right\}^{dt}
\\
&=
\left\{
  \begin{aligned}%{ll}
    &-\frac{1}{2\sqrt{-1}\sqrt{\pi}}\psi_S \tr_s\left[g\frac{\partial B_{t^2}}{\partial t}\exp(-B_{t^2}^2)\right], & \hbox{$n$ is even;} \\
    &-\frac{1}{\sqrt{\pi}}\psi_S \tr^{\mathrm{even}}\left[g\frac{\partial B_{t^2}}{\partial t}
    \exp(-B_{t^2}^2)\right], & \hbox{$n$ is odd.}
  \end{aligned}
\right.
\end{split}
\end{align}
After changing the variable, we still have
\begin{align}\label{e01070}
\tilde{\eta}_g(T^HW, g^{TZ}, h^{L}, h^E, \nabla^{L}, \nabla^E):=-\int_0^\infty \gamma'(t)dt.
\end{align}
\end{rem}

\begin{rem}\label{e01202}
The $\mathrm{Spin}^c$ condition used here is just to get an explicit local index representative in Theorem \ref{e01061}.
In fact, Definition \ref{e01083} can be extended to equivariant Clifford module case.
\end{rem}

\subsection{Anomaly formula}\label{s0105}

From the construction in Section \ref{s0104}, the equivariant eta form only depends on the sextuple
($T^HW, g^{TZ}, h^{L_Z}, h^E, \nabla^{L_Z}, \nabla^E$).
We now describe how $\tilde{\eta}_g(T^HW,g^{TZ},$ $h^{L_Z}, h^E, \nabla^{L_Z}, \nabla^E)$ depends on its arguments. Let
($T^HW, g^{TZ}, h^{L_Z}, h^E,\nabla^{L_Z}$, $\nabla^E$) and ($T^{'H}W, g^{'TZ}, h^{'L_Z}, h^{'E}, \nabla^{'L_Z}, \nabla^{'E}$)
be two sextuples of geometric data.
We will mark the objects associated to the second sextuple by $'$.

First, a horizontal subbundle on $W$ is simply a splitting of the exact sequence
\begin{align}\label{e01145}
0\rightarrow TZ\rightarrow TW \rightarrow \pi^*TS \rightarrow 0.
\end{align}
As the space of the splitting map is affine and $G$ is compact, it follows that any pair of equivariant horizontal subbundles
can be connected by a smooth path of equivariant horizontal distributions. Let $s\in [0,1]$ parametrize a smooth path
$\{T_s^HW\}_{s\in [0,1]}$ such that $T_0^HW=T^HW$ and $T_1^HW=T^{'H}W$. Similarly, let $g_s^{TZ}$, $h_s^{L_Z}$ and $h_s^E$
be the $G$-invariant metrics on $TZ$, $L_Z$ and $E$, depending smoothly on $s\in [0,1]$, which coincide with $g^{TZ}$,
$h^{L_Z}$ and $h^E$ at $s=0$ and with $g^{'TZ}$,
$h^{'L_Z}$ and $h^{'E}$ at $s=1$.
Let $\nabla$ and $\nabla'$ be equivariant Euclidean connections on $(TZ, g^{TZ})$ and $(TZ, g^{'TZ})$.
By the same reason, we can choose $G$-invariant connections $\nabla_s$, $\nabla_s^{L_Z}$ and $\nabla_s^E$
on $TZ$, $L_Z$ and $E$ preserving $g_s^{TZ}$, $h_s^{L_Z}$ and $h_s^E$
such that $\nabla_0=\nabla$, $\nabla_1=\nabla^{'}$, $\nabla_0^{L_Z}=\nabla^{L_Z}$, $\nabla_1^{L_Z}=\nabla^{'L_Z}$,
$\nabla_0^{E}=\nabla^{E}$, $\nabla_1^{E}=\nabla^{'E}$.

%Let $\widetilde{S}=[0, 1]\times S$ and $\mathrm{pr'}:\widetilde{S}\rightarrow S$ be the projection.
%We consider the bundle $\widetilde{\pi}:\widetilde{W}:=[0,1]\times W\rightarrow\widetilde{S}$ together with the canonical projection
%$\mathrm{Pr'}:\widetilde{W}\rightarrow W$.
%Then $T^H\widetilde{W}_{(s, \cdot)}=\R\times T_s^HW$ defines a horizontal subbundle of $T\widetilde{W}$,
%and $T\widetilde{Z}:=\mathrm{Pr'}^*TZ$, $\widetilde{L}_Z:=\mathrm{Pr'}^*L_Z$ and $\widetilde{E}:=\mathrm{Pr'}^*E$ are naturally equipped with
%metrics $g^{T\widetilde{Z}}$, $h^{\widetilde{L}_Z}$, $h^{\widetilde{E}}$ and connections
%$\widetilde{\nabla}$, $\nabla^{\widetilde{L}_Z}$, $\nabla^{\widetilde{E}}$.
%Then the fiberwise $G$ action can be naturally extended to $\widetilde{\pi}:\widetilde{W}\rightarrow\widetilde{S}$
%such that $G$ acts as identity on $\widetilde{S}$ and $g^{T\widetilde{Z}}, h^{\widetilde{L}_Z}$, $h^{\widetilde{E}}$,
%$\widetilde{\nabla}$,  $\nabla^{\widetilde{L}_Z}$,
%$\nabla^{\widetilde{E}}$
%are $G$-invariant.
%Let $D^{\widetilde{Z}}$ be the fiberwise Dirac operator associated to
%($T^H\widetilde{W}, g^{T\widetilde{Z}}, \nabla^{\widetilde{L}_Z}, \nabla^{\widetilde{E}}$).

Let $\widetilde{S}=[0, 1]\times S$, $\widetilde{W}:=[0,1]\times W$. From the construction above, we can get
a family of equivariant geometric data $(T^H\widetilde{W}, g^{T\widetilde{Z}}, \widetilde{\nabla},
h^{\widetilde{E}}, \nabla^{\widetilde{E}}, h^{\widetilde{L}_Z}, \nabla^{\widetilde{L}_Z})$
with respect to $\widetilde{\pi}:\widetilde{W}\rightarrow\widetilde{S}$.
Let $D^{\widetilde{Z}}$ be the fiberwise Dirac operator associated to
($T^H\widetilde{W}, g^{T\widetilde{Z}}, \nabla^{\widetilde{L}_Z}, \nabla^{\widetilde{E}}$).

\begin{assump}\label{e01146}
We assume that there exists such a smooth path such that $\ker D^{\widetilde{Z}}$ is locally constant.
\end{assump}

Under Assumption \ref{e01146}, from (\ref{e01141}), we can define the connection
$\nabla^{\ker D^{\widetilde{Z}}}$
on $\ker D^{\widetilde{Z}}$.
From \cite[Theorem B.5.4]{MR2339952}, modulo exact forms, the Chern-Simons forms
\begin{align}\label{e01127}
\begin{split}
&\widetilde{\widehat{\mathrm{A}}}_g(TZ,\nabla,\nabla^{'}):=-\int_0^1[\widehat{\mathrm{A}}_g(TZ,\widetilde{\nabla})]^{ds}ds,
\\
&\widetilde{\ch}_g(L_Z^{1/2}, \nabla^{L_Z^{1/2}}, \nabla^{'L_Z^{1/2}}):=-\int_0^1[\ch_g(\widetilde{L}_Z^{1/2},
\nabla^{\widetilde{L}_Z^{1/2}})]^{ds}ds,
\\
&\widetilde{\ch}_g(E, \nabla^{E}, \nabla^{'E}):=-\int_0^1[\ch_g(\widetilde{E},
\nabla^{\widetilde{E}})]^{ds}ds,
\\
&\widetilde{\ch}_g(\ker D^Z, \nabla^{\ker D^Z}, \nabla^{'\ker D^Z})
:=-\int_0^1[\ch_g(\ker D^{\widetilde{Z}}, \nabla^{\ker
D^{\widetilde{Z}}})]^{ds}ds
\end{split}
\end{align}
do not depend on the
choices of the objects with $\ \widetilde{}\ $. Moreover,
\begin{align}\label{e01110}
\begin{split}
&d\widetilde{\widehat{\mathrm{A}}}_g(TZ,\nabla, \nabla^{'})=\widehat{\mathrm{A}}_g(TZ,\nabla^{'})
-\widehat{\mathrm{A}}_g(TZ,\nabla),
\\
&d\widetilde{\ch}_g(L_Z^{1/2}, \nabla^{L_Z^{1/2}}, \nabla^{'L_Z^{1/2}})=
\ch_g(L_Z^{1/2},
\nabla^{'L_Z^{1/2}})-\ch_g(L_Z^{1/2},
\nabla^{L_Z^{1/2}}),
\\
&d\widetilde{\ch}_g(E, \nabla^{E}, \nabla^{'E})=
\ch_g(E,
\nabla^{'E})-\ch_g(E,
\nabla^{E}),
\\
&d\widetilde{\ch}_g(\ker D^Z, \nabla^{\ker D^Z}, \nabla^{'\ker D^Z})
=\ch_g(\ker D^Z,\nabla^{'\ker D^Z})-\ch_g(\ker D^Z,\nabla^{\ker D^Z}).
\end{split}
\end{align}

Now we can obtain the anomaly formula for the equivariant eta forms.
\begin{thm}\label{e01105}
Assume that Assumption \ref{e01146} holds.

i) When $n$ is odd, modulo exact forms on $S$, we have
\begin{multline}\label{e01106}
\tilde{\eta}_g(T^{'H}W,g^{'TZ},h^{'L_Z}, h^{'E},\nabla^{'L_Z}, \nabla^{'E})-\tilde{\eta}_g(T^HW,g^{TZ},h^{L_Z}, h^E,\nabla^{L_Z}, \nabla^E)
\\
=\int_{Z^g}\widetilde{\widehat{\mathrm{A}}}_g(TZ, \nabla^{TZ}, \nabla^{'TZ})\wedge \ch_g(L_Z^{1/2}, \nabla^{L_Z^{1/2}})
\wedge \ch_g(E, \nabla^{E})
\\
+\int_{Z^g}\widehat{\mathrm{A}}_g(TZ, \nabla^{'TZ})\wedge \widetilde{\ch}_g(L_Z^{1/2}, \nabla^{L_Z^{1/2}},\nabla^{'L_Z^{1/2}})\wedge \ch_g(E, \nabla^{E})
\\
+\int_{Z^g}\widehat{\mathrm{A}}_g(TZ, \nabla^{'TZ})\wedge \ch_g(L_Z^{1/2}, \nabla^{'L_Z^{1/2}})
\wedge \widetilde{\ch}_g(E, \nabla^{E},\nabla^{'E}).
\end{multline}

ii) When $n$ is even,
modulo exact forms on $S$, we have
\begin{multline}\label{e01142}
\tilde{\eta}_g(T^{'H}W,g^{'TZ},h^{'L_Z}, h^{'E},\nabla^{'L_Z}, \nabla^{'E})-\tilde{\eta}_g(T^HW,g^{TZ},h^{L_Z}, h^E,\nabla^{L_Z}, \nabla^E)
\\
=\int_{Z^g}\widetilde{\widehat{\mathrm{A}}}_g(TZ, \nabla^{TZ}, \nabla^{'TZ})\wedge \ch_g(L_Z^{1/2}, \nabla^{L_Z^{1/2}})
\wedge \ch_g(E, \nabla^{E})
\\
+\int_{Z^g}\widehat{\mathrm{A}}_g(TZ, \nabla^{'TZ})\wedge \widetilde{\ch}_g(L_Z^{1/2},
\nabla^{L_Z^{1/2}},\nabla^{'L_Z^{1/2}})\wedge \ch_g(E, \nabla^{E})
\\
+\int_{Z^g}\widehat{\mathrm{A}}_g(TZ, \nabla^{'TZ})\wedge \ch_g(L_Z^{1/2}, \nabla^{'L_Z^{1/2}})
\wedge \widetilde{\ch}_g(E, \nabla^{E},\nabla^{'E})
\\
-\widetilde{\ch}_g(\ker D^Z, \nabla^{\ker D^Z},\nabla^{'\ker D^Z}).
\end{multline}
\end{thm}
\begin{proof}
Let $\widetilde{B}$ be the Bismut superconnection associated to ($T^H\widetilde{W},
\widetilde{g}^{TZ}, h^{\widetilde{L_Z}}, \nabla^{\widetilde{L_Z}},
\nabla^{\widetilde{E}}$).
From (\ref{e01102}),
\begin{align}\label{e01111}
\widehat{\widetilde{B}}=\widetilde{B}_t+dt\wedge \frac{\partial}{\partial t}-\frac{n}{4t}dt
\end{align}
is the Bismut superconnection associated to the fibration $(0,+\infty)\times[0,1]\times W\rightarrow (0,+\infty)\times[0,1]\times S$.
We decompose
\begin{align}\label{e01112}
\psi_S\widetilde{\tr}[g\exp(-\widehat{\widetilde{B}}\,^2)]=dt\wedge\gamma+ds\wedge r_1+dt\wedge ds\wedge r_2+r_3,
\end{align}
where $\gamma, r_1, r_2, r_3$ do not contain $dt$ neither $ds$ and by (\ref{e01115}),
\begin{align}\label{e01128}
r_1(t,s)=\left.\left\{\psi_S\widetilde{\tr}[g\exp(-\widetilde{B}_t^2)]\right\}^{ds}\right|_{(t,s)}.
\end{align}
From (\ref{e01112}) and Definition \ref{e01083}, we have
\begin{align}\label{e01203}
\tilde{\eta}_g(T_s^HW, g_s^{TZ}, h_s^{L_Z}, h_s^E, \nabla_s^{L}, \nabla_s^E):=-\int_0^\infty \gamma(t,s)dt.
\end{align}

Since $(dt\wedge \frac{\partial}{\partial t}+ds\wedge \frac{\partial}{\partial s}+d^{S})
\psi_S\widetilde{\tr}[g\exp(-\widehat{\widetilde{B}}\,^2)]=0$,
we have
 \begin{align}\label{e01113}
\frac{\partial \gamma}{\partial s}=\frac{\partial r_1}{\partial t}+d^S r_2.
\end{align}
From (\ref{e01203}), we have
\begin{align}\label{e01114}
\begin{split}
&\tilde{\eta}_g(T^{'H}W,g^{'TZ},h^{'L_Z}, h^{'E},\nabla^{'L_Z}, \nabla^{'E})-\tilde{\eta}_g(T^HW,g^{TZ},h^{L_Z}, h^E,\nabla^{L_Z}, \nabla^E)
\\
=&\int_{0}^{+\infty}(\gamma(t,1)-\gamma(t,0))dt=\int_0^{+\infty}\int_0^1 \frac{\partial}{\partial s}\gamma(t,s)ds dt=\int_0^1\int_0^{+\infty} \frac{\partial}{\partial s}\gamma(t,s)dt ds
\\
=&\int_0^1\int_0^{+\infty} \frac{\partial}{\partial t}r_1(t,s)dt ds+d^S\int_0^1\int_0^{+\infty} r_2(t,s)dt ds
\\
=&-\int_0^1 (r_1(0,s)-r_1(\infty,s))ds+d^S\int_0^1\int_0^{+\infty} r_2(t,s)dt ds.
\end{split}
\end{align}
The commutative property of the integrals
 in the above formula is granted by the uniformness of (\ref{e01123}) and (\ref{e01124}) for $s\in [0,1]$.

Let $\nabla^{T\widetilde{Z}}$ be the Euclidean connection associated to $(T^H\widetilde{W}, g^{T\widetilde{Z}})$ as in (\ref{e01014}).
By (\ref{e01140}), (\ref{e01139}) and (\ref{e01128}), we have
\begin{align}\label{e01143}
r_1(0,s)=\left.\left\{\int_{Z^g}\widehat{A}_g(TZ,\nabla^{T\widetilde{Z}})\wedge\ch_g(\widetilde{L}_Z^{1/2},
\nabla^{\widetilde{L}_Z^{1/2}})\wedge\ch_g(\widetilde{E},
\nabla^{\widetilde{E}})\right\}^{ds}\right|_{\{s\}\times S}
\end{align}
and
\begin{align}\label{e01144}
r_1(\infty,s)=
\left\{
  \begin{aligned}%{ll}
    &\{\ch_g(\ker D^{\widetilde{Z}}, \nabla^{\ker D^{\widetilde{Z}}})\}^{ds}|_{\{s\}\times S},
    & \hbox{if $n$ is even;} \\
    &0,& \hbox{if $n$ is odd.}
  \end{aligned}
\right.
\end{align}
Then Theorem \ref{e01105} follows from  (\ref{e01127}), (\ref{e01114}), (\ref{e01143}) and (\ref{e01144}).

The proof of Theorem \ref{e01105} is complete.
\end{proof}

\section{Functoriality of equivariant eta forms}\label{s02}

In this section, we state our main result. %This section is organized as follows.

\subsection{Functoriality of equivariant eta forms}

Let $W$, $V$, $S$ be smooth manifolds. Let $\pi_1:W\rightarrow V$, $\pi_2:V\rightarrow S$ be smooth fibrations with closed
oriented fibers $X$, $Y$, with $\dim X=n-m$, $\dim Y=m$. Then $\pi_3=\pi_2\circ \pi_1: W\rightarrow S$ is a smooth
fibration with closed oriented fiber $Z$ with $\dim Z=n$.
Then we have the diagram of smooth fibrations:

\begin{center}\label{e02001}
\begin{tikzpicture}[>=angle 90]
\matrix(a)[matrix of math nodes,
row sep=2em, column sep=2.5em,
text height=1.5ex, text depth=0.25ex]
{X&Z&W\\
&Y&V&S.\\};
\path[->](a-1-1) edge (a-1-2);
\path[->](a-1-2) edge node[left]{\footnotesize{}} (a-2-2);
\path[->](a-1-2) edge (a-1-3);
\path[->](a-2-2) edge (a-2-3);
\path[->](a-1-3) edge node[left]{\footnotesize{$\pi_1$}} (a-2-3);
\path[->](a-2-3) edge node[above]{\footnotesize{$\pi_2$}} (a-2-4);
\path[->](a-1-3) edge node[above]{\footnotesize{$\pi_3$}} (a-2-4);
\end{tikzpicture}
\end{center}

Let $TX$, $TY$, $TZ$ be the relative tangent bundles.
We assume that $TX$ and $TY$ have the $\mathrm{Spin}^c$ structures with complex line bundles $L_X$ and $L_Y$ respectively.
Let
\begin{align}\label{e02004}
L_Z=\pi_1^*(L_Y)\otimes L_X.
\end{align}
Then $TZ$ have a $\mathrm{Spin}^c$ structure with complex line bundle $L_Z$.
Recall the notations in Section 1, we take quadruples ($T_1^HW, g^{TX}, h^{L_X}, \nabla^{L_X}$),
($T_2^HV, g^{TY}, h^{L_Y}, \nabla^{L_Y}$) and
($T_3^HW$, $g^{TZ}$, $h^{L_Z}$, $\nabla^{L_Z}$) with respect to
fibrations  $\pi_1$, $\pi_2$ and $\pi_3$ respectively.
Then we can define connections $\nabla^{TX}$, $\nabla^{TY}$, $\nabla^{TZ}$,
fundamental spinors  $\mS(TX, L_X)$, $\mS(TY, L_Y)$, $\mS(TZ, L_Z)$,
metrics $h^{\mS_X}$, $h^{\mS_Y}$, $h^{\mS_Z}$ and connections
$\nabla^{\mS_X}$, $\nabla^{\mS_Y}$, $\nabla^{\mS_Z}$ as in Section \ref{s0102}.
If $U\in TS$, $U'\in TV$, let $U_1^{'H}\in T_1^HW$, $U_2^{H}\in T_2^HV$, $U_3^{H}\in T_3^HW$
be the horizontal lifts of $U'$, $U$, $U$, so that $\pi_{1,*}(U_1^{'H})=U'$, $\pi_{2,*}(U_2^{H})=U$, $\pi_{3,*}(U_3^{H})=U$.

Set $T^HZ:=T_1^HW\cap TZ$. Then we have the splitting of smooth vector bundles over $W$,
\begin{align}\label{e02003}
TZ=T^HZ\oplus TX,
\end{align}
and
\begin{align}\label{e02002}
T^HZ\cong \pi_1^*TY.
\end{align}
Let $\,^0\nabla^{TZ}$ be the connection on $TZ=T^HZ\oplus TX$ defined by
\begin{align}\label{e02114}
\,^0\nabla^{TZ}=\pi^*\nabla^{TY}\oplus\nabla^{TX}
\end{align}
as in (\ref{e01015}). Set
\begin{align}\label{e02005}
\,^0\nabla^{L_Z}=\pi_1^*\nabla^{L_Y}\otimes 1+1 \otimes \nabla^{L_X}.
\end{align}

Let $(E, \nabla^E)$ be a Hermitian vector bundle with Hermitian connection $\nabla^{E}$.
For $v\in V$, let $\cE_{X,v}$ be the set of
smooth sections over $X_v$ of $\mS(TX, L_X)\otimes E$. We still regard $\cE_X$ as
 an infinite dimensional fiber bundle over $V$.
For any $v\in V$, $s_1, s_2\in \cE_{X,v}$,
as in (\ref{e01025}), we define the scalar product
\begin{align}\label{e02006}
\la s_1, s_2 \ra_{\cE_{X,v}}=\int_{X_v}\la s_1(x), s_2(x) \ra_X\, dv_X,
\end{align}
where $\la\cdot,\cdot\ra_X=h^{\mS_X\otimes E}(\cdot,\cdot)$.
Let $\{e_i\}$ be a local orthonormal frame of $(TX, g^{TX})$.
As in (\ref{e01026}) and (\ref{e01028}), for $U\in TV$, we
set
\begin{align}\label{e02109}
\nabla^{\cE_X,u}_{U}:=\nabla^{\mS_X\otimes E}_{U_1^H}-\frac{1}{2}\la S_1(e_i)e_i,U_1^H\ra.
\end{align}
Then $\nabla^{\cE_X,u}$ preserves the scalar product $\la\cdot,\cdot\ra_{\cE_X}$.

Let $D^X$ and $D^Z$ be the fiberwise Dirac operators associated to ($T_1^HW$, $ g^{TX}$, $\nabla^{L_X}$, $h^E$, $\nabla^E$) and ($T_3^HW$, $ g^{TZ}$, $\nabla^{L_Z}$, $h^E$, $\nabla^E$).
We assume that $\ker D^X$ is locally constant.
Then $\ker D^X$ forms a vector bundle over $V$.
Let $P^{\ker D^X}: \cE_{X}\rightarrow \ker D^X$ be the
orthonomal projection with respect to the scalar product (\ref{e02006}).
Let $h^{\ker D^X}$ be the $L^2$ metric induced by $h^{\mS_X\otimes E}$ and
\begin{align}\label{e02030}
\nabla^{\ker D^X}:=P^{\ker D^X}\nabla^{\cE_X,u}P^{\ker D^X}.
\end{align}
Then $\nabla^{\ker D^X}$ preserves the metric $h^{\ker D^X}$.
Let $D^Y$ be the Dirac operator associated to ($T_2^HV, g^{TY}, \nabla^{\mS_Y\otimes \ker D^X}$).

\begin{assump}\label{e02116}
We assume that the geometric data ($T_1^HW, g^{TX}, h^{L_X}, \nabla^{L_X}$, $h^E$, $\nabla^E$) and
($T_2^HV$, $g^{TY}$, $h^{L_Y}$, $\nabla^{L_Y}$) satisfy the conditions that $\ker D^X$ is locally constant and $\ker D^Y=0$.
\end{assump}

Let $G$ be a compact Lie group which acts on $W$ such that for any $g\in G$,
$g\cdot\pi_1=\pi_1\cdot g$ and $\pi_3\cdot g=\pi_3$. Then we know that
$G$ acts as identity on $S$. We assume that the action of $G$ preserves the $\mathrm{Spin}^c$ structures of $TX$, $TY$, $TZ$
and the quadruples ($T_1^HW, g^{TX}, h^{L_X}, \nabla^{L_X}$),
($T_2^HV, g^{TY}, h^{L_Y}, \nabla^{L_Y}$),
($T_3^HW$, $g^{TZ}$, $h^{L_Z}$, $\nabla^{L_Z}$) and ($E, h^E, \nabla^E$) are $G$-invariant.

On the other hand,
we take another equivariant horizontal subbundle $T_3^{'H}W\subset TW$, which is complement of $TZ$, such that
\begin{align}\label{e02110}
T_3^{'H}W \subset T_1^HW.
\end{align}
Let $g^{'TZ}$ be another metric on $TZ$ such that
\begin{align}\label{e02111}
g^{'TZ}=\pi_1^*g^{TY}\oplus g^{TX}.
\end{align}
Let $\nabla^{'TZ}$ be the connection associated to  $(T_3^{'H}W, g^{'TZ})$ as in (\ref{e01014}).

Let $\mS'(TZ, L_Z)$ be the fundamental spinor associated to $(g^{'TZ}, L_Z)$.
Then
\begin{align}\label{e02112}
\mS'(TZ, L_Z)\simeq\pi_1^*\mS(TY, L_Y)\otimes \mS(TX, L_X).
\end{align}
Set
\begin{align}\label{e02113}
h^{'L_Z}:=\pi_1^*h^{L_Y}\otimes h^{L_X}.
\end{align}

Let
\begin{align}\label{e02117}
g_T^{'TZ}=\pi_1^*g^{TY}\oplus T^{-2}g^{TX}.
\end{align}
We denote the Clifford algebra bundle of $TZ$ with respect to $g_T^{'TZ}$
 by $C_T(TZ)$. Let $\{f_{p}\}$ be a local orthonormal frame of $(TY, g^{TY})$.
Then $\{Te_i\}\cup\{f_{p,1}^H\}$ is a local orthonormal frame of $(TZ, g_T^{'TZ})$.
We define a Clifford algebra isomorphism
\begin{align}\label{e02027}
\mG_T:C_T(TZ)\rightarrow C(TZ)
\end{align}
by
\begin{align}\label{e02028}
\begin{split}
\mG_T(c(f_{p,1}^H))=c(f_{p,1}^H),\quad \mG_T(c_T(Te_i))=c(e_i).
\end{split}
\end{align}
Under this isomorphism, we can also consider $\mS'(TZ, L_Z)$ in (\ref{e02112}) as a spinor associated to ($TZ, g_T^{'TZ}$).
Let $D_T^Z$ be the fiberwise Dirac operator associated to ($T_3^{'H}W, g_T^{'TZ}, \,^0\nabla^{L_Z}$, $h^E$, $\nabla^E$).

Comparing with \cite[Theorem 1.5]{MR1088332}, we can get the following lemma.
\begin{lemma}\label{e02118}
If Assumption \ref{e02116} holds, there exists $T_0\geq 1$, such that when $T\geq T_0$, $\ker D_T^Z=0$.
\end{lemma}
We will give another proof of this lemma in Section \ref{s0403}.

Now we state an analogue of Assumption \ref{e01146} as follows.
\begin{assump}\label{e02119}
We assume that there exist
an equivariant horizontal subbundle $T_3^{'H}W\subset TW$ satisfying (\ref{e02110}) and
 a smooth path constructed as the argument before Assumption \ref{e01146},
 connecting the quadruples ($T_3^HW$, $g^{TZ}$, $h^{L_Z}$, $\nabla^{L_Z}$)
and ($T_3^{'H}W$, $g_{T_0}^{'TZ}$, $h^{'L_Z}$, $\,^0\nabla^{L_Z}$),
such that $\ker (D^{\widetilde{Z}})=0$.
\end{assump}

For any $g\in G$, let $T_1^H(W|_{V^g})=T_1^HW|_{V^g}\cap T(W|_{V^g})$ be the equivariant horizontal subbundle of $T(W|_{V^g})$.
We state our main result as follows.

\begin{thm}\label{e02084}
If Assumption \ref{e02116} and \ref{e02119} hold, for any $g\in G$,
we have the following identity in $\Omega^*(S)/d^S\Omega^*(S)$,
\begin{multline}\label{e02089}
\tilde{\eta}_{g}(T_3^HW,g^{TZ},h^{L_Z},\nabla^{L_Z}, h^E, \nabla^E)=\tilde{\eta}_g(T_2^HV,g^{TY},h^{L_Y},
h^{\ker D^X},\nabla^{L_Y},\nabla^{\ker D^X})
\\
+\int_{Y^g}{\rm \widehat{A}}_g(TY,\nabla^{TY})\wedge\ch_g(L_Y^{1/2}, \nabla^{L_Y^{1/2}})\wedge
\tilde{\eta}_g(T_1^H(W|_{V^g}),g^{TX},h^{L_X}, \nabla^{L_X}, h^E, \nabla^E)
\\
-\int_{Z^g}{\rm \widetilde{{\widehat{A}}}}_g(TZ,\nabla^{TZ},\,^0\nabla^{TZ})\wedge\ch_g(L_Z^{1/2}, \nabla^{L_Z^{1/2}})\wedge \ch_g(E,\nabla^E)
\\
-\int_{Z^g}{\rm {\widehat{A}}}_g(TZ,\,^0\nabla^{TZ})\wedge\widetilde{\ch}_g(L_Z^{1/2}, \nabla^{L_Z^{1/2}}, \,^0\nabla^{L_Z^{1/2}})\wedge \ch_g(E,\nabla^E).
\end{multline}
\end{thm}

\subsection{Simplifying assumptions}\label{s0202}

By anomaly formula Theorem \ref{e01105}, we only need to prove Theorem \ref{e02084}
when ($T_3^HW$, $g^{TZ}$, $h^{L_Z}$, $\nabla^{L_Z}$)=$(T_3^{H'}W, g_{T_0}^{'TZ}, h^{'L_Z}, \nabla^{'L_Z})$.
Therefore, in the following sections, we assume that
\begin{align}
\begin{split}
T_3^HW\subset T_1^HW,&\quad g^{TZ}=g^{TX}\oplus \pi_1^*g^{TY},\quad h^{L_Z}=\pi_1^*h^{L_Y}\otimes h^{L_X},
\\
\nabla^{L_Z}&=\pi_1^*\nabla^{L_Y}\otimes 1+1\otimes \nabla^{L_X}.
\end{split}
\end{align}
Let
\begin{align}\label{e02021}
g_T^{TZ}=\pi_1^*g^{TY}\oplus\frac{1}{T^2}g^{TX}
\end{align}
and $D_T^Z$ be the fiberwise Dirac operator associated to ($T_3^{H}W, g_T^{TZ}, \nabla^{L_Z}$, $h^E$, $\nabla^E$).
We assume that $\ker D^X$ is locally constant, $\ker D^Y=0$ and for any $T\geq 1$, $\ker D_T^Z=0$.

\section{Proof of Theorem \ref{e02084}}\label{s03}

In this section, we use the assumptions and the notations in Section \ref{s0202}.

This Section is organized as follows. In Section 3.1, we introduce a $1$-form on $\R_+\times\R_+$.
In Section 3.2, we state some intermediate results which we need for the proof of Theorem \ref{e02084},
whose proofs are delayed to Section 4-8. In Section 3.3, we prove Theorem \ref{e02084}.
For the convenience to compare the results in this paper with those in \cite{MR2072502},
the intermediate results and the proof of Theorem \ref{e02084} in this section are formulated almost
the same as in \cite[Theorem 5.11]{MR2072502}. We left the main difficulties in the proofs of intermediate
results later.

\subsection{A fundamental $1$-form}\label{s0301}

Let  $\nabla_T^{TZ}$ be the connection associated to $(T_3^HW, g_T^{TZ})$ as in (\ref{e01014}).
Let $S_{1,T}$ be the tensor associated to ($T_1^HW, T^{-2}g^{TX}$) as in (\ref{e01017}).
Comparing with \cite[(3.10)]{Bismut1985} and \cite[Theorem 5.1]{MR1942300}, we have
\begin{align}\label{e04310}
\nabla^{TZ}_T=\,^0\nabla^{TZ}+P^{TZ}S_{1,T}P^{TZ}=\,^0\nabla^{TZ}+P^{TX}S_1P^{T^HZ}+\frac{1}{T^2}P^{T^HZ}S_1P^{TZ}.
\end{align}

Let $\nabla^{\mS_Z,T}$ be the connection on $\mS(TZ, L_Z)$ induced by $\nabla^{TZ}_T$ and $\nabla^{L_Z}$.
Set
\begin{align}\label{e03012}
\,^0\nabla^{\mS_Z}:=\pi_1^*\nabla^{\mS_Y}\otimes 1+ 1\otimes \nabla^{\mS_X}.
\end{align}
Then by (\ref{e04310}),
\begin{align}\label{e02029}
\begin{split}
\nabla^{\mS_Z,T}=\, ^0\nabla^{\mS_Z}+\frac{1}{2T}\la S_{1}(\cdot) e_i,f_{p,1}^H\ra c(e_i)c(f_{p,1}^H)
+\frac{1}{4T^2}\la S_{1}(\cdot)f_{p,1}^H,f_{q,1}^H\ra c(f_{p,1}^H)c(f_{q,1}^H).
\end{split}
\end{align}

As the construction in Section \ref{s0104}, We consider the space $\widehat{S}:=\R_{+,T}\times\R_{+,u}\times S$.
Let $\mathrm{pr}_S:\widehat{S}\rightarrow S$ denote the projection and consider the fibration $\widehat{\pi}_3:
\widehat{W}:=\R_{+,T}\times\R_{+,u}\times W\rightarrow \widehat{S}$.
Let $\mathrm{Pr}_W:\widehat{W}\rightarrow W$ be the canonical projection.
Set $T^H\widehat{W}=T(\R_{+}\times\R_{+})\oplus\mathrm{Pr}_W^*(T_1^HW)$.
Then $T^H\widehat{W}$ is a horizontal subbundle of $T\widehat{W}$ as in (\ref{e01010}).
We define the metric $\widehat{g}^{TZ}$ such that it restricts to $u^{-2}g_T^{TZ}$ over $(T,u)\times W$.
Let $h^{\widehat{L}_Z}=\mathrm{Pr}_W^*h^{L_Z}$,  $\nabla^{\widehat{L}_Z}=\mathrm{Pr}_W^*\nabla^{L_Z}$, $h^{\widehat{E}}=\mathrm{Pr}_W^*h^{E}$ and $\nabla^{\widehat{E}}=\mathrm{Pr}_W^*\nabla^{E}$.
We naturally extend the $G$-actions to this case such that the $G$-action is identity on $\widehat{S}$.

We denote by $B_{3,u^2,T}$ the Bismut superconnection associated to $(T_3^HW$, $u^{-2}g_T^{TZ}$, $h^{L_Z}$, $\nabla^{L_Z}$, $h^E$, $\nabla^E$).
We know that the $G$-action commutes with this Bismut superconnection.

Let $\widehat{B}$ be the Bismut superconnection for the fibration $\widehat{W}\rightarrow \widehat{S}$,
by the arguments above (\ref{e01068}), we can get
\begin{align}\label{e03039}
\widehat{B}_{(T,u,b)}=B_{3,u^2,T}+dT\wedge\frac{\partial}{\partial T}+du\wedge\frac{\partial}{\partial u}-\frac{n}{2u}du-\frac{n-m}{2T}dT.
\end{align}

\begin{defn}\label{e03001}
We define $\beta_g=du\wedge \beta_g^u+dT\wedge \beta_g^T$ to be the part of $\psi_S\widetilde{\tr}[g\exp(-\widehat{B}^2)]$
of degree one with respect to the coordinates $(T,u)$, with functions $\beta_g^u$, $\beta_g^T: \mathbb{R}_{+,T}\times\mathbb{R}_{+,u}
\rightarrow \Omega^*(S)$.
\end{defn}
From (\ref{e01104}) and (\ref{e03039}), we have
\begin{align}\label{e03040}
\begin{split}
&\beta_g^u(T,u)=
\left\{
  \begin{aligned}%{ll}
    &-\frac{1}{2\sqrt{-1}\sqrt{\pi}}\psi_S \tr_s\left[g\frac{\partial B_{3,u^2,T}}{\partial u}\exp(-B_{3,u^2,T}^2)\right],
    & \hbox{if $n$ is even;} \\
    &-\frac{1}{\sqrt{\pi}}\psi_S \tr^{\mathrm{even}}\left[g\frac{\partial B_{3,u^2,T}}{\partial u}\exp(-B_{3,u^2,T}^2)\right],
    & \hbox{if $n$ is odd,}
  \end{aligned}
\right.
\\
&\beta_g^T(T,u)=
\left\{
  \begin{aligned}%{ll}
    &-\frac{1}{2\sqrt{-1}\sqrt{\pi}}\psi_S \tr_s\left[g\frac{\partial B_{3,u^2,T}}{\partial T}\exp(-B_{3,u^2,T}^2)\right],
    & \hbox{if $n$ is even;} \\
    &-\frac{1}{\sqrt{\pi}}\psi_S \tr^{\mathrm{even}}\left[g\frac{\partial B_{3,u^2,T}}{\partial T}\exp(-B_{3,u^2,T}^2)\right],
    & \hbox{if $n$ is odd.}
  \end{aligned}
\right.
\end{split}
\end{align}
By Definition \ref{e01083} and Remark \ref{e01067}, we know that
\begin{align}\label{e03201}
\widetilde{\eta}_g(T^H_1W, g_T^{TZ}, h^{L_Z}, \nabla^{L_Z}, h^E, \nabla^E)=-\int_0^{+\infty}\beta_g^u(T,u)du.
\end{align}

\begin{prop}\label{e03003}
There exists a smooth family $\alpha_g:\mathbb{R}_{+,T}\times\mathbb{R}_{+,u}\rightarrow\Omega^*(S)$ such that
\begin{align}\label{e03004}
\left(du\wedge\frac{\partial}{\partial u}+dT\wedge\frac{\partial}{\partial T}\right)\beta_g=dT\wedge du\wedge d^S\alpha_g.
\end{align}
\end{prop}
\begin{proof}

We denote by $\alpha_g$ the coefficient of $du\wedge dT$ component of $\psi_S\widetilde{\tr}[g\exp(-\widehat{B}^2)]$.
Then
\begin{align}\label{e03007}
\psi_S\widetilde{\tr}[g\exp(-\widehat{B}^2)]=\psi_S\widetilde{\tr}[g\exp(-B_{3,u^2,T}^2)]+\beta_g+du\wedge dT\wedge \alpha_g.
\end{align}
Since $\psi_S\widetilde{\tr}[g\exp(-\widehat{B}^2)]$ and $\psi_S\widetilde{\tr}[g\exp(-B_{3,u^2,T}^2)]$ are closed forms,
we have
\begin{multline}\label{e03002}
\left(du\wedge\frac{\partial}{\partial u}+dT\wedge\frac{\partial}{\partial T}\right)
\psi_S\widetilde{\tr}[g\exp(-B_{3,u^2,T}^2)]-dT\wedge du\wedge d^S\alpha_g+d^S\beta_g
\\
+
\left(du\wedge\frac{\partial}{\partial u}+dT\wedge\frac{\partial}{\partial T}\right)\beta_g=0.
\end{multline}
Then Proposition \ref{e03003} follows from comparing the coefficient of $dT\wedge du$ in (\ref{e03002}).
\end{proof}

Take $\var, A, T_0$, $0<\var\leq 1\leq A<\infty$, $1\leq T_0<\infty$. Let $\Gamma=\Gamma_{\var,A,T_0}$ be the oriented contour in
$\mathbb{R}_{+,T}\times\mathbb{R}_{+,u}$.

\

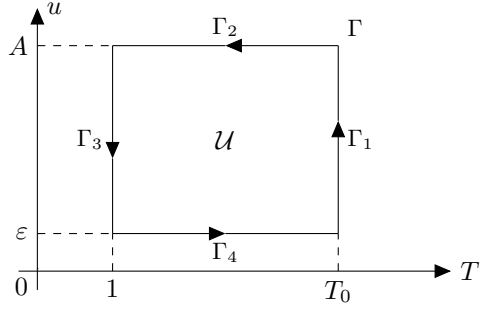
\begin{center}\label{e03010}
\begin{tikzpicture}
\draw[->][ -triangle 45] (-0.25,0) -- (5.5,0);
\draw[->][ -triangle 45] (0,-0.25) -- (0,3.5);
\draw[->][ -triangle 45] (1,0.5) -- (2.5,0.5);
\draw (2.5,0.5) -- (4,0.5);
\draw[->][ -triangle 45] (1,3) -- (1,1.5);
\draw (1,1.5) -- (1,0.5);
\draw[->][ -triangle 45] (4,0.5) -- (4,2);
\draw (4,2) -- (4,3);
\draw[->][ -triangle 45] (4,3) -- (2.5,3);
\draw (2.5,3) -- (1,3);
\draw[dashed] (0,0.5) -- (1,0.5);
\draw[dashed] (0,3) -- (1,3);
\draw[dashed] (1,0) -- (1,0.5);
\draw[dashed] (4,0) -- (4,0.5);
\foreach \x in {0}
\draw (\x cm,1pt) -- (\x cm,1pt) node[anchor=north east] {$\x$};
\draw
(2.5,1.75)  node {$\mU$}(2.5,1.75);
\draw
(0,3.5)  node[anchor=west] {$u$}(0,3.5);
\draw
(5.5,0)  node[anchor=west] {$T$}(5.5,0);
\draw
(0,0.5)  node[anchor=east] {$\varepsilon$}(0,0.5);
\draw
(0,3)  node[anchor=east] {$A$}(0,3);
\draw
(1,0)  node[anchor=north] {$1$}(1,0);
\draw
(4,0)  node[anchor=north] {$T_0$}(4,0);
\draw
(4,1.75)  node[anchor=west] {\small{$\Gamma_1$}}(4,1.75);
\draw
(2.5,0.5)  node[anchor=north] {\small{$\Gamma_4$}}(2.5,0.5);
\draw
(2.5,3)  node[anchor=south] {\small{$\Gamma_2$}}(2.5,3);
\draw
(1,1.75)  node[anchor=east] {\small{$\Gamma_3$}}(1,1.75);
\draw
(4,3)  node[anchor=south west] {\small{$\Gamma$}}(4,3);
\end{tikzpicture}
\end{center}
\

The contour $\Gamma$ is made of four oriented pieces $\Gamma_1,\cdots,\Gamma_4$ indicated in the above picture.
For $1\leq k\leq 4$,
set $I_k^0=\int_{\Gamma_k}\beta_g$. Then by Stocks' formula and Proposition \ref{e03003},
\begin{align}\label{e03011}
\sum_{k=1}^4I_k^0=\int_{\partial \mU}\beta_g=\int_{\mU}\left(du\wedge\frac{\partial}{\partial u}
+dT\wedge\frac{\partial}{\partial T}\right)\beta_g=d^S\left(\int_\mU \alpha_g dT\wedge du\right).
\end{align}

\subsection{Intermediate results}\label{s0302}

Now we state without proof some intermediate results, which will play an essential role in the proof of Theorem \ref{e02084}.
The proofs of these results are delayed to Section 4-8.

In the sequence, we will assume for simplicity that $S$ is compact. If $S$ is non-compact,
the various constants $C>0$ depend explicitly on the compact subset of $S$ on which the given estimate is valid.

As the arguments in the beginning of Section \ref{s0104},
let $\mathrm{Pr}_V:\widehat{V}=\R_+\times V\rightarrow V$ be the projection.
For the fibration $\widehat{V}\rightarrow \widehat{S}=\R_+\times S$,
let
($T^H_2\widehat{V}, \widehat{g}^{TY}, h^{\widehat{L}_Y}, \nabla^{\widehat{L}_Y}$)
be the quadruple such that
$T_2^H\widehat{V}=T(\R_+)\oplus \mathrm{Pr}_V^*(T^H_2V)$, $\widehat{g}^{TY}_{(t,v)}=t^{-2}g^{TY}_{v}$ for $t\in \R_+$, $v\in V$,
$\widehat{L}_Y=\mathrm{Pr}_V^*L_Y$,
$h^{\widehat{L}_Y}=\mathrm{Pr}_V^*h^{L_Y}$
and $\nabla^{\widehat{L}_Y}=\mathrm{Pr}_V^*\nabla^{L_Y}$.
Let $h^{\ker D^{\widehat{X}}}$ and $\nabla^{\ker D^{\widehat{X}}}$
be the induced metric and connection on the vector bundle $\ker D^{\widehat{X}}$.
Let $h^{\widehat{\mS}_Y}$ and $\nabla^{\widehat{\mS}_Y}$ be the induced metric and connection on $\mathrm{Pr}_V^*\mS(TY,L_Y)$.
We naturally extend the $G$-action to this case such that the $G$-action is identity on $\R_+\times S$.

Let $B_2$, $\widehat{B}_2$ and $B_{2,u^2}$ be the Bismut
superconnections  associated to ($T^H_2V, g^{TY}, h^{L_Y},$ $h^{\ker D^X},$
$\nabla^{L_Y}, \nabla^{\ker D^X}$),
($T^H_2\widehat{V}, \widehat{g}^{TY},h^{\widehat{L}_Y}$, $h^{\ker D^{\widehat{X}}}, \nabla^{\widehat{L}_Y},\nabla^{\ker D^{\widehat{X}}}$)
and ($T^H_2V, u^{-2}g^{TY}, h^{L_Y},$ $h^{\ker D^X}, \nabla^{L_Y}, \nabla^{\ker D^X})$ respectively.
For any $g\in G$, let us decompose
\begin{align}\label{e03041}
\psi_S\widetilde{\tr}[g\exp(-\widehat{B}_2^2)]=dt\wedge \gamma_2(t)+r_2(t),
\end{align}
where $\gamma_2(t), r_2(t)\in \Omega^*(S)$.
By Definition \ref{e01083}  and Remark \ref{e01067},
\begin{align}\label{e03042}
\int_0^{+\infty}\gamma_2(t)dt=-\widetilde{\eta}_g(T_2^HV, g^{TY},
h^{L_Y}, h^{\ker D^X}, \nabla^{L_Y}, \nabla^{\ker D^X}).
\end{align}

\begin{thm}\label{e03009}
i) For any $u>0$, we have
\begin{align}\label{e03015}
\lim_{T\rightarrow \infty}\beta_g^u(T,u)=\gamma_2(u).
\end{align}

ii) For $0<u_1<u_2$ fixed, there exists $C>0$ such that, for $u\in [u_1,u_2]$, $T\geq 1$, we have
\begin{align}\label{e03016}
|\beta_g^u(T,u)|\leq C.
\end{align}

iii) We have the following identity:
\begin{align}\label{e03017}
\lim_{T\rightarrow +\infty} \int_{1}^{\infty}\beta_g^u(T,u)du=\int_{1}^{\infty}\gamma_2(u)du.
\end{align}
\end{thm}

\begin{thm}\label{e03018}
We have the following identity:
\begin{align}\label{e03100}
\lim_{u\rightarrow +\infty} \int_{1}^{\infty}\beta_g^T(T,u)dT=0.
\end{align}
\end{thm}

Let $\mathrm{Pr}_W|_{V^g}:\widehat{W}|_{V^g}=\R_+\times W|_{V^g}\rightarrow W|_{V^g}$ be the projection.
For the fibration $\widehat{W}|_{V^g}\rightarrow \widehat{V}^g=\R_+\times V^g$,
let
($T^H_1(\widehat{W}|_{V^g}), \widehat{g}^{TX}, h^{\widehat{L}_X}, \nabla^{\widehat{L}_X}, h^{\widehat{E}}, \nabla^{\widehat{E}}$)
be the quadruple such that
$T^H_1(\widehat{W}|_{V^g})=T(\R_+)\oplus (\mathrm{Pr}_W|_{V^g})^*T^H_1(W|_{V^g})$,
$\widehat{g}^{TX}_{(t,w)}=t^{-2}g^{TX}_{w}$ for $t\in \R_+$, $w\in W|_{V^g}$,
$\widehat{L}_X=(\mathrm{Pr}_W|_{V^g})^*L_X$, $\widehat{E}=(\mathrm{Pr}_W|_{V^g})^*E$,
$h^{\widehat{L}_X}=(\mathrm{Pr}_W|_{V^g})^*h^{L_X}$, $h^{\widehat{E}}=(\mathrm{Pr}_W|_{V^g})^*h^{E}$,
$\nabla^{\widehat{L}_X}=(\mathrm{Pr}_W|_{V^g})^*\nabla^{L_X}$ and $\nabla^{\widehat{E}}=(\mathrm{Pr}_W|_{V^g})^*\nabla^{E}$.
We naturally extend the $G$-actions to this case such that $g$ acts trivially on $\widehat{V}^g$.

Let $\widehat{B}_1$ be the Bismut
superconnection  associated to
($T^H_1(\widehat{W}|_{V^g}), \widehat{g}^{TX}, \nabla^{\widehat{L}_X}, h^{\widehat{E}}, \nabla^{\widehat{E}}$).
For any $g\in G$, let us decompose
\begin{align}\label{e03043}
\psi_{V^g}\widetilde{\tr}[g\exp(-\widehat{B}_1^2)]=dt\wedge \gamma_1(t)+r_1(t),
\end{align}
where $\gamma_1(t), r_1(t)\in \Omega^*(S)$.
By Definition \ref{e01083} and Remark \ref{e01067},
\begin{align}\label{e03044}
\int_0^{+\infty}\gamma_1(t)dt=-\widetilde{\eta}_g(T_1^H(W|_{V^g}), g^{TX},
h^{L_X}, \nabla^{L_X}).
\end{align}

By (\ref{e01051}), $\widehat{\mathrm{A}}_g(TZ,\nabla^{TZ})$ only depends on $g\in G$ and $R^{TZ}$.
So we can denote it by $\widehat{\mathrm{A}}_g(R^{TZ})$. Let $R_T^{TZ}$ be the curvature of $\nabla_T^{TZ}$. Set
\begin{align}\label{e03005}
\gamma_{\mA}(T)=-\left.\frac{\partial}{\partial b}\right|_{b=0}\widehat{\mathrm{A}}_g\left(R_T^{TZ}+b\frac{\partial \nabla_T^{TZ}}
{\partial T}\right).
\end{align}
By a standard argument in Chern-Weil theory (see \cite[Appendix B]{MR2339952} and \cite{MR1864735}), we know that
\begin{align}\label{e03008}
\frac{\partial}{\partial T}\widetilde{\widehat{\mathrm{A}}}_g(TZ, \nabla^{TZ}, \nabla^{TZ}_T)=-\gamma_{\mA}(T).
\end{align}

\begin{prop}\label{e03047}
When $T\rightarrow +\infty$, we have $\gamma_{\mA}(T)=O(T^{-2})$. Moreover, modulo exact forms on $W^g$, we have
\begin{align}\label{e03048}
\widetilde{\widehat{\mathrm{A}}}_g(TZ,\nabla^{TZ}, \,^0\nabla^{TZ})=-\int_1^{+\infty}\gamma_{\mA}(T)dT.
\end{align}
\end{prop}

\begin{thm}\label{e03022}
i) For any $u>0$, there exist $C>0$ and $\delta>0$ such that, for $T\geq 1$, we have
\begin{align}\label{e03023}
|\beta_g^T(T,u)|\leq\frac{C}{T^{1+\delta}}.
\end{align}

ii) For any $T>0$, we have
\begin{align}\label{e03024}
\lim_{\var\rightarrow 0}\var^{-1}\beta_g^T(T\var^{-1},\var)=
\int_{Y^g}{\rm
\widehat{A}}_g(TY,\nabla^{TY})\wedge\ch_g(L_Y^{1/2}, \nabla^{L_Y^{1/2}})\wedge \gamma_1(T).
\end{align}

iii) There exists $C>0$ such that for $\var\in (0,1]$, $\var \leq T\leq 1$,
\begin{align}\label{e03025}
\var^{-1}\left|\beta_g^T(T\var^{-1},\var)
+ \int_{Z^g}\gamma_{\mA}(T\var^{-1})\wedge\ch_g(L_Z^{1/2}, \nabla^{L_Z^{1/2}})\wedge \ch_g(E,\nabla^E)\right|
\leq C.
\end{align}

iv) There exist $\delta\in (0,1]$,  $C>0$ such that, for $\var\in (0,1]$, $T\geq 1$,
\begin{align}\label{e03026}
\var^{-1}|\beta_g^T(T\var^{-1},\var)|\leq \frac{C}{T^{1+\delta}}.
\end{align}
\end{thm}

\subsection{Proof of Theorem \ref{e02084}}\label{s0303}

We now finish the proof of Theorem \ref{e02084} under the simplifying assumptions in Section \ref{s0202}.
By (\ref{e03011}), we know that
\begin{multline}\label{e03027}
\int_{\var}^{A}\beta_g^u(T_0,u)du-\int_{1}^{T_0}\beta_g^T(T,A)dT-\int_{\var}^{A}\beta_g^u(1,u)du+
\int_{1}^{T_0}\beta_g^T(T,\var)dT
\\
=I_1+I_2+I_3+I_4
\end{multline}
is an exact form.
We take the limits $A\rightarrow\infty$, $T\rightarrow\infty$ and then $\var\rightarrow 0$
in the indicated order. Let $I_j^k$, $j=1,2,3,4$, $k=1,2,3$ denote the value of the part $I_j$ after
the $k$th limit. By \cite[\S 22, Theorem 17]{MR0346830}, $d\Omega(S)$ is closed under uniformly
convergence on compact sets of $S$. Thus,
\begin{align}\label{e03006}
\sum_{j=1}^4I_j^3\equiv0\ \mathrm{mod}\ d\Omega^*(S).
\end{align}

From (\ref{e03201}), we obtain that
\begin{align}\label{e03029}
I_3^3=\tilde{\eta}_{g}(T_3^HW,,g^{TZ},h^{L_Z},\nabla^{L_Z}, h^E, \nabla^E).
\end{align}
Furthermore, by Theorem \ref{e03018}, we get
\begin{align}\label{e03030}
I_2^2=I_2^3=0.
\end{align}
From (\ref{e03042}) and Theorem \ref{e03009}, we conclude that
\begin{align}\label{e03031}
I_1^3
=-\tilde{\eta}_g(T_2^HV,g^{TY},h^{L_Y}, h^{\ker D^X}, \nabla^{L_Y}, \nabla^{\ker D^X}).
\end{align}

Finally, using Theorem \ref{e03022}, we get
\begin{multline}\label{e03032}
I_4^3=-\int_{Y^g}\mathrm{\widehat{A}}_g(TY,\nabla^{TY})\wedge\ch_g(L_Y^{1/2},\nabla^{L_Y^{1/2}})\wedge \tilde{\eta}_g(
T_1^H(W|_{V^g}),g^{TX},h^{L_X},\nabla^{L_X}, h^E, \nabla^E)
\\
+\int_{Z^g}\widetilde{\mathrm{\widehat{A}}}_{g}(TZ,\nabla^{TZ},\,^0\nabla^{TZ})\wedge\ch_g(L_Z^{1/2},\nabla^{L_Z^{1/2}})\wedge \ch_g(E,\nabla^E)
\end{multline}
as follows: We write
\begin{align}\label{e03033}
\begin{split}
\int_1^{+\infty}\beta_g^T(T,\var)dT
=\int_{\var}^{+\infty}\var^{-1}\beta_g^T(T\var^{-1},\var)dT.
\end{split}
\end{align}
Convergence of the integrals above is granted by (\ref{e03023}).
Using (\ref{e03024}), (\ref{e03026}) and Proposition \ref{e03047}, we get
\begin{align}\label{e03034}
\lim_{\var\rightarrow 0}\int_1^{+\infty}\var^{-1}\beta_g^T(T\var^{-1},\var)dT
=\int_{Y^g}\mathrm{\widehat{A}}_g(TY,\nabla^{TY})\wedge\ch_g(L_Y^{1/2},\nabla^{L_Y^{1/2}})\wedge
\int_1^{+\infty}\gamma_1(T)dT
\end{align}
and
\begin{multline}\label{e03035}
\lim_{\var\rightarrow 0}\int_{\var}^1\var^{-1}\left[\beta_g^T(T\var^{-1},\var)dT
+\int_{Z^g}\gamma_{\mA}(T\var^{-1})\wedge\ch_g(L_Z^{1/2},\nabla^{L_Z^{1/2}})\wedge\ch_g(E,\nabla^E)\right]dT
\\
=\int_{Y^g}\mathrm{\widehat{A}}_g(TY,\nabla^{TY})\wedge\ch_g(L_Y^{1/2},\nabla^{L_Y^{1/2}})\wedge
\int_0^{1}\gamma_1(T)dT.
\end{multline}
The remaining part of the integral yields by (\ref{e03025})
\begin{multline}\label{e03037}
\int_{\var }^1\var^{-1}\int_{Z^g}\gamma_{\mA}(T\var^{-1})\wedge\ch_g(L_Z^{1/2},\nabla^{L_Z^{1/2}})\wedge\ch_g(E,\nabla^E)dT
\\
=\int_{Z^g}\int_1^{+\infty}\gamma_{\mA}(T)\wedge\ch_g(L_Z^{1/2},\nabla^{L_Z^{1/2}})\wedge\ch_g(E,\nabla^E)dT
\\
=-\int_{Z^g}\widetilde{\mathrm{\widehat{A}}}_g(TZ,\nabla^{TZ},\,^0\nabla^{TZ})\wedge\ch_g(L_Z^{1/2},\nabla^{L_Z^{1/2}})\wedge\ch_g(E,\nabla^E).
\end{multline}
These four equations for $I_k^3$, $k=1,2,3,4$, imply Theorem \ref{e02084}.

\section{Proof of Theorem \ref{e03009}}\label{s04}

In this section, we use the assumptions and the notations of Section \ref{s0202} except that $D_T^Z$ is invertible for any $T\geq 1$.

This Section is organized as follows. In Section 4.1, we make some estimates of the fibrewise Dirac operator $D_T^Z$.
In Section 4.2, we write the operator $\mB_T$ in a matrix form. In Section 4.3, we state two intermediate results,
from which Theorem \ref{e03009} follows easily. We prove one of them in Section 4.3 and leave the proof of the other one to Section 4.4.
As a by-product of the first intermediate result in Section 4.3, we get a new proof of Lemma \ref{e02118}.
In Section 4.5, we prove a technical result Proposition \ref{e03047}.

\subsection{Estimates of $D_{T}^{Z,2}$}\label{s0401}

\begin{defn}
For $v\in V$, $b\in S$, let $\mathbb{E}_v$, $\mathbb{E}_{0,b}$ (resp. $\mathbb{E}_{1,b}$) be the vector spaces of the smooth
sections of $\pi_3^*\Lambda (T^*S)\widehat{\otimes}\mS(TZ,L_Z)\otimes E$ on $X_v$, $Z_b$ (resp. $\pi_2^*\Lambda (T^*S)
\widehat{\otimes}\mS(TY,L_Y)\otimes\ker D^X$ on $Y_b$).
For $\mu\in \R$, let $\mathbb{E}_{v}^{\mu}$, $\mathbb{E}_{0,b}^{\mu}$, $\mathbb{E}_{1,b}^{\mu}$
 be the Sobolev spaces of the order $\mu$
of sections of $\pi_3^*\Lambda (T^*S)\widehat{\otimes}\mS(TZ,L_Z)\otimes E$, $\pi_3^*\Lambda (T^*S)\widehat{\otimes}\mS(TZ,L_Z)\otimes E$,
$\pi_2^*\Lambda (T^*S)
\widehat{\otimes}\mS(TY,L_Y)\otimes\ker D^X$
 on $X_v$, $Z_b$, $Y_b$
with Sobolev norms $\|\cdot\|_{X,\mu}$, $\|\cdot\|_{\mu}$, $\|\cdot\|_{Y,\mu}$.
\end{defn}

For $v\in V$, in this section, we simply denote by $P_b$ the projection from $\mathbb{E}_{0,b}^0$ to $\mathbb{E}_{1,b}^0$ %with respect to
and let $P^{\bot}=1-P$.
Let $\mathbb{E}_1^{0,\bot}$ be the orthogonal bundle to $\mathbb{E}_1^0$ in $\mathbb{E}_0^0$.
Let $\mathbb{E}_1^{\mu,\bot}=\mathbb{E}_1^{0,\bot}\cap \mathbb{E}_0^{\mu}$.

Let $\{e_i\}$, $\{f_p\}$, $\{g_{\alpha}\}$ be the local orthonormal frames of $TX$, $TY$, $TS$ respectively
and $\{e^i\}$, $\{f^p\}$, $\{g^{\alpha}\}$ be their dual.
Recall that $\nabla^{\cE_X,u}$ is the connection in (\ref{e02109}). Set
\begin{align}\label{e04154}
\nabla^{\mS_Y\otimes \cE_X,u}=\nabla^{\mS_Y}\otimes 1+1\otimes \nabla^{\cE_X,u}.
\end{align}
Let
\begin{align}\label{e04152}
D^H=c(f_{p,1}^H)\nabla^{\mS_Y\otimes \cE_X,u}_{f_{p,1}^H}.
\end{align}
By (\ref{e02030}), we have
\begin{align}\label{e04155}
PD^HP=D^Y.
\end{align}
Let $S_2$ and $S_3$ be the tensor associated to ($T_2^HV, g^{TY}$) and ($T_3^HW, g^{TZ}$)
as in (\ref{e01017}). Let $T_1$ $T_2$, $T_3$, be the torsion tensors defined before (\ref{e01131})
associated to $(T_1^HW, g^{TX})$, $(T_2^HV, g^{TY})$, $(T_3^HW, g^{TZ})$. By (\ref{e01131}), we have
\begin{align}\label{e04070}
\la T_3(g_{\alpha,3}^H, g_{\beta,3}^H),f_{p,1}^H\ra=\la T_2(g_{\alpha,2}^H, g_{\beta,2}^H),f_{p}\ra.
\end{align}
From (\ref{e01017}), (\ref{e02029}) and (\ref{e04152}) (see also \cite[Theorem 10.19]{MR2273508}),
the Dirac operator $D^Z$ associated to $(T_3^HW, g^{TZ}, \nabla^{L_Z}, \nabla^E)$
can be written by
\begin{align}\label{e04009}
D^Z=D^X+D^H-\frac{1}{8}\la T_1(f_{p,1}^H, f_{q,1}^H),e_i\ra c(e_i)c(f_{p,1}^H)c(f_{q,1}^H).
\end{align}
If we replace the metric $g^{TZ}$ by $g_T^{TZ}$, by (\ref{e01131}), we have
\begin{align}\label{e04010}
D_T^Z=TD^X+D^H+\frac{1}{8T}\la [f_{p,1}^H, f_{q,1}^H],e_i\ra c(e_i)c(f_{p,1}^H)c(f_{q,1}^H).
\end{align}

\begin{defn}\label{e04019}
For $s,s'\in \mathbb{E}_0$, $T\geq 1$, we set
\begin{align}\label{e04020}
|s|_{T,0}^2:=\|s\|_0^2,
\end{align}
\begin{align}\label{e04221}
|s|_{T,1}^2:=\|Ps\|_{0}^2+T^2\|P^\bot s\|_0^2
+\sum_p\|\,^0\nabla^{\mS_Z\otimes E}_{f_{p,1}^H}s\|_0^2+T^2\sum_i\|\,^0\nabla^{\mS_Z\otimes E}_{e_i}P^\bot s\|_0^2,
\end{align}
where $\,^0\nabla^{\mS_Z\otimes E}=\,^0\nabla^{\mS_Z}\otimes 1+1\otimes \nabla^{ E}$.
Set
\begin{align}\label{e04022}
|s|_{T,-1}=\sup_{0\neq s'\in \mathbb{E}_0^1}\frac{|\langle s,s'\rangle_{0}|}{|s'|_{T,1}}.
\end{align}
\end{defn}

Then (\ref{e04221}) and (\ref{e04022}) define Sobolev norms on $\mathbb{E}_0^1$ and $\mathbb{E}_0^{-1}$.
Since $\,^0\nabla^{\mS_Z\otimes E}_{e_i}P$ is an operator along the fiber $X$ with smooth kernel,
we know that $|\cdot|_{T,1}$ (resp. $|\cdot|_{T,-1}$) is equivalent to $\|\cdot\|_1$ (resp. $\|\cdot\|_{-1}$)
on $\mathbb{E}_0^{1}$ (resp. $\mathbb{E}_0^{-1}$).

\begin{lemma}\label{e04030}
There exist $C_1, C_2, C_3>0,\, T_0\geq 1$, such that for
any $T\geq T_0$, $s,s'\in  \mathbb{E}_0$,
\begin{align}\label{e04031}
\begin{split}
\langle D_T^{Z,2}s,s\ra_{0}&\geq C_1|s|_{T,1}^2-C_2|s|_{T,0}^2,
\\
|\langle D_T^{Z,2}s,s'\ra_{0}|&\leq C_3|s|_{T,1}|s'|_{T,1}.
\end{split}
\end{align}
\end{lemma}
\begin{proof}
The proof of Lemma \ref{e04030} is almost the same as that of \cite[theorem 5.19]{MR1305280}.
For the completeness of this paper, we state the proof here.

Easy to check that $D_T^Z$ is a fiberwisely self-adjoint operator associated to $\la\cdot,\cdot\ra_{0}$ in (\ref{e01025}).
Set
\begin{align}\label{e04205}
D_T^H=D^H+\frac{1}{8T}\la [f_{p,1}^H, f_{q,1}^H],e_i\ra c(e_i)c(f_{p,1}^H)c(f_{q,1}^H).
\end{align}
Then by (\ref{e04010}),
\begin{align}\label{e04032}
D_T^{Z,2}=T^2D^{X,2}+D_T^{H,2}+T[D^{X},D_T^{H}].
\end{align}
The family of operators ($D^X, D_T^H$) is uniformly elliptic. So there exists $C_1', C_2'>0$, such that
for $T\in [1, +\infty]$, $s\in\mathbb{E}_0$,
\begin{align}\label{e04038}
\begin{split}
\|D^X s\|_0^2+\|D_T^H s\|_0^2\geq C_1'\|s\|_1^2-C_2'\|s\|_0^2.
\end{split}
\end{align}
Since $\ker D^X$ is a vector bundle, there exists $C_3'>0$,
\begin{align}\label{e04034}
\|D^X P^\bot s\|_0^2\geq C_3'\| P^\bot s\|_0^2.
\end{align}
Using (\ref{e04038}) and (\ref{e04034}), we get for $T\in [1, +\infty)$,
\begin{align}\label{e04039}
\begin{split}
T^2\|D^XP^\bot s\|_0^2+\|D_T^HP^\bot s\|_0^2
\geq
C_1'\|P^\bot s\|_1^2&+\frac{T^2-1}{2}\|D^XP^\bot s\|_0^2
\\
&+\left(\frac{C_3'(T^2-1)}{2}-C_2'\right)\|P^\bot s\|_0^2.
\end{split}
\end{align}
By elliptic estimate associated to the norm $\|\cdot\|_{X,\mu}$ and (\ref{e04034}),
there exists $C_4'>0$, such that
\begin{align}\label{e04002}
\|D^XP^{\bot}s\|_0^2\geq C_4'\sum_i\|\,^0\nabla^{\mS_Z\otimes E}_{e_i}P^{\bot}s\|_0^2.
\end{align}
Let $\,^0R$ be the curvature of $\,^0\nabla^{\mS_Z\otimes E}-\frac{1}{2}\la S_1(e_i)e_i, \cdot\ra$.
Then from a easy computation given by \cite[Theorem 2.5]{Bismut1985}, we have
\begin{align}\label{e04072}
[D^X, D^H]=c(e_i)c(f_{p,1}^H)\left(\,^0R(e_i,f_{p,1}^H)-\,^0\nabla^{\mS_Z\otimes E}_{T_1(e_i, f_{p,1}^H)}\right).
\end{align}
Since $T_1(e_i, f_{p,1}^H)\in TX$,
$[D^X,D^H]$ is a fiberwise first order elliptic operator along the fibers $X$.
By (\ref{e04205}), (\ref{e04034}), (\ref{e04002}) and (\ref{e04072}), there exists $C_5', C_6'>0$, such that for $T\geq 1$, $s\in \mathbb{E}_0$,
\begin{align}\label{e04037}
|\langle T[D^X,D_T^H] s, s\ra_{0}|\leq T|\langle[D^X,D^H]P^\bot s,P^\bot s\ra_{0}| +C_5'\|P^{\bot}s\|_0^2
\leq C_6'T\|D^XP^\bot s\|_0^2.
\end{align}
From (\ref{e04221}), (\ref{e04032}), (\ref{e04039}), (\ref{e04002}) and (\ref{e04037}), there exist $C_1'',C_2''>0$, $T_0\geq 1$ such that
for any $T\geq T_0$, $s\in\mathbb{E}_0$
\begin{align}\label{e04207}
\begin{split}
\la D_T^{Z,2} P^\bot s, P^\bot s\ra_0\geq C_1''|P^\bot s|_{T,1}^2+C_1'\|P^\bot s\|_1^2-C_2''\|s\|_{0}^2.
\end{split}
\end{align}

From (\ref{e04032}) and (\ref{e04038}), we have
\begin{align}\label{e04208}
\begin{split}
\la D_T^{Z,2} Ps, Ps\ra_0\geq C_1'\|Ps\|_1^2-C_2'\|s\|_0^2.
\end{split}
\end{align}

Since
\begin{align}\label{e04044}
\begin{split}
&\langle D_T^{H,2}P^\bot s,Ps\ra_{0}=\langle P^\bot s,D_T^{H,2}Ps\ra_{0}
\\
=&2\langle P^\bot s,[D_T^H,P]D_T^{H}s\ra_{0}+\langle P^\bot s,[D_T^{H},[D_T^H,P]]s\ra_{0}
\end{split}
\end{align}
and $[D_T^H,P]$, $[D_T^{H},[D_T^H,P]]$ are operators with smooth kernels along the fiber $X$, there exists $C_3''>0$, such that
\begin{align}\label{e04045}
\begin{split}
|\langle D_T^{H,2}P^\bot s,Ps\ra_{0}|\leq C_3''\|P^\bot s\|_1\|Ps\|_{0}.
\end{split}
\end{align}
As in (\ref{e04037}), there exists $C_4''>0$, such that
\begin{align}\label{e04046}
\begin{split}
|\langle T[D^X,D_T^H]P^\bot s,Ps\ra_{0}|\leq C_4''|P^\bot s|_{T,1}\|Ps\|_0.
\end{split}
\end{align}
So by (\ref{e04032}),
\begin{align}\label{e04047}
|\langle D_T^{Z,2}P^\bot s,Ps\rangle_0|\leq (C_3''+C_4'')|s|_{T,1}|s|_{T,0}.
\end{align}

Since $[^0\nabla^{\mS_Z\otimes E}, P]$ and $[^0\nabla^{\mS_Z\otimes E}, P^{\bot}]$ are bounded operators, there exists $C>0$, such that
\begin{align}\label{e04043}
\begin{split}
\|P^\bot s\|_1+\|Ps\|_1\geq \sum_p\|\,^0\nabla^{\mS_Z\otimes E}_{f_{p,1}^H}s\|_0^2+\sum_i\|^0\nabla^{\mS_Z\otimes E}_{e_i}P^\bot s\|_0^2-C\|s\|_0^2.
\end{split}
\end{align}
So from (\ref{e04207}), (\ref{e04208}), (\ref{e04047}) and (\ref{e04043}), we get the first inequality of (\ref{e04031}).
The second inequality follows directly from (\ref{e04032}) and (\ref{e04037}).

The proof of Lemma \ref{e04030} is complete.
\end{proof}

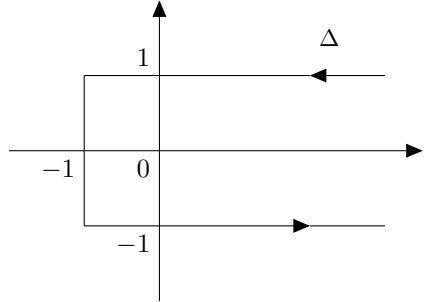
\begin{center}\label{e04051}
\begin{tikzpicture}[>=stealth]
\draw[->][ -triangle 45] (-2,0) -- (3.5,0);
\draw[->][ -triangle 45] (0,-2) -- (0,2);
\draw[->][ -triangle 45] (3,1) -- (2,1);
\draw (2,1) -- (-1,1);
\draw (-1,1) -- (-1,-1);
\draw[->][ -triangle 45] (-1,-1) -- (2,-1);
\draw (2,-1) -- (3,-1);
\foreach \x in {0}
\draw (0,0) node[anchor=north east] {$0$};
\draw
(0,2)  node[anchor=west] {$$}(0,2);
\draw
(3.5,0)  node[anchor=west] {$$}(3.5,0);
\draw
(0,1)  node[anchor=south east] {$1$}(0,1);
\draw
(0,-1)  node[anchor=north east] {$-1$}(0,-1);
\draw
(-1,0)  node[anchor=north east] {$-1$}(-1,0);
\draw
(2,1.5)  node[anchor=west] {\small{$\Delta$}}(2,1.5);
\end{tikzpicture}
\end{center}

Let $\Delta$ be the oriented contour in the above picture.

If $A\in \mathscr{L}(\mathbb{E}_0^0)$ (resp. $\mathscr{L}(\mathbb{E}_0^{-1},\mathbb{E}_0^{1})$), we note $\|A\|$
(resp. $|A|_{T}^{-1,1}$) the norm of $A$ with respect to the norm $\|\cdot\|_{0}$ (resp. the norms $|\cdot|_{T,-1}$ and $|\cdot|_{T,1}$).
Comparing with \cite[Theorem 11.27]{MR1188532}, we have the following lemma.
\begin{lemma}\label{e04055}
There exist $T_0\geq 1, C>0$, such that for $T\geq T_0$, $\lambda\in \Delta$,
the resolvent $(\lambda-D_T^{Z,2})^{-1}$ exists, extends to a continuous linear operator from
$\mathbb{E}_0^{-1}$ into $\mathbb{E}_0^1$, and moreover
\begin{align}\label{e04057}
\begin{split}
&\|(\lambda-D_T^{Z,2})^{-1}\|\leq C,
\\
&|(\lambda-D_T^{Z,2})^{-1}|_{T}^{-1.1}\leq C(1+|\lambda|)^2.
\end{split}
\end{align}
\end{lemma}
\begin{proof}
Since $D_T^Z$ is fiberwisely self-adjoint, for $\lambda\in \C\backslash \R^+$, $(\lambda-D_T^{Z,2})^{-1}$ exists.

For $\lambda=a\pm i\in\C$, $s\in  \mathbb{E}_0^2$,
\begin{align}\label{e04059}
\begin{split}
|\la(D_T^{Z,2}-\lambda)s,s\ra_{0}|&\geq \|s\|_0^2.
\end{split}
\end{align}
So there exists $C>0$, such that for any $\lambda\in \Delta$,
\begin{align}\label{e04060}
\|(\lambda-D_T^{Z,2})^{-1}s\|_0\leq C \|s\|_0.
\end{align}
So we get the first inequality of (\ref{e04057}).

Take $C_2$ the constant in Lemma \ref{e04030}.
For $\lambda_0\in\R$, $\lambda_0\leq -2C_2$,
by (\ref{e04031}), we have
\begin{align}\label{e04063}
|\la (\lambda_0-D_T^{Z,2})s,s\ra_{0}|\geq C_1|s|_{T,1}^2.
\end{align}
Then by (\ref{e04022}) and (\ref{e04063}),
\begin{align}\label{e04064}
|(\lambda_0-D_T^{Z,2})s|_{T,-1}=\sup_{0\neq s'\in \mathbb{E}_0^1}\frac{|\langle (\lambda_0-D_T^{Z,2})s,s'\rangle_{0}|}{|s'|_{T,1}}
\geq C_1|s|_{T,1}.
\end{align}
For $\lambda\in \Delta$,
\begin{align}\label{e04065}
(\lambda-D_T^{Z,2})^{-1}=(\lambda_0-D_T^{Z,2})^{-1}+(\lambda-\lambda_0)(\lambda-D_T^{Z,2})^{-1}(\lambda_0-D_T^{Z,2})^{-1}.
\end{align}
From (\ref{e04060}), (\ref{e04064}) and (\ref{e04065}), we deduce that $(\lambda-D_T^{Z,2})^{-1}$ extends to a
linear map from $ \mathbb{E}_0^{-1}$ into $ \mathbb{E}_0^0$ and
\begin{align}\label{e04066}
\begin{split}
|(\lambda-D_T^{Z,2})^{-1}s|_{T,0}&\leq
|(\lambda_0-D_T^{Z,2})^{-1}s|_{T,0}+|\lambda_0-\lambda||(\lambda-D_T^{Z,2})^{-1}(\lambda_0-D_T^{Z,2})^{-1}s|_{T,0}
\\
&\leq C_1^{-1}|s|_{T,-1}+C|\lambda_0-\lambda||(\lambda_0-D_T^{Z,2})^{-1}s|_{T,0}
\\
&\leq (C_1^{-1}+CC_1^{-1}|\lambda_0-\lambda|)|s|_{T,-1}.
\end{split}
\end{align}
On the other hand,
\begin{align}\label{e04067}
(\lambda-D_T^{Z,2})^{-1}=(\lambda_0-D_T^{Z,2})^{-1}+(\lambda-\lambda_0)(\lambda_0-D_T^{Z,2})^{-1}(\lambda-D_T^{Z,2})^{-1}.
\end{align}
So from (\ref{e04064}), (\ref{e04066}) and (\ref{e04067}), we deduce that $(\lambda-D_T^{Z,2})^{-1}$ extends to a
linear map from $ \mathbb{E}_0^{-1}$ into $ \mathbb{E}_0^1$ and
\begin{align}\label{e04068}
\begin{split}
|(\lambda-D_T^{Z,2})^{-1}s|_{T,1}&\leq
|(\lambda_0-D_T^{Z,2})^{-1}s|_{T,1}+|\lambda_0-\lambda||(\lambda_0-D_T^{Z,2})^{-1}(\lambda-D_T^{Z,2})^{-1}s|_{T,1}
\\
&\leq C_1^{-1}|s|_{T,-1}+C_1^{-1}|\lambda_0-\lambda||(\lambda-D_T^{Z,2})^{-1}s|_{T,0}
\\
&\leq (C_1^{-1}+C_1^{-1}|\lambda_0-\lambda|(C_1^{-1}+CC_1^{-1}|\lambda_0-\lambda|))|s|_{T,-1}.
\end{split}
\end{align}
Then we get the second inequality of (\ref{e04057}).

The proof of Lemma \ref{e04055} is complete.
\end{proof}

\subsection{The matrix structure}\label{s0402}

In the sequence, if $\alpha_T$ $(T\in [1,+\infty])$ is a family of tensors (resp. differential operators),
we write that as $T\rightarrow +\infty$,
\begin{align}\label{e04104}
\alpha_T=\alpha_{\infty}+O\left(\frac{1}{T^k}\right),
\end{align}
if for any $p\in \N$, there exists $C>0$, such that for $T\geq 1$, the sup of the norms of the coefficients of
$\alpha_T-\alpha_{\infty}$ and their derivatives of order $\leq p$ is dominated by $C/T^k$.

Recall that $\cE_Z$ is the infinite dimensional fiber bundle over $S$, whose fibers are
the set of smooth sections over $Z$ of $\mS(TZ, L_Z)$. Comparing with (\ref{e01028}), for $U\in TS$,
we define the connections on $\cE_Z$
\begin{align}\label{e04005}
\begin{split}
&\,^0\nabla^{\cE_Z,u}_U=\,^0\nabla_{U_3^H}^{\mS_Z\otimes E}-\frac{1}{2}\la S_3(e_i)e_i,U_3^H\ra -\frac{1}{2}\la S_3(f_{p,1}^H,f_{p,1}^H), U_3^H\ra,
\\
&\nabla^{\cE_Z,T,u}_U=\nabla_{U_3^H}^{\mS_Z,T}\otimes 1+1\otimes \nabla^E-\frac{1}{2}\la S_3(e_i)e_i,U_3^H\ra -\frac{1}{2}\la S_3(f_{p,1}^H,f_{p,1}^H), U_3^H\ra.
\end{split}
\end{align}
By (\ref{e02029}) and (\ref{e04005}), we have
\begin{align}\label{e04200}
\nabla^{\cE_Z, T, u}_U=\,^0\nabla^{\cE_Z,u}_U+\frac{1}{2T}\la S_1(U_3^H)e_i, f_{p,1}^H \ra c(e_i)c(f_{p,1}^H).
\end{align}

Recall that $B_{3,u^2,T}$ is the Bismut superconnection associated to $(T_3^HW$, $u^{-2}g_T^{TZ}$, $h^{L_Z}$, $\nabla^{L_Z}, h^E, \nabla^E)$.
Denote by $B_{3,T}=B_{3,1,T}$. From (\ref{e01032}), (\ref{e01042}), (\ref{e02027}), (\ref{e04070}), (\ref{e04010}),
(\ref{e04005}) and (\ref{e04200}), we can calculate $B_{3,T}$ and $B_{3,u^2,T}$ exactly.
\begin{prop}\label{e04006}
For $T>0$ and $u>0$,
\begin{multline}\label{e04007}
B_{3,T}
=TD^X+\,^0\nabla^{\cE_Z,u}+D^H-\frac{c(T_2)}{4}-\frac{1}{8T}\la T_1(f_{p,1}^H, f_{q,1}^H),e_i\ra c(e_i)c(f_{p,1}^H)c(f_{q,1}^H)
\\
+\frac{1}{2T}\la S_1(g_{\alpha}^H)e_i, f_{p,1}^H \ra c(e_i)c(f_{p,1}^H)g^{\alpha}\wedge
-\frac{1}{8T}\la T_3(g_{\alpha,3}^H, g_{\beta,3}^H),e_i\ra c(e_i)g^{\alpha}\wedge g^{\beta}\wedge,
\end{multline}
and
\begin{multline}\label{e04008}
B_{3,u^2,T}=uTD^X+uD^H-\frac{u}{8T}\la T_1(f_{p,1}^H, f_{q,1}^H),e_i\ra c(e_i)c(f_{p,1}^H)c(f_{q,1}^H)
\\
+\,^0\nabla^{\cE_Z,u}+\frac{1}{2T}\la S_1(g_{\alpha}^H)e_i, f_{p,1}^H \ra c(e_i)c(f_{p,1}^H)g^{\alpha}\wedge
\\
-\frac{c(T_2)}{4u}-\frac{1}{8uT}\la T_3(g_{\alpha,3}^H, g_{\beta,3}^H),e_i\ra c(e_i)g^{\alpha}\wedge g^{\beta}\wedge.
\end{multline}
\end{prop}

Let $\cE_Y$ be the infinite dimensional fiber bundle over $S$, whose fibers are the set of
smooth sections over $Y$ of $\mS(TY, L_Y)\otimes \ker D^X$. By (\ref{e01028}), for $U\in TS$,
we define the connections on $\cE_Y$
\begin{align}\label{e04151}
\nabla^{\cE_Y,u}_U=\nabla_{U_2^H}^{\mS_Y\otimes \ker D^X}-\frac{1}{2}\la S_2(f_p)f_p, U_2^H\ra.
\end{align}
From \cite[Theorem 5.2]{MR1942300}, we have
\begin{align}\label{e04071}
\la S_3(f_{p,1}^H, f_{q,1}^H),U_{3}^H\ra=\la S_2(f_p)f_p, U_2^H\ra.
\end{align}
So by (\ref{e02109}), (\ref{e02030}), (\ref{e04005}), (\ref{e04151}) and (\ref{e04071}), we have
\begin{align}\label{e04203}
\nabla^{\cE_Y, u}=P\,^0\nabla^{\cE_Z,u}P.
\end{align}
Recall that $B_2$ is the Bismut superconnection associated to
($T_2^HV$, $g^{TY}$,$ h^{L_Y}$, $h^{\ker D^X},$$ \nabla^{L_Y}$, $\nabla^{\ker D^X}$)
and $B_{2,u^2}=u^2\delta_{u^2}B_2\delta_{u^2}^{-1}$.
Then by (\ref{e01032}),
\begin{align}\label{e04003}
B_2=D^Y+\nabla^{\cE_Y,u}-c(T_2)/4.
\end{align}

\begin{lemma}\label{e04013}
For any $T\in [1,+\infty]$, the operator $PB_{3,T}P$ is a superconnection on $\mathbb{E}_1$. When $T\rightarrow +\infty$,
\begin{align}\label{e04014}
PB_{3,T}P=B_2+O\left(\frac{1}{T}\right).
\end{align}
\end{lemma}
\begin{proof}
Set
\begin{align}\label{e04020}
\mC=\,^0\nabla^{\cE_Z,u}+D^H-\frac{c(T_2)}{4}.
\end{align}
By (\ref{e04007}), we have
\begin{align}\label{e04201}
PB_{3,T}P=P\mC P+O\left(\frac{1}{T}\right).
\end{align}
From (\ref{e04155}), (\ref{e04203}) and (\ref{e04003}), we get
\begin{align}\label{e04204}
P\mC P=B_2.
\end{align}

So Lemma \ref{e04013} follows from (\ref{e04201}) and (\ref{e04204}).
\end{proof}

Set
\begin{align}\label{e04017}
\mB_T=B_{3,T}^2+u^{-2} du\wedge\, \delta_{u^2}^{-1}\frac{\partial B_{3,u^2,T}}{\partial u}\delta_{u^2}.
\end{align}
Then $\mB_T$ is a differential operator along the fiber $Z$ with values in $\Lambda(T^*(\R_+\times S))$.
Set
\begin{align}\label{e04301}
\mB_{u,T}=u^2\delta_{u^2}\mB_T\delta_{u^2}^{-1}=B_{3,u^2,T}^2+du\wedge \frac{\partial B_{3,u^2,T}}{\partial u}.
\end{align}
Then by (\ref{e03040}), we have
\begin{align}\label{e04012}
\beta_g^u=\left\{\psi_S\widetilde{\tr}[g\exp(-\mB_{u,T})]\right\}^{du}=
\left\{\psi_S\delta_{u^2}\widetilde{\tr}[g\exp(-u^2\mB_{T})]\right\}^{du}.
\end{align}
From Proposition \ref{e04006},
\begin{align}\label{e04015}
\delta_{u^2}^{-1}\frac{\partial B_{3,u^2,T}}{\partial u}\delta_{u^2}=TD^X+D^H+\frac{c(T_2)}{4}+O\left(\frac{1}{T}\right).
\end{align}

Set
\begin{align}\label{e04302}
\mB_2=B_2^2+u^{-2}du\wedge\delta_{u^2}^{-1}\frac{\partial B_{2,u^2}}{\partial u}\delta_{u^2}.
\end{align}
By (\ref{e03041}), we have
\begin{align}\label{e04026}
\gamma_2(u)=\left\{\psi_S\delta_{u^2}\widetilde{\tr}[g\exp(-u^2\mB_{2})]\right\}^{du}.
\end{align}

From (\ref{e04003}), (\ref{e04015}) and Lemma \ref{e04013}, we have
\begin{align}\label{e04303}
P\mB_{T}P=\mB_2+O\left(\frac{1}{T}\right).
\end{align}
Put
\begin{align}\label{e04018}
\begin{split}
&E_T=P\mB_TP,\ \ \ \ F_T=P\mB_TP^\bot,\ \
\\
&G_T=P^\bot \mB_TP,\ \ H_T=P^\bot \mB_TP^\bot.
\end{split}
\end{align}
Then we write $\mB_T$ in matrix form with respect to the splitting
$\mathbb{E}_0=\mathbb{E}_1^0\oplus \mathbb{E}_1^{0,\bot}$,
\begin{align}%\label{e04019}
\mB_T=\left(
  \begin{array}{cc}
    E_T & F_T \\
    G_T & H_T \\
  \end{array}
\right).
\end{align}

Similarly as in \cite[Theorem 5.5]{MR1942300}, we have
\begin{prop}\label{e04023}
There exist operators $E,F,G,H$ %and first order pseudodifferential operators $E(T), F(T), G(T), H(T)$
%along the fiber $Z$
such that, as $T\rightarrow +\infty$,
\begin{align}\label{e04024}
\begin{split}
&E_T=E+O(1/T),\quad \quad F_{T}=TF+O(1),\\
&G_{T}=TG+O(1),\ \ \quad \quad H_{T}=T^2H+O(T).
\end{split}
\end{align}
Let
\begin{align}\label{e04025}
Q=[D^X,\mC].
\end{align}
Then $Q(\mathbb{E}_1^0)\subset \mathbb{E}_1^{0,\bot}$, and $Q$ is a smooth family of first order elliptic operators
acting along the fibers $X$.
Moreover,
\begin{align}\label{e04027}
\begin{split}
E&=P(\mC^2+u^{-2}du\wedge(D^Y-c(T_2)/4))P,\quad F=PQP^{\bot},
\\
G&=P^{\bot}QP,\quad \quad \quad \quad \quad \quad \quad \quad \quad H= P^\bot D^{X,2}P^\bot,
\end{split}
\end{align}
and
\begin{align}\label{e04305}
\mB_{2}=E-FH^{-1}G.
\end{align}
\end{prop}
\begin{proof}
By (\ref{e04007}) and (\ref{e04020}), we have
\begin{align}\label{e04021}
B_{3,T}=TD^X+\mC+O\left(\frac{1}{T}\right).
\end{align}
From (\ref{e04017}) and (\ref{e04018}),
we get (\ref{e04027}).

Let $\,^0R_Z$ be the curvature of $\,^0\nabla^{\mS_Z\otimes E}-\frac{1}{2}\la S_3(e_i)e_i, \cdot\ra-\frac{1}{2}\la S_3(f_{p,1}^H)f_{p,1}^H, \cdot\ra$.
As in (\ref{e04072}), we have
\begin{align}\label{e04073}
\begin{split}
&[D^X, \,^0\nabla^{\cE_Z,u}]=c(e_i)g^{\alpha,H}_3\wedge\left(\,^0R_Z(e_i,g_{\alpha,3}^H)-\,^0\nabla^{\mS_Z\otimes E}_{T_3(e_i, g_{\alpha,3}^H)}\right),
\\
&(\,^0\nabla^{\cE_Z,u})^2=g_3^{\alpha,H}\wedge g_3^{\beta,H}\wedge\left(\,^0R_Z(g_{\alpha,3}^H,g_{\alpha,3}^H)
-\,^0\nabla^{\mS_Z\otimes E}_{T_3(g_{\alpha,3}^H,g_{\alpha,3}^H)}\right)
\end{split}
\end{align}
and $T_3(e_i, g_{\alpha,3}^H)\in TX$, $T_3(g_{\alpha,3}^H,g_{\alpha,3}^H)\in TZ$.
By (\ref{e04072}), (\ref{e04020}) and (\ref{e04073}), we know that $Q=[D^X,\mC]$ is a smooth family of first order elliptic operators
acting along the fibers $X$ and $Q(\mathbb{E}_1^0)\subset \mathbb{E}_1^{0,\bot}$.

By (\ref{e04003}), (\ref{e04302}) and (\ref{e04027}), we know that
\begin{multline}\label{e04028}
E-FH^{-1}G=P(\mC^2+u^{-2}du\wedge(D^Y-c(T_2)/4))P
\\
-P\mC D^X P^{\bot} (D^{X,2})^{-2}P^{\bot}D^X\mC P
=(P\mC P)^2+u^{-2}du\wedge(D^Y-c(T_2)/4)=\mB_2
\end{multline}

The proof of Proposition \ref{e04023} is complete.
\end{proof}

\subsection{Proof of Theorem \ref{e03009}}\label{s0403}

If $C$ is an operator, let $\mathrm{Sp}(C)$ be the spectrum of $C$.
The following lemma is an analogue of  \cite[Proposition 9.2]{Bismut1997}.

\begin{lemma}\label{e04004}
For any $u>0,\,T\geq 1$,
\begin{align}\label{e04058}
\begin{split}
{\rm Sp}(\mB_2)&=\mathrm{Sp}(D^{Y,2}),
\\
{\rm Sp}(\mB_{u,T})&={\rm Sp}(u^2D_T^{Z,2})={\rm Sp}(u^2\mB_{T}).
\end{split}
\end{align}
\end{lemma}
\begin{proof}
We only prove the first formula. The proof of the second one is the same.

By (\ref{e04003}) and (\ref{e04302}),
set
\begin{multline}\label{e29}
\mR:=\mB_2-D^{Y,2}=\left(\nabla^{\cE_Y,u}-\frac{1}{4}c(T_2)\right)^2+
\left[D^Y,\nabla^{\cE_Y,u}-\frac{1}{4}c(T_2)\right]
\\
+\frac{1}{u^2}du\wedge\left(D^Y-\frac{c(T_2)}{4}\right).
\end{multline}
Take $\lambda\notin {\rm Sp}(D^{Y,2})$. Then
\begin{align}\label{e30}
(\lambda-\mB_2)^{-1}-(\lambda-D^{Y,2})^{-1}=(\lambda-D^{Y,2})^{-1}\mR(\lambda-\mB_2)^{-1}.
\end{align}
Inductively,
\begin{align}\label{e31}
\begin{split}
(\lambda-\mB_2)^{-1}&=(\lambda-D^{Y,2})^{-1}+(\lambda-D^{Y,2})^{-1}\mR(\lambda-D^{Y,2})^{-1}
\\
&+(\lambda-D^{Y,2})^{-1}\mR(\lambda-D^{Y,2})^{-1}\mR(\lambda-D^{Y,2})^{-1}+\cdots.
\end{split}
\end{align}
Since $\mR$ has positive degree in $\Lambda (T^{*}(\R\times S))$, the expansion above has finite terms.

By elliptic estimate, there exist $c_1,c_2>0$,  such that for any $s\in \mathbb{E}_1$,
\begin{align}\label{e32}
\|(\lambda-D^{Y,2})s\|_{Y,0}\geq c_1\|s\|_{Y,2}-c_2\|s\|_{Y,0}.
\end{align}
Then there exists $c>0$ such that
\begin{align}\label{e33}
\|(\lambda-D^{Y,2})^{-1}s\|_{Y,2}\leq \frac{1}{c_1}\|s\|_{Y,0}+\frac{c_2}{c_1}
\|(\lambda-D^{Y,2})^{-1}s\|_{Y,0}\leq c\|s\|_{Y,0}.
\end{align}

From (\ref{e04073}) and (\ref{e29}), there exists $c>0$ such that
\begin{align}\label{e34}
\|\mR s\|_{Y,0}\leq c\|s\|_{Y,1}.
\end{align}
By (\ref{e31}), (\ref{e33}) and (\ref{e34}), there exists $c>0$, such that
\begin{align}\label{e35}
\|(\lambda-\mB_{2})^{-1}s\|_{Y,0}\leq c\|s\|_{Y,0}.
\end{align}
So $\lambda\notin {\rm Sp}(\mB_2)$.

Exchange $\mB_2$ and $D^{Y,2}$, we get  the first formula of (\ref{e04058}).
\end{proof}

By Lemma \ref{e04004}, we have
\begin{align}\label{e04099}
\begin{split}
\exp(-u^2\mB_T)=\frac{1}{2\pi\sqrt{-1}}\int_{\Delta}\frac{\exp(-u^2\lambda)}{\lambda-\mB_T}d\lambda,
\\
\exp(-u^2\mB_2)=\frac{1}{2\pi\sqrt{-1}}\int_{\Delta}\frac{\exp(-u^2\lambda)}{\lambda-\mB_2}d\lambda.
\end{split}
\end{align}

\begin{lemma}\label{e04304}
There exist $T_0\geq 1, C>0, k\in\N$, such that for $T\geq T_0$, $\lambda\in \Delta$,
the resolvent $(\lambda-\mB_T)^{-1}$ exists, extends to a continuous linear operator from
$\mathbb{E}_0^{-1}$ into $\mathbb{E}_0^1$, and moreover
\begin{align}\label{e04195}
\begin{split}
&\|(\lambda-\mB_T)^{-1}\|\leq C(1+|\lambda|)^k,
\\
&|(\lambda-\mB_T)^{-1}|_{T}^{-1.1}\leq C(1+|\lambda|)^k.
\end{split}
\end{align}
\end{lemma}
\begin{proof}
Set
\begin{align}\label{e04052}
\mR_T:=\mB_T-D_T^{Z,2}.
\end{align}
By  (\ref{e04072}), (\ref{e04007}), (\ref{e04017}) and (\ref{e04073}),  we know that
$\mR_T$ is a first order fiberwise differential operator along the fiber $Z$.
Moreover, from (\ref{e04221}), for $i=-1,0$, there exists $C_i>0$, such that for any $s\in \mathbb{E}_0^i$,
\begin{align}\label{e04054}
|\mR_Ts|_{T,i}\leq C_i|s|_{T,i+1}.
\end{align}
Take $\lambda\in \Delta$. Then
\begin{align}\label{e04053}
\begin{split}
(\lambda-\mB_T)^{-1}&=(\lambda-D_T^{Z,2})^{-1}+(\lambda-D_T^{Z,2})^{-1}\mR_T(\lambda-D_T^{Z,2})^{-1}
\\
&+(\lambda-D_T^{Z,2})^{-1}\mR_T(\lambda-D_T^{Z,2})^{-1}\mR_T(\lambda-D_T^{Z,2})^{-1}+\cdots.
\end{split}
\end{align}
Since $\mR_T$ has positive degree in $\Lambda (T^*(\R\times S))$, the expansion above has finite terms.

From  (\ref{e04054}), and (\ref{e04053}) and Lemma \ref{e04055}, there exist
$T_0\geq 1, C>0, k\in \N$, such that for $T\geq T_0$, $\lambda\in \Delta$,
the resolvent $(\lambda-\mB_T)^{-1}$ exists, extends to a continuous linear operator from
$\mathbb{E}_0^{-1}$ into $\mathbb{E}_0^1$, and moreover
\begin{align}\label{e04056}
\begin{split}
&\|(\lambda-\mB_T)^{-1}\|\leq C(1+|\lambda|)^k,
\\
&|(\lambda-\mB_T)^{-1}|_{T}^{-1,1}\leq C(1+|\lambda|)^k.
\end{split}
\end{align}

The proof of Lemma \ref{e04304} is complete.
\end{proof}

Similarly, there exist $C>0, k\in \N$, such that for $\lambda\in \Delta$,
the resolvent $(\lambda-\mB_2)^{-1}$ exists, and for any $s\in \mathbb{E}_1^0$, $s'\in \mathbb{E}_1^{-1}$, we have
\begin{align}\label{e04061}
\begin{split}
&\|(\lambda-\mB_2)^{-1}s\|_{Y,0}\leq C(1+|\lambda|)^k\|s\|_{Y,0},
\\
&\|(\lambda-\mB_2)^{-1}s'\|_{Y,1}\leq C(1+|\lambda|)^k\|s'\|_{Y,-1}.
\end{split}
\end{align}

Replacing $\mB_T$ by $H_T$ and $D_T^{Z,2}$ by $P^{\bot}D_T^{Z,2}P^{\bot}$ in the proof of Lemma \ref{e04304}, we can get the following lemma.
\begin{lemma}\label{e04075}
There exist $T_0\geq 1, C>0, k\in \N$, such that for $T\geq T_0, \lambda\in \Delta$,
the resolvent $(\lambda-H_T)^{-1}$ exists, and for any $s\in \mathbb{E}_0^{2,\bot}$, we have
\begin{align}\label{e04062}
\begin{split}
&\|(\lambda-H_T)^{-1}s\|_{0}\leq C(1+|\lambda|)^k\|s\|_{0},
\\
&|(\lambda-H_T)^{-1}s|_{T,1}\leq C(1+|\lambda|)^k|s|_{T,-1}.
\end{split}
\end{align}
\end{lemma}

Choose $s,s'\in \mathbb{E}_0$ such that
$s=(\lambda-\mB_T)^{-1}s'$, $\lambda\in \Delta$. Then by (\ref{e04018}), we have
\begin{align}\label{e04077}
\begin{split}
Ps'&=(\lambda-E_T)Ps-F_TP^{\bot}s,
\\
P^\bot s'&=-G_TPs+(\lambda-H_T)P^{\bot}s.
\end{split}
\end{align}
Let
\begin{align}\label{e04078}
\mE_T(\lambda)=\lambda-E_T-F_T(\lambda-H_T)^{-1}G_T.
\end{align}
Then
\begin{align}\label{e04079}
P(\lambda-\mB_T)^{-1}P=\mE_T(\lambda)^{-1}.
\end{align}
By (\ref{e04079}) and Lemma \ref{e04304},
there exist $T_0\geq 1, C>0, k\in\N$, such that for $T\geq T_0$, $\lambda\in \Delta$, $s\in \mathbb{E}_0$,
\begin{align}\label{e04074}
\begin{split}
&\|\mE_T(\lambda)^{-1}s\|_0\leq C(1+|\lambda|)^k\|s\|_0,
\\
&|\mE_T(\lambda)^{-1}s|_{T,1}\leq C(1+|\lambda|)^k|s|_{T,-1}.
\end{split}
\end{align}

\begin{lemma}\label{e04081}
There exist $C>0$, $T_0\geq 1$, $k\in \N$, such that for $T\geq T_0$, $\lambda\in \Delta$, $s\in \mathbb{E}_0$,
\begin{align}\label{e04082}
\|(\mE_T(\lambda)^{-1}-P(\lambda-\mB_2)^{-1}P)s\|_{0}\leq \frac{C(1+|\lambda|)^{k}}{T}\|s\|_{0}.
\end{align}
\end{lemma}
\begin{proof}
We know that
\begin{align}\label{e04090}
\mE_T(\lambda)^{-1}-P(\lambda-\mB_2)^{-1}P=P\mE_T(\lambda)^{-1}(\lambda-\mB_2-\mE_T(\lambda))(\lambda-\mB_2)^{-1}P.
\end{align}
By (\ref{e04305}) and (\ref{e04078}),
\begin{multline}\label{e04084}
\lambda-\mB_2-\mE_T(\lambda)
=E_T+F_T(\lambda-H_T)^{-1}G_T-E+FH^{-1}G
\\
=(E_T-E)+(F_T-TF)(\lambda-H_T)^{-1}G_T+\lambda TF(\lambda-H_T)^{-1}(T^{2}H)^{-1}G_T
\\
-TF(\lambda-H_T)^{-1}(H_T-T^2H)(T^{2}H)^{-1}G_T+TF(T^{2}H)^{-1}(G_T-TG).
\end{multline}

By (\ref{e04044}) and (\ref{e04007}), the 2-order term of the differential operator $\mB_T$
is
\begin{align}\label{e04001}
T^2P^{\bot}D^{X,2}P^{\bot}+PD^{H,2}P+P^{\bot}D^{H,2}P^{\bot}
\end{align}
and the coefficient of $T$ in the expansion of $\mB_T$ is a 1-order differential operator along the fiber $X$.

From (\ref{e04001}) and Proposition \ref{e04023}, there exist $C>0$, $T_0\geq 1$, such that for any $s, s'\in \mathbb{E}_0$,
$T\geq T_0$,
\begin{align}\label{e04085}
\begin{split}
|\la(E_T-E)Ps, Ps'\ra_0|
\leq \frac{C}{T}\|Ps\|_0\|Ps'\|_{1}.
\end{split}
\end{align}
So we have
\begin{align}\label{e04230}
|(E_T-E)Ps|_{T,-1}\leq \frac{C}{T}\|Ps\|_0.
\end{align}
Also from (\ref{e04001}) and Proposition \ref{e04023}, there exist $C>0$, $T_0\geq 1$, such that for any $s\in \mathbb{E}_0$,
$T\geq T_0$,
\begin{align}\label{e04231}
\|F_TP^{\bot}s\|_{0}\leq \|TQP^{\bot}s\|_0+C\|Ps\|_1\leq C|P^{\bot}s|_{T,1}.
\end{align}
Similarly, we have
\begin{align}\label{e04232}
|G_TPs|_{T,-1}\leq C\|Ps\|_0.
\end{align}
From (\ref{e04231}), (\ref{e04232}) and Lemma \ref{e04075}, there exist $C>0$, $T_0\geq 1$,
$k\in \N$ such that for any $s\in \mathbb{E}_0$,
$T\geq T_0$,
\begin{align}\label{e04233}
\|F_T(\lambda-H_T)^{-1}G_TPs\|_{0}\leq C(1+|\lambda|)^k\|Ps\|_0.
\end{align}
From Proposition \ref{e04023}, there exists $C>0$, such that
\begin{align}\label{e04234}
\|FH^{-1}GPs\|_{0}\leq C\|Ps\|_0.
\end{align}
By (\ref{e04061}), (\ref{e04074}), (\ref{e04084}), (\ref{e04230}), (\ref{e04233}), (\ref{e04234}) and Lemma \ref{e04075}, we can get
\begin{align}\label{e04235}
\|(\mE_T(\lambda)^{-1}-P(\lambda-\mB_2)^{-1}P)s\|_{0}\leq C(1+|\lambda|)^{k}\|Ps\|_{0}.
\end{align}

Comparing with (\ref{e04231}) and (\ref{e04232}), from (\ref{e04001}) and Proposition \ref{e04023},
there exist $C>0$, $T_0\geq 1$, such that for any $s\in \mathbb{E}_0$,
$T\geq T_0$,
\begin{align}\label{e04236}
\begin{split}
|(F_T-TF)P^{\bot}s|_{T,-1}&\leq C\|P^{\bot}s\|_0,\quad
|TFP^{\bot}s|_{T,-1}\leq C|P^{\bot}s|_{T,1},
\\
&\|(G_T-TG)Ps\|_{-1}\leq C\|Ps\|_0.
\end{split}
\end{align}

From (\ref{e04034}) and (\ref{e04002}), there exists $C>0$, such that for any $s\in \mathbb{E}_0$,
\begin{align}\label{e04237}
\la Hs, s\ra_0\geq \|P^{\bot}s\|_{X,1}^2.
\end{align}
So by Proposition \ref{e04023}, there exists $C>0$, such that
\begin{align}\label{e04238}
|QH^{-1}s|_{T,-1}\geq C\|P^{\bot}s\|_{-1}.
\end{align}

Thus, by (\ref{e04034}), (\ref{e04232}), (\ref{e04236}), (\ref{e04238}) and Lemma \ref{e04075}, we can get
\begin{multline}\label{e04239}
|(F_T-TF)(\lambda-H_T)^{-1}G_TPs|_{T,-1}\leq C\|P^{\bot} (\lambda-H_T)^{-1}G_TPs\|_0
\\
\leq \frac{C}{T}|(\lambda-H_T)^{-1}G_TPs|_{T,1}\leq \frac{C}{T}(1+|\lambda|)^k|G_TPs|_{T,-1}
\leq \frac{C}{T}(1+|\lambda|)^k\|Ps\|_{0},
\end{multline}
\begin{multline}\label{e04240}
|TF(\lambda-H_T)^{-1}(T^2H)^{-1}G_TPs|_{T,-1}\leq C |(\lambda-H_T)^{-1}(T^2H)^{-1}G_TPs|_{T,1}
\\
\leq  C(1+|\lambda|)^k |(T^2H)^{-1}G_TPs|_{T,-1}\leq \frac{C}{T^2}(1+|\lambda|)^k|G_TPs|_{T,-1}
\\
\leq \frac{C}{T^2}(1+|\lambda|)^k\|Ps\|_0
\end{multline}
and
\begin{multline}\label{e04241}
|TF(T^2H)^{-1}(G_T-TG)Ps|_{T,-1}=\frac{1}{T} |QH^{-1}(G_T-TG)Ps|_{T,-1}
\\
\leq  \frac{C}{T} \|(G_T-TG)Ps\|_{-1}\leq \frac{C}{T}\|Ps\|_0.
\end{multline}

So from (\ref{e04061}), (\ref{e04074}), (\ref{e04090}), (\ref{e04084}), (\ref{e04230}), (\ref{e04235}),
(\ref{e04239}), (\ref{e04240}), (\ref{e04241}) and Lemma \ref{e04075},
we have
\begin{multline}\label{e04242}
\|(\mE_T(\lambda)^{-1}TF(\lambda-H_T)^{-1}(H_T-T^2H)(T^{2}H)^{-1}G_T(\lambda-\mB_2)^{-1}Ps\|_0
\\
\leq C(1+|\lambda|)^{k}\|Ps\|_{0}.
\end{multline}

On the other hand, from (\ref{e04001}), we have
\begin{align}\label{e04243}
|(H_T-T^2H)P^{\bot}s|_{T,-1}\geq C\|P^{\bot}s\|_{1}.
\end{align}
So from (\ref{e04074}), (\ref{e04236}), (\ref{e04243}) and Lemma \ref{e04075}, we have
\begin{multline}\label{e04244}
\|(\mE_T(\lambda)^{-1}TF(\lambda-H_T)^{-1}(H_T-T^2H)(T^{2}H)^{-1}G_T(\lambda-\mB_2)^{-1}Ps\|_0
\\
\leq C(1+|\lambda|)^{k}|TF(\lambda-H_T)^{-1}(H_T-T^2H)(T^{2}H)^{-1}G_T(\lambda-\mB_2)^{-1}Ps|_{T,-1}
\\
\leq C(1+|\lambda|)^{k}|(\lambda-H_T)^{-1}(H_T-T^2H)(T^{2}H)^{-1}G_T(\lambda-\mB_2)^{-1}Ps|_{T,1}
\\
\leq C(1+|\lambda|)^{k}|(H_T-T^2H)(T^{2}H)^{-1}G_T(\lambda-\mB_2)^{-1}Ps|_{T,-1}
\\
\leq \frac{C}{T^2}(1+|\lambda|)^{k}\|H^{-1}G_T(\lambda-\mB_2)^{-1}Ps\|_{1}.
\end{multline}
Since $H^{-1}G_T(\lambda-\mB_2)^{-1}=O(T)$, by (\ref{e04242}) and (\ref{e04244}), we have
\begin{multline}\label{e04245}
\|(\mE_T(\lambda)^{-1}TF(\lambda-H_T)^{-1}(H_T-T^2H)(T^{2}H)^{-1}G_T(\lambda-\mB_2)^{-1}Ps\|_0
\\
\leq \frac{C}{T}(1+|\lambda|)^{k}\|Ps\|_{0}.
\end{multline}
Then from (\ref{e04090}), (\ref{e04084}), (\ref{e04230}), (\ref{e04239}), (\ref{e04240}), (\ref{e04241}),
(\ref{e04245}) and Lemma \ref{e04075}, we can obtain the Lemma.
\end{proof}

\begin{lemma}\label{e04094}
There exist $C>0$, $T_0\geq 1$, $k\in \N$, such that for $T\geq T_0$, $\lambda\in \Delta$,
\begin{align}\label{e04095}
\begin{split}
\|(\lambda-\mB_T)^{-1}-P(\lambda-\mB_2)^{-1}P\|\leq \frac{C}{T}(1+|\lambda|)^{k}.
\end{split}
\end{align}
\end{lemma}
\begin{proof}
From (\ref{e04079}) and Lemma \ref{e04081}, we have
\begin{align}\label{e04246}
\|P(\lambda-\mB_T)^{-1}P-P(\lambda-\mB_2)^{-1}P\|\leq \frac{C}{T}(1+|\lambda|)^{k}.
\end{align}

By (\ref{e04077}), we find that
\begin{align}\label{e04096}
\begin{split}
&P(\lambda-\mB_T)^{-1}P^{\bot}=\mE_T(\lambda)^{-1}F_T(\lambda-H_T)^{-1},
\\
&P^{\bot}(\lambda-\mB_T)^{-1}P=(\lambda-H_T)^{-1}G_T\mE_T(\lambda)^{-1},
\\
&P^{\bot}(\lambda-\mB_T)^{-1}P^{\bot}=(\lambda-H_T)^{-1}(1+G_TP(\lambda-\mB_T)^{-1}P^{\bot}).
\end{split}
\end{align}

From  (\ref{e04074}), (\ref{e04231}) and Lemma \ref{e04075}, there exists $C>0$, such that for $s\in \mathbb{E}_0$,
\begin{multline}\label{e04307}
\|P(\lambda-\mB_T)^{-1}P^{\bot}s\|_0=\|\mE_T(\lambda)^{-1}F_T(\lambda-H_T)^{-1}P^{\bot}s\|_{0}
\\
\leq C\|F_T(\lambda-H_T)^{-1}P^{\bot}s\|_{0}
\leq C|(\lambda-H_T)^{-1}P^{\bot}s|_{T,1}\leq C(1+|\lambda|)^{k}|P^{\bot}s|_{T,-1}
\\
\leq \frac{C}{T}(1+|\lambda|)^{k}\|s\|_0.
\end{multline}

From (\ref{e04074}), (\ref{e04232}) and Lemma \ref{e04075}, there exists $C>0$, such that for $s\in \mathbb{E}_0$,
\begin{multline}\label{e04308}
\|P^{\bot}(\lambda-\mB_T)^{-1}Ps\|_0=\|(\lambda-H_T)^{-1}G_T\mE_T(\lambda)^{-1}Ps\|_{0}
\\
\leq \frac{1}{T}|(\lambda-H_T)^{-1}G_T\mE_T(\lambda)^{-1}Ps|_{T,1}
\leq \frac{C}{T}(1+|\lambda|)^k|G_T\mE_T(\lambda)^{-1}Ps|_{T,-1}
\\
\leq \frac{C}{T}(1+|\lambda|)^k\|\mE_T(\lambda)^{-1}Ps\|_{1}
\leq \frac{C}{T}(1+|\lambda|)^{2k}\|s\|_0.
\end{multline}

From (\ref{e04307}) and (\ref{e04308}), there exists $C>0$, such that for $s\in \mathbb{E}_0$,
\begin{multline}\label{e04309}
\|(\lambda-H_T)^{-1}G_T\mE_T(\lambda)^{-1}F_T(\lambda-H_T)^{-1}P^{\bot}s\|_{0}
\\
\leq \frac{C}{T}(1+|\lambda|)^{k}\|F_T(\lambda-H_T)^{-1}P^{\bot}s\|_0\leq \frac{C}{T^2}(1+|\lambda|)^{2k}\|s\|_0.
\end{multline}
From Lemma \ref{e04075}, we have
\begin{multline}\label{e04035}
\|(\lambda-H_T)^{-1}s\|_0\leq \frac{1}{T}|(\lambda-H_T)^{-1}s|_{T,1}\leq \frac{C}{T}(1+|\lambda|)^{k}|P^{\bot}s|_{T,-1}
\\
\leq \frac{C}{T^2}(1+|\lambda|)^{k}\|s\|_0.
\end{multline}
By (\ref{e04309}) and (\ref{e04035}), we get
\begin{align}\label{e04247}
\|P^{\bot}(\lambda-\mB_T)^{-1}P^{\bot}\|\leq\frac{C}{T^2}(1+|\lambda|)^k.
\end{align}

The proof of Lemma \ref{e04094} is complete.
\end{proof}

We assume that $\ker D^Y=0$. There exists $c_1>0$, such that $\mathrm{Sp}(\mB_2)=\mathrm{Sp}(D^{Y,2})\subset [2c_1, +\infty)$.
By Lemma \ref{e04004} and Proposition \ref{e04094}, we know that when $T$ is sufficiently large,
\begin{align}\label{e04040}
\mathrm{Sp}(D_T^{Z,2})=\mathrm{Sp}(\mB_T)\subset [c_1, +\infty).
\end{align}
Note that in this section, we need not assume that $\ker D_T^Z=0$.
Therefore, we get another proof of Lemma \ref{e02118}.

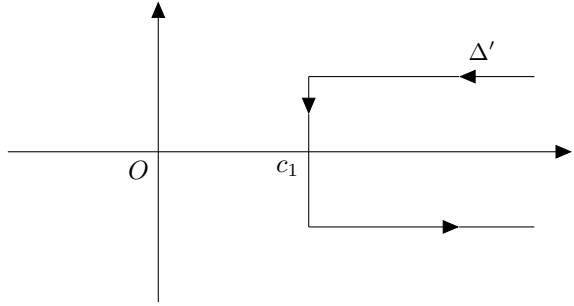
\begin{center}\label{e05010}
\begin{tikzpicture}[>=stealth]
\draw[->][ -triangle 45] (-2,0) -- (5.5,0);
\draw[->][ -triangle 45] (0,-2) -- (0,2);
\draw[->][ -triangle 45] (5,1) -- (4,1);
\draw (4,1) -- (2,1);
\draw[->][ -triangle 45] (2,1) --(2,0.5);
\draw (2,0.5) --(2,-1);
\draw[->][ -triangle 45] (2,-1) -- (4,-1);
\draw(4,-1) -- (5,-1);
\draw
(2,0)  node[anchor=north east] {$c_1$}(2,0);
\draw
(4,1.35)  node[anchor=west] {\small{$\Delta'$}}(4,1.75);
\draw
(0,2)  node[anchor=west] {$$}(0,2);
\draw
(5.5,0)  node[anchor=west] {$$}(5.5,0);
\draw
(0,0)  node[anchor=north east] {$O$}(0,0);
\end{tikzpicture}
\end{center}

Let $\Delta'$ be the oriented contour in the above picture.
Then all the estimates
in this Section  hold for any $\lambda\in \Delta'$.
From (\ref{e04040}), there exists $T_0\geq 1$, for $u>0$, $T\geq T_0$,
\begin{align}\label{e05009}
\exp(-u^2\mB_T)=%\frac{1}{2\pi\sqrt{-1}}\int_{\delta_0}\frac{e^{-u^2\lambda}}{\lambda-\mB_T}d\lambda
\frac{1}{2\pi\sqrt{-1}}\int_{\Delta'}\frac{e^{-u^2\lambda}}{\lambda-\mB_T}d\lambda.
\end{align}

From (\ref{e04099}) and Lemma \ref{e04094},  we get the following theorem.
\begin{thm}\label{e05040}
For $u_0>0$ fixed,
there exist $C, C'>0$ and $T_0\geq 1$ such that for $T\geq T_0$, $u\geq u_0$,
\begin{align}\label{e04101}
\|\exp(-u^2\mB_T)-P\exp(-u^2\mB_2)P\|\leq \frac{C}{T}\exp(-C'u^2).
\end{align}
\end{thm}

Let $\exp(-u^2\mB_T)(z,z')$, $P\exp(-u^2\mB_2)P(z,z')$
$(z,z'\in Z_b, b\in S)$ be the smooth kernels of the operators
$\exp(-u^2\mB_T)$, $P\exp(-u^2\mB_2)P$ calculated with respect to $dv_Z(z')$.

By using the proof of \cite[Theorems 5.22]{MR1765553} and the fact that $\ker D^Y=0$, we have

\begin{prop}\label{e04102}
(i) For $u_0>0$ fixed, for $m\in \N$, $ b\in S$, there exist $C, C'>0$, $T_0\geq 1$, such that for $z,z'\in Z_b$, $u\geq u_0$,
$T\geq T_0$,
\begin{align}\label{e04103}
\sup_{|\alpha|,|\alpha'|\leq m}\left|\frac{\partial^{|\alpha|+|\alpha'|}}{\partial z^{\alpha}\partial {z}^{'\alpha'}}
\exp(-u^2\mB_T)(z,z')\right|\leq C\exp(-C'u^2).
\end{align}

(ii) For $u_0>0$ fixed, for $m\in \N$, $ b\in S$, there exist $C, C'>0$, $T_0\geq 1$, such that for $z,z'\in Z_b$, $u\geq u_0$,
$T\geq T_0$,
\begin{align}\label{e04142}
\sup_{|\alpha|,|\alpha'|\leq m}\left|\frac{\partial^{|\alpha|+|\alpha'|}}{\partial z^{\alpha}\partial {z}^{'\alpha'}}
P\exp(-u^2\mB_2)P(z,z')\right|\leq C\exp(-C'u^2).
\end{align}
\end{prop}

The complete proof of Proposition \ref{e04102} is left to the next subsection.

From Proposition \ref{e04102} i), we obtain Theorem \ref{e03009} ii).

Let $\mathrm{inj}^Z$ be the injectivity radius of ($Z_b, g^{TZ_b}$).
For $(g^{-1}z,z)\in Z_b\times Z_b$, we will identify $B^{T_{g^{-1}z}Z_b}(0,\var)\times B^{T_{z}Z_b}(0,\var)$
 with $B^{Z_b}(g^{-1}z,\var)\times B^{Z_b}(z,\var)$
by the canonical exponential map when $\var < \mathrm{inj}^Z$.

Let $\phi:\mathbb{R}^n\rightarrow [0,1]$
be a smooth function with compact support in $B(0,\mathrm{inj}^Z/2)$, equal 1 near 0 such that
$\int_{\mathbb{R}^n}\phi(W)dv(W)=1$.
 Take $v\in(0,1]$.
By Taylor expansion and Proposition \ref{e04102}, there exists $c>0$, such that
\begin{align}\label{e04145}
\begin{split}
|(\exp(-u^2\mB_T)-P\exp(-u^2\mB_2)P)&(vW,vW')
\\
-(\exp(-u^2\mB_T)-&P\exp(-u^2\mB_2)P)(0,0)|
\leq cv \exp(-C'u^2)
\end{split}
\end{align}
for $|W|, |W'|$ are sufficiently small. Then for $U,U'\in \mathbb{E}_0$,
\begin{align}\label{e04146}
\begin{split}
|\la&(\exp(-u^2\mB_T)-P\exp(-u^2\mB_2)P)(0,0)U,U'\ra_{0}
\\
&-\int_{\mathbb{R}^n\times \mathbb{R}^n}\la(\exp(-u^2\mB_T)-P\exp(-u^2\mB_2)P)(vW,vW')U,U'\ra_{0}
\\
&\quad\times \phi(W)\phi(W')dv(W)dv(W')|
 \leq cv\|U\|_0\|U'\|_0 \exp(-C'u^2).
\end{split}
\end{align}

On the other hand, %set $\phi''(W)=\phi'(W/v)$.
By Theorem \ref{e05040},
\begin{align}\label{e04147}
\begin{split}
&\left|\int_{\mathbb{R}^n\times \mathbb{R}^n}\la(\exp(-u^2\mB_T)-P\exp(-u^2\mB_2)P)(vW,vW')U,U'\ra_{0}\right.
\\
&\quad \quad \quad \times \left.\phi(W)\phi(W')dv(W)dv(W')\right|
\\
\leq & \frac{c}{Tv^{n}}\|U\|_0\|U'\|_0 \exp(-C'u^2).
\end{split}
\end{align}
Take $v=T^{-\frac{1}{n+1}}$. From (\ref{e04146}) and (\ref{e04147}), we get
\begin{align}\label{e04148}
\begin{split}
|(\exp(-u^2\mB_T)-P\exp(-u^2\mB_2)P)(0,0)|
\leq c\,T^{-\frac{1}{n+1}} \exp(-C'u^2).
\end{split}
\end{align}
Therefore, we can get the following theorem.
\begin{thm}\label{e04223}
For $u_0>0$ fixed,
there exist $C, C'>0$, $T_0\geq 1$, $\delta>0$, such that for $u\geq u_0$,
$T\geq T_0$,
\begin{align}\label{e04149}
\left|\psi_S \delta_{u^2}\widetilde{\tr}[g\exp(-u^2\mB_T)]-\psi_S \delta_{u^2}\widetilde{\tr}[g\exp(-u^2\mB_2)]\right|
\leq \frac{C}{T^{\delta}} \exp(-C'u^2).
\end{align}
\end{thm}
By (\ref{e04012}) and (\ref{e04026}), we can get Theorem \ref{e03009} i) by taking the coefficients of $du$ in (\ref{e04149}).
From the dominated convergence theorem, we get Theorem \ref{e03009} iii) from Theorem \ref{e03009} i) and (\ref{e04149}).

The proof of Theorem \ref{e03009} is complete.

\subsection{Proof of Theorem \ref{e04102}}\label{s0404}

Recall that we assume that $S$ is compact for simplicity in Section \ref{s0302}. There exists a family of $\cC^{\infty}$
sections of $TY$ (resp. $TX$), $U_1,\cdots, U_r$ (resp. $U_1',\cdots, U_{r'}'$), such that for any $y\in V$ (resp. $x\in W$),
$U_1(y),\cdots, U_r(y)$ (resp. $U_1'(x),\cdots, U_{r'}(x)$) span $T_yY$ (resp. $T_xX$).

\begin{defn}\label{e04105}%{e05011}
Let $\mD$ be a family of operators on $\mathbb{E}_0$,
\begin{align}\label{e04106}%{e05012}
\mD=\left\{P\,^0\nabla^{\mS_Z\otimes E}_{U_{p,1}^H}P+\ P^\bot\, ^0\nabla^{\mS_Z\otimes E}_{U_{p,1}^H}P^\bot,\ P^\bot \, ^0\nabla^{\mS_Z\otimes E}_{U_i'}P^\bot\right\}.
\end{align}
\end{defn}

Note that in \cite[(5.60)]{MR1765553}, the corresponding set of operators is stated as
$\{p_T\,^0\nabla^{\Lambda (T^{*(0,1)}Z)\otimes \xi}_{U_{l,1}^H}p_T,$ $
\ p_T^\bot\,^0\nabla^{\Lambda (T^{*(0,1)}Z)\otimes \xi}_{U_{l,1}^H}p_T^\bot,
\ p_T^\bot\,^0\nabla^{\Lambda (T^{*(0,1)}Z)\otimes \xi}_{U_i'}p_T^\bot\}$.
We need to read \cite[(5.60)]{MR1765553} as $\mD_T=\\ \{p_T\,^0\nabla^{\Lambda (T^{*(0,1)}Z)\otimes \xi}_{U_{l,1}^H}p_T+
\ p_T^\bot\,^0\nabla^{\Lambda (T^{*(0,1)}Z)\otimes \xi}_{U_{l,1}^H}p_T^\bot,
\ p_T^\bot\,^0\nabla^{\Lambda (T^{*(0,1)}Z)\otimes \xi}_{U_i'}p_T^\bot\}$.
In this way, the corresponding commutator $[Q_1, [Q_2,\cdots[Q_k, A_T^2],\cdots]]$ has the same structure as $A_T^2$
(see the following proof of Lemma \ref{e04107}).

\begin{lemma}\label{e04107}%{e05013}
For any $k\in \mathbb{N}$ fixed, there exists $C_k>0$, $T_0\geq 1$ such that for $T\geq T_0$, $Q_1,\cdots, Q_k\in\mD$
and $s,s'\in \mathbb{E}_0^2$, we have
\begin{align}\label{e04108}%{e05014}
|\langle [Q_1, [Q_2,\cdots[Q_k, \mB_T],\cdots]]s,s'\ra_{0}|\leq C_k|s|_{T,1}|s'|_{T,1}.
\end{align}
\end{lemma}
\begin{proof}
Set
$
\mathscr{S}$ be the set of
uniformly
 bounded operators along the fiber $X$ with smooth kernel.
Set
\begin{align}\label{e04011}
\begin{split}
&\Theta_1=\left\{a_{ij}\,^0\nabla^{\mS_Z\otimes E}_{U_i'}\,^0\nabla^{\mS_Z\otimes E}_{U_j'}+b\,:\, a_{ij}\in \cC^{\infty}(W, C(TZ)), b\in \mathscr{S}\right\},
\\
&\Theta_2=\left\{a_{i}\,^0\nabla^{\mS_Z\otimes E}_{U_i'}+b\,:\, a_{i}\in \cC^{\infty}(W, C(TZ)), b\in \mathscr{S}\right\},
\\
&\Theta_3=\left\{b_{pq}\,^0\nabla^{\mS_Z\otimes E}_{U_p}\,^0\nabla^{\mS_Z\otimes E}_{U_q}+
b_{p}\,^0\nabla^{\mS_Z\otimes E}_{U_p}+a_{i}\,^0\nabla^{\mS_Z\otimes E}_{U_i'}+b\,:\, a_{i}\in \cC^{\infty}(W, C(TZ)), \right.
\\
&\quad\quad\quad\quad\quad\quad\quad\quad\quad\quad\quad\quad\quad\quad\quad\quad\quad\quad\quad\quad\quad\quad\quad
b_{pq}, b_p, b\in \mathscr{S}\}.
\end{split}
\end{align}

By (\ref{e04072}), (\ref{e04007}), (\ref{e04017}), (\ref{e04015}) and (\ref{e04073}), we can split the operator $\mB_T$
such that
\begin{align}\label{e04041}
\mB_T=T^2P^{\bot}A_1P^{\bot}+T(P^{\bot}A_2P^{\bot}+PA_2'P^{\bot}+P^{\bot}A_2'P)+A_3,
\end{align}
where $A_1\in \Theta_1$, $A_2, A_2'\in \Theta_2$, $A_3\in \Theta_3$.

First, we consider the case when $k=1$.

a) The case where $Q=P\,^0\nabla^{\mS_Z\otimes E}_{U_{p,1}^H}P+\ P^\bot\, ^0\nabla^{\mS_Z\otimes E}_{U_{p,1}^H}P^\bot$.

We observe that
if $b\in \mathscr{S}$, so are $\left[\,^0\nabla^{\mS_Z\otimes E}_{U_{p,1}^H}, b\right]$,
$\,^0\nabla^{\mS_Z\otimes E}_{U_i'}\,b$ and $b\,\,^0\nabla^{\mS_Z\otimes E}_{U_i'}$.
Then we have
\begin{align}\label{e04042}
\begin{split}
[Q, P^{\bot}A_1P^{\bot}]&=P^{\bot}\left(\left[\,^0\nabla^{\mS_Z\otimes E}_{U_{p,1}^H}, A_1\right]-\left[\,^0\nabla^{\mS_Z\otimes E}_{U_{p,1}^H},
P\right]A_1-A_1\left[\,^0\nabla^{\mS_Z\otimes E}_{U_{p,1}^H}, P\right]\right)P^{\bot},
\\
[Q, P^{\bot}A_2P^{\bot}]&=P^{\bot}\left(\left[\,^0\nabla^{\mS_Z\otimes E}_{U_{p,1}^H}, A_2\right]-\left[\,^0\nabla^{\mS_Z\otimes E}_{U_{p,1}^H},
P\right]A_2-A_2\left[\,^0\nabla^{\mS_Z\otimes E}_{U_{p,1}^H}, P\right]\right)P^{\bot},
\\
[Q, PA_2'P^{\bot}]&=P\left(\left[\,^0\nabla^{\mS_Z\otimes E}_{U_{p,1}^H}, A_2'\right]+\left[\,^0\nabla^{\mS_Z\otimes E}_{U_{p,1}^H},
P\right]A_2'-A_2'\left[\,^0\nabla^{\mS_Z\otimes E}_{U_{p,1}^H}, P\right]\right)P^{\bot},
\\
[Q, P^{\bot}A_2'P]&=P^{\bot}\left(\left[\,^0\nabla^{\mS_Z\otimes E}_{U_{p,1}^H}, A_2'\right]-\left[\,^0\nabla^{\mS_Z\otimes E}_{U_{p,1}^H},
P\right]A_2'+A_2'\left[\,^0\nabla^{\mS_Z\otimes E}_{U_{p,1}^H}, P\right]\right)P,
\end{split}
\end{align}
and
$\left[\,^0\nabla^{\mS_Z\otimes E}_{U_{p,1}^H}, A_i\right]\in \Theta_i, A_i\left[\,^0\nabla^{\mS_Z\otimes E}_{U_{p,1}^H}, P\right]\in \Theta_i$,
$\left[\,^0\nabla^{\mS_Z\otimes E}_{U_{p,1}^H}, A_2'\right]\in \Theta_2, A_2'\left[\,^0\nabla^{\mS_Z\otimes E}_{U_{p,1}^H}, P\right]\in \Theta_2$
for $i=1,2,3$.
For the element in $\Theta_3$, since the principal symbol of $Q$ is identity, we have
$[Q, A_3]\in \Theta_3$.

So $[Q,\mB_T]$ has the same structure as $\mB_T$ in (\ref{e04041}).
Thus there exists $C>0$, $T_0\geq 1$ such that for $T\geq T_0$, $s,s'\in \mathbb{E}_0^2$, we have
\begin{align}\label{e04248}
|\langle [Q, \mB_T]s,s'\ra_{0}|\leq C|s|_{T,1}|s'|_{T,1}.
\end{align}

b) The case where $Q=P^{\bot}\,^0\nabla_{U_i'}^{\mS_Z\otimes E}P^{\bot}$.

As in (\ref{e04042}), we have
\begin{align}\label{e04100}
\begin{split}
[Q, P^{\bot}A_1P^{\bot}]&=P^{\bot}\left(\left[\,^0\nabla^{\mS_Z\otimes E}_{U_{i}'}, A_1\right]-\left(\,^0\nabla^{\mS_Z\otimes E}_{U_i'}
P\right)A_1+A_1\left(P\,^0\nabla^{\mS_Z\otimes E}_{U_i'}\right)\right)P^{\bot},
\\
[Q, P^{\bot}A_2P^{\bot}]&=P^{\bot}\left(\left[\,^0\nabla^{\mS_Z\otimes E}_{U_i'}, A_2\right]-\left(\,^0\nabla^{\mS_Z\otimes E}_{U_i'}
P\right)A_2+A_2\left(P\,^0\nabla^{\mS_Z\otimes E}_{U_i'}\right)\right)P^{\bot},
\\
[Q, PA_2'P^{\bot}]&=P\left(-\left[\,^0\nabla^{\mS_Z\otimes E}_{U_i'}, A_2'\right]
+\left(P\,^0\nabla^{\mS_Z\otimes E}_{U_i'}\right)A_2'-A_2'\left(P\,^0\nabla^{\mS_Z\otimes E}_{U_i'}\right)\right)P^{\bot},
\\
[Q, P^{\bot}A_2'P]&=P^{\bot}\left(\left[\,^0\nabla^{\mS_Z\otimes E}_{U_i'}, A_2'\right]
+A_2'\left(\,^0\nabla^{\mS_Z\otimes E}_{U_i'}
P\right)-\left(\,^0\nabla^{\mS_Z\otimes E}_{U_i'}
P\right)A_2'\right)P.
\end{split}
\end{align}
Since $[Q, A_3]\in \Theta_3$, we know that $[Q,\mB_T]$ has the same structure as $\mB_T$ in (\ref{e04041}).
Thus there exists $C>0$, $T_0\geq 1$ such that for $T\geq T_0$, $s,s'\in \mathbb{E}_0^2$, we have
\begin{align}\label{e04249}
|\langle [Q, \mB_T]s,s'\ra_{0}|\leq C|s|_{T,1}|s'|_{T,1}.
\end{align}

c) Higher order commutators

The estimate of higher order commutators are obtained inductively from a) and b).

The proof of Lemma \ref{e04107} is complete.
\end{proof}

For $k\in \mathbb{N}$, let $\mD^k$ be the family of operators $Q$ which can be written in the form
\begin{align}\label{e04121}%{e05027}
Q=Q_1\cdots Q_k,\quad Q_i\in \mD.
\end{align}
If $k\in \mathbb{N}$, we define the Hilbert norm $\|\cdot \|_{k}'$ by
\begin{align}\label{e04122}%{e05028}
\|s\|_{k}^{'2}=\sum_{\ell=0}^k\sum_{Q\in \mD^\ell}\|Qs\|^2_0.
\end{align}

Since $[\,^0\nabla_{f_{p,1}^H}^{\mS_Z\otimes E}, P]$, $P \,^0\nabla_{e_i}^{\mS_Z\otimes E}$ and $\,^0\nabla_{e_i}^{\mS_Z\otimes E} P$
are operators along the fiber $X$ with smooth kernels,
the sobolev norm  $\|\cdot\|_{k}'$ is equivalent to the canonical sobolev norm $\|\cdot\|_k$.% defined in Definition \ref{e04018}.

Thus, we also denote the Sobolev space with respect to $\|\cdot \|_{k}'$ by $\mathbb{E}_0^k$.

\begin{lemma}\label{e04126}%{e05033}
For any $m\in \mathbb{N}$, there exist $p_m\in\mathbb{N}$, $C_m>0$ and $T_0\geq 1$ such that for $T\geq T_0$,
$\lambda\in \Delta'$, $s\in \mathbb{E}_0^m$,
\begin{align}\label{e04127}%{e05034}
\|(\lambda-\mB_T)^{-1}s\|_{m+1}'\leq C_m(1+|\lambda|)^{p_m}\|s\|_{m}'.
\end{align}
\end{lemma}
\begin{proof}
Clearly for $T\geq 1$,
\begin{align}\label{e04128}%{e05035}
\|s\|_{1}'\leq C|s|_{T,1}.
\end{align}
When $m=0$, we obtain the lemma from (\ref{e04128}) and Lemma \ref{e04304}.

For the general case,
let $\cR_T$ be the family of operators
\begin{align}\label{e04129}%{e05036}
\cR_T=\{[Q_{i_1}, [Q_{i_2},\cdots[Q_{i_p}, \mB_T],\cdots]]\}£¬
\end{align}
where $Q_{i_1}, \cdots Q_{i_p}\in \mD$. We can express
\begin{align}\label{e04130}%{e05037}
Q_1\cdots Q_{k+1}(\lambda-\mB_T)^{-1}
\end{align}
as a linear combination of operators of the type
\begin{align}\label{e04131}%{e05038}
(\lambda-\mB_T)^{-1}\cR_1(\lambda-\mB_T)^{-1}\cR_2\cdots \cR_{k'}(\lambda-\mB_T)^{-1}Q_{k'+1}\cdots Q_{k+1},\quad k'\leq k,
\end{align}
with $\cR_1,\cdots,\cR_{k'}\in\cR_T$.
By Lemma \ref{e04107},
 we have
\begin{align}\label{e04132}%{e05039}
|\cR_is|_{T,-1}\leq |s|_{T,1}.
\end{align}
From (\ref{e04128}), (\ref{e04132}) and Lemma \ref{e04304}, we have
\begin{align}\label{e04133}%{e05040}
\begin{split}
&\|(\lambda-\mB_T)^{-1}s\|_{k+1}'\leq C\sum\|Q_2\cdots Q_{k+1}(\lambda-\mB_T)^{-1}s\|_1'
\\
\leq &C\sum \|(\lambda-\mB_T)^{-1}\cR_2(\lambda-\mB_T)^{-1}\cR_3\cdots \cR_{k'}(\lambda-\mB_T)^{-1}Q_{k'+1}\cdots Q_{k+1}s\|_1'
\\
\leq &C_k(1+|\lambda|)^{p_k}\sum \|Q_{k'+1}\cdots Q_{k+1}s\|_{0}
\\
\leq &C_k(1+|\lambda|)^{p_k}\|s\|_{k}'.
\end{split}
\end{align}

The proof of Lemma \ref{e04126} is complete.
\end{proof}

Now we can complete the proof of Theorem \ref{e04102}.

From (\ref{e05009}), for any $k\in\mathbb{N}^*$,
\begin{align}\label{e04134}%{e05045}
\exp(-u^2\mB_T)=\frac{1}{2\pi \sqrt{-1}}\int_{\Delta'}\frac{e^{-u^2\lambda}}{(\lambda-\mB_T)}d\lambda=\frac{(-1)^{k-1}(k-1)!}
{2\pi \sqrt{-1}u^{k-1}}
\int_{\Delta'}\frac{e^{-u^2\lambda}}{(\lambda-\mB_T)^k}d\lambda.
\end{align}
 By Lemma \ref{e04126}, there exist $C>0$, $r\in \mathbb{N}^*$, such that for any $m'$-order (resp. $m''$) fiberwise differential
 operator $R$ (resp. $R'$) along $Z$, $m', m''\geq n/2$,
choosing $k\geq m'+m''$,
\begin{align}\label{e04136}%{e05047}
\begin{split}
\|R(\lambda-\mB_T)^{-k}R's\|_{0}\leq C\|(\lambda-\mB_T)^{-k}R's\|_{m'}'\leq C(1+|\lambda|)^r\|s\|_{0}.
\end{split}
\end{align}
From (\ref{e04134}) and (\ref{e04136}), there exist $C, C'>0$, such that
\begin{align}\label{e04137}%{e05048}
\|R\exp(-u^2\mB_T)R's\|_0\leq C\exp(-C'u^2)\|s\|_0. %,\quad \|QF_u(\mB_T)Q's\|_0\leq C\exp(-C'u^2)\|s\|_0.
\end{align}

Now applying Sobolev embedding theorem, for $R''$ a fiberwise differential operator of order $m'-n/2$ along $Z$,
there exists $C>0$, such that for any $s\in \mathbb{E}_0$,
\begin{align}\label{e04138}%{e05049}
|R''\exp(-u^2\mB_T)R's|_{\cC^0}\leq C\exp(-C'u^2)\|s\|_0,
\end{align}
and
\begin{align}\label{e04139}%{e05050}
(R''\exp(-u^2\mB_T)R's)(z)=\int_Z (R'_{z'}R_z''\exp(-u^2\mB_T)(z,z'))s(z')dv_Z(z'),
\end{align}
here $R'_{z'}$ acts on $(\mS(TZ,L_Z)\times E)^*$ by identifying $(\mS(TZ,L_Z)\times E)^*$ to $\mS(TZ,L_Z)\otimes E$ by $h^{\mS_Z\otimes E}$. Thus, we have
\begin{align}\label{e04140}%{e05051}
\|R'_{\cdot}R_z''\exp(-u^2\mB_T)(z,\cdot)\|_{0}\leq C\exp(-C'u^2).
\end{align}
Applying the Sobolev embedding theorem to the $z'$-variable, from (\ref{e04140}),
we can get (\ref{e04103}).

From (\ref{e04061}), for any $m\in \mathbb{N}$, there exist $p_m\in\mathbb{N}$, $C_m>0$ and $T_0\geq 1$ such that for $T\geq T_0$,
$\lambda\in \Delta'$, $s\in \mathbb{E}_0^m$,
\begin{align}\label{e04127}%{e05034}
\|P(\lambda-\mB_2)^{-1}Ps\|_{m+1}'\leq C_m(1+|\lambda|)^{p_m}\|Ps\|_{m}'.
\end{align}
Following the same process, we get (\ref{e04142}).

\subsection{Proof of Proposition \ref{e03047}}\label{s0405}

Let $N_X$ be the number operator acting on $TZ$ such that for $s\in TZ$,
\begin{align}
N_XP^{TX}s=P^{TX}s,\quad N_XP^{T^HZ}s=0.
\end{align}
Let
\begin{align}
'\nabla^{TZ}_T=T^{-N_X}\nabla_T^{TZ}T^{N_X}.
\end{align}
Let $'R^{TZ}_T$ be the curvature of $'\nabla^{TZ}_T$. By (\ref{e04310}), we have
\begin{align}\label{e04156}
'\nabla^{TZ}_T=\,^0\nabla^{TZ}+\frac{1}{T}(P^{TX}S_1P^{T^HZ}+P^{T^HZ}S_1P^{TX})+\frac{1}{T^2}P^{T^HZ}S_1P^{T^HZ}.
\end{align}
Then by (\ref{e03005}), we have
\begin{align}
\gamma_{\mA}(T)
=\left.\frac{\partial}{\partial b}\right|_{b=0}\widehat{\mathrm{A}}_g\left('R^{TZ}_T+b\frac{\partial '\nabla_T^{TZ}}{\partial T}\right).
\end{align}
From (\ref{e04156}), we have
\begin{align}
\frac{\partial '\nabla_T^{TZ}}{\partial T}=O\left(\frac{1}{T^2}\right) \quad \text{and} \quad 'R_T^{TZ}=O(1).
\end{align}
Then Proposition \ref{e03047} follows from $'\nabla_{\infty}^{TZ}=\,^0\nabla^{TZ}$.

\section{Proof of Theorem \ref{e03018} and Theorem \ref{e03022} i)}\label{s05}

In this Section, we use the notations in Section \ref{s04} and assumptions in Section \ref{s0202}.

Set
\begin{align}\label{e06003}
\mB_T'=B_{3,T}^2+dT\wedge \frac{\partial B_{3,T}}{\partial T}.
\end{align}
By (\ref{e04007}), we have
\begin{multline}\label{e06040}
\frac{\partial B_{3,T}}{\partial T}=D^X-\frac{1}{8T^2}\left(\la [f_{p,1}^H,f_{q,1}^H], e_i\ra c(e_i)c(f_{p,1}^H)c(f_{q,1}^H)\right.
\\
+\left.\la [g_{\alpha,3}^H,g_{\beta,3}^H], e_i\ra c(e_i)g^{\alpha}\wedge g^{\beta}\wedge+4\la S_1(g_{\alpha,3}^H)e_i, f_{p,1}^H\ra
c(e_i)c(f_{p,1}^H)g^{\alpha}\wedge \right).
\end{multline}
By Definition \ref{e03001}, we have
\begin{align}\label{e06041}
\beta_g^T(T,u)=\left\{\psi_S\delta_{u^2}\widetilde{\tr}[g\exp(-u^2\mB_{T}')]\right\}^{dT}.
\end{align}

Recall that $B_2$ is the Bismut superconnection in (\ref{e04003}).
Comparing with (\ref{e04303}), by Lemma \ref{e04013}, we have
\begin{align}\label{e06072}
P\mB_T'P=B_{2}+O\left(\frac{1}{T}\right).
\end{align}

By (\ref{e06072}), if we replace $\mB_T$ to $\mB_T'$ and $\mB_2$ to $B_2$, then everything in Section \ref{s04} works well.
As an analogue of Theorem \ref{e04223}, we can get the following theorem.
\begin{thm}\label{e05041}
For $u_0>0$ fixed, there exist $C, C'>0$, $T_0\geq 1$, $\delta>0$, such that for $u\geq u_0$,
$T\geq T_0$,
\begin{align}\label{e04149}
\left|\psi_S \delta_{u^2}\widetilde{\tr}[g\exp(-u^2\mB_T')]-\psi_S \delta_{u^2}\widetilde{\tr}[g\exp(-u^2 B_2)]\right|
\leq \frac{C}{T^{\delta}} \exp(-C'u^2).
\end{align}
\end{thm}

Take $s>0$. By replacing $T$ to $sT$ in Theorem \ref{e05041} and taking the coefficient of $ds$, for
$sT\geq T_0$, we have
\begin{align}\label{e05050}
\left|\left\{\psi_S \delta_{u^2}\widetilde{\tr}[g\exp(-u^2\mB_{sT}')]\right\}^{ds}\right|
\leq \frac{C}{(sT)^{\delta}} \exp(-C'u^2).
\end{align}
By (\ref{e06041}), for %$u>0$,
$T\geq T_0$, we have
\begin{multline}\label{e05042}
\beta_g^T(T,u)=\left.\left\{\psi_S\delta_{u^2}\widetilde{\tr}[g\exp(-u^2\mB_{sT}')]\right\}^{d(sT)}\right|_{s=1}
\\
=\left.T^{-1}\cdot\left\{\psi_S\delta_{u^2}\widetilde{\tr}[g\exp(-u^2\mB_{sT}')]\right\}^{ds}\right|_{s=1}.
\end{multline}
From (\ref{e05050}) and (\ref{e05042}), for $u_0>0$ fixed,
there exist $C, C'>0$, $T_0\geq 1$, $\delta>0$, such that
for $u\geq u_0$,
$T\geq T_0$, we have
\begin{align}\label{e05043}
\left|\beta_g^T(T,u)\right|
\leq \frac{C}{T^{1+\delta}} \exp(-C'u^2).
\end{align}

Then we get Theorem \ref{e03018}  and Theorem \ref{e03022} i).

\section{Proof of Theorem \ref{e03022} ii)}\label{s06}

In this section, we use the notations in Section 2.2, 4, 5 and assumptions in Section 2.2.

In the first three subsections, we prove Theorem \ref{e03022} ii)
when $\dim Y$ and $\dim Z$ are all even. In Section 6.4, we discuss the other cases. In Section 6.5, we prove the technical result
Theorem \ref{e107}.

\subsection{The proof is local on $\pi_1^{-1}(V^g)$}\label{s0601}

Recall that $\mB_T'$ is the operator defined in (\ref{e06003}).
As in (\ref{e04301}), we set
\begin{align}\label{e06004}
\mB_{\var, T/\var}'=\var^2 \delta_{\var^2}\mB_{T/\var}'\delta_{\var^2}^{-1}=
B_{3,\var^2,T/\var}^2+\var^{-1}dT\wedge \left.\frac{\partial B_{3,\varepsilon^2,T'}}{\partial T'}\right|_{T'=T\varepsilon^{-1}}.
\end{align}
By Definition \ref{e03001}, we have
\begin{align}\label{e06005}
\var^{-1}\beta_g^T(T/\var,\var)=\left\{\psi_S\widetilde{\tr}[g\exp(-\mB_{\var, T/\var}')]\right\}^{dT}.
\end{align}
Precisely, by (\ref{e04008}), we have
\begin{multline}\label{e06001}
B_{3,\var^2,T/\var}=TD^X+\var D^H+\frac{\var^2}{8T}\la [f_{p,1}^H, f_{q,1}^H],e_i\ra c(e_i)c(f_{p,1}^H)c(f_{q,1}^H)
\\
+\,^0\nabla^{\cE_Z,u}-\frac{c(T_2)}{4\var}
+\frac{\var}{2T}\la S_1(g_{\alpha,3}^H)e_i, f_{p,1}^H \ra c(e_i)c(f_{p,1}^H) g^{\alpha}_3\wedge
\\
+\frac{1}{8T}\la [g_{\alpha,3}^H, g_{\beta,3}^H],e_i\ra c(e_i)g^{\alpha}\wedge g^{\beta}\wedge,
\end{multline}
and
\begin{multline}\label{e06002}
\var^{-1} \left.\frac{\partial B_{3,\varepsilon^2,T'}}{\partial T'}\right|_{T'=T\varepsilon^{-1}}
=D^X-\frac{1}{8T^2}(\la \var^2[f_{p,1}^H, f_{q,1}^H],e_i\ra c(e_i)c(f_{p,1}^H)c(f_{q,1}^H)
\\
+4\var \la S_1(g_{\alpha,3}^H)e_i, f_{p,1}^H \ra c(e_i)c(f_{p,1}^H) g^{\alpha}_3\wedge
+\la [g_{\alpha,3}^H, g_{\beta,3}^H],e_i\ra c(e_i)g^{\alpha}\wedge g^{\beta}\wedge).
\end{multline}

Set $B_1|_{V^g}$ be the Bismut superconnection associated to ($T_1^H(W|_{V^g}), g^{TX}, h^{L_X},$ $\nabla^{L_X}, h^E, \nabla^E$).
For $t>0$, we denote $\delta_t^V$ the operator on $\Lambda^i(T^*V^g)$ by multiplying by $t^{-i/2}$.
As in (\ref{e01041}), set
\begin{align}\label{e06050}%\label{e01041}
B_{1,T^2}|_{V^g}=T\delta_{T^2}^V\circ B_{1}|_{V^g}\circ (\delta_{T^2}^V)^{-1}.
\end{align}
As in (\ref{e04301}), we set
\begin{align}\label{e05013}
\mB_{T^2}''|_{V^g}=(B_{1,T^2}|_{V^g})^2+dT\wedge \left.\frac{\partial B_{1,T^2}}{\partial T}\right|_{V^g}.
\end{align}
Then by (\ref{e03043}), we have
\begin{align}\label{e06051}
\gamma_1(T)=\left\{\psi_{V^g}\widetilde{\tr}[g\exp(-\mB_{ T^2}''|_{V^g})]\right\}^{dT}.
\end{align}

In the first three subsections we assume that $\dim Y=m$ and $\dim Z=n$ are all even.

Let $d^V$, $d^W$ be the distance functions on $V$, $W$ associated to $g^{TV}$, $g^{TW}$.
Let $\mathrm{Inj}^{V}$, $\mathrm{Inj}^{W}$ be the injective radius of $V$, $W$.
In the sequel, we assume that given $0<\alpha<\alpha_0<\inf\{\mathrm{Inj}^V,\mathrm{Inj}^W\}$
are chosen small enough so that if $y\in V$, $d^V(g^{-1}y,y)\leq \alpha$,
then $d^V(y,V^g)\leq\frac{1}{4}\alpha_0$, and if $z\in W$, $d^W(g^{-1}z,z)\leq \alpha$,
then $d^W(z,W^g)\leq\frac{1}{4}\alpha_0$. %If $y\in V$, let $B^V(y,\alpha_0)$ be the open ball
%of center $y$ and radius $\alpha_0$ in $V$.

Let $f$ be a smooth even function defined on $\R$ with values in $[0,1]$,
such that
\begin{align}\label{e06006}
f(t)=
\left\{
  \begin{array}{ll}
    1, & \hbox{$|t|\leq \alpha/2$;} \\
    0, & \hbox{$|t|\geq \alpha$.}
  \end{array}
\right.
\end{align}
For $t\in (0,1]$, $a\in\C$, set
\begin{align}\label{e06007}
\left\{
  \begin{aligned}%{ll}
    &\mathbf{F}_t(a)=\int_{-\infty}^{+\infty}\cos(\sqrt{2}va)e^{-\frac{v^2}{2}}f(\sqrt{t}v)\frac{dv}{\sqrt{2\pi}},\\
    &\mathbf{G}_t(a)=\int_{-\infty}^{+\infty}\cos(\sqrt{2}va)e^{-\frac{v^2}{2}}(1-f(\sqrt{t}v))\frac{dv}{\sqrt{2\pi}}.
  \end{aligned}
\right.
\end{align}
Clearly,
\begin{align}\label{e06053}
\mathbf{F}_t(a)+\mathbf{G}_t(a)=\exp(-a^2).
\end{align}

The functions $\mathbf{F}_t(a)$ and $\mathbf{G}_t(a)$ are even holomorphic functions and the restrictions of $\mathbf{F}_t(a)$,
$\mathbf{G}_t(a)$ to $\R$ lie in the Schwartz space.
So there exist holomorphic functions $\widetilde{\mathbf{F}}_t(a)$ and $\widetilde{\mathbf{G}}_t(a)$ on $\C$ such that
\begin{align}\label{e06008}
\mathbf{F}_t(a)=\widetilde{\mathbf{F}}_t(a^2),\quad \mathbf{G}_t(a)=\widetilde{\mathbf{G}}_t(a^2).
\end{align}
From (\ref{e06053}), we deduce that
\begin{align}\label{e06009}
\exp(-\mB_{\var, T/\var}')=\widetilde{\mathbf{F}}_{\var^2}(\mB_{\var, T/\var}')+
\widetilde{\mathbf{G}}_{\var^2}(\mB_{\var, T/\var}').
\end{align}

Fix $b\in S$. For $z,z'\in Z_b$, let $\widetilde{\mathbf{F}}_{\var^2}(\mB_{\var, T/\var}')(z,z')$
and $\widetilde{\mathbf{G}}_{\var^2}(\mB_{\var, T/\var}')(z,z')$ be the smooth kernels associated to
$\widetilde{\mathbf{F}}_{\var^2}(\mB_{\var, T/\var}')$ and $\widetilde{\mathbf{G}}_{\var^2}(\mB_{\var, T/\var}')$
with respect to the volume form $dv_Z(z')$.

\begin{lemma}\label{e06010}
For $\delta>0$ fixed , there exist $C_1,C_2>0$, such that for any $z,z'\in Z_b$, $0<\var\leq \delta$, $T\geq 1$,
\begin{align}\label{e06011}
\left|\widetilde{\mathbf{G}}_{\frac{\var^2}{T^2}}(\mB_{\frac{\var}{T},T}')(z,z')\right|\leq C_1\exp\left(-\frac{C_2T^2}{\varepsilon^2}\right).
\end{align}
In particular,
\begin{align}\label{e06012}
\left|\psi_S\tr_s\left[g\widetilde{\mathbf{G}}_{\frac{\var^2}{T^2}}
(\mB_{\frac{\var}{T},T}')\right]\right|\leq C_1\exp\left(-\frac{C_2T^2}{\varepsilon^2}\right).
\end{align}
\end{lemma}
\begin{proof}
By (\ref{e04007}), (\ref{e06003}) and the elliptic estimate, there exists $C>0$ such that for any $T\geq 1$,
\begin{align}\label{e06021}
\|s\|_2\leq C\|\mB_{T}'s\|_0+CT^2\|s\|_0.
\end{align}
Then for a $m$-order fiberwise differential operator $Q$ along $Z$ with scalar principal symbol, by (\ref{e06021}), we have
\begin{align}\label{e65}
\|Qs\|_2\leq C\|\mB_{T}'Qs\|_0+CT^2\|Qs\|_0
\leq C\|Q\mB_{T}'s\|_0+CT^2\|Qs\|_0+C\|[\mB_{T}',Q]s\|_0.
\end{align}
By (\ref{e04007}) and (\ref{e06003}), we have
\begin{align}\label{e65}
\|[\mB_{T}',Q]s\|_0\leq CT^2\|s\|_{m+1}.
\end{align}
Thus we get the estimate
\begin{align}\label{e66}
\|s\|_{m+2}\leq C\|\mB_{T}'s\|_{m}+CT^2\|s\|_{m+1}\leq CT^2(\|\mB_{T}'s\|_{m+1}
+\|s\|_{m+1}).
\end{align}
By induction, there exist $c_k>0$ for $0\leq k\leq m$, such that
\begin{align}\label{e67}
\|s\|_{m}\leq T^{2m}\sum_{k=0}^m c_k\|(\mB_{T}')^ks\|_0.
\end{align}
Let $\mB_{T}'^*$ be the adjoint of $\mB_{T}'$.
Similarly, we have
\begin{align}\label{e455}
\|s\|_{m}\leq T^{2m}\sum_{k=0}^m c_k\|(\mB_{T}'^*)^ks\|_0.
\end{align}
For $m$-order fiberwise differential operator $Q$,
for $m'\in \N$, by (\ref{e67}) and (\ref{e455}), we have
\begin{align}\label{e69}
\begin{split}
&\left|\left\la (\mB_{T}')^{m'}\widetilde{\mathbf{G}}_{\frac{\var^2}{T^2}}\left(\frac{\var^2}{T^2}\mB_{T}'\right)Qs,s'\right\ra\right|
\\
=&
\left|\left\la
s,Q^*\widetilde{\mathbf{G}}_{\frac{\var^2}{T^2}}\left(\frac{\var^2}{T^2}\mB_{T}'^*\right)(\mB_{T}'^*)^{m'}s'\right\ra\right|
\leq \left\|\widetilde{\mathbf{G}}_{\frac{\var^2}{T^2}}\left(\frac{\var^2}{T^2}\mB_{T}'^*\right)(\mB_{T}'^*)^{m'}s'\right\|_m\|s\|_0
\\
\leq & \left(T^m\sum_{k=0}^{m}c_k\left\|(\mB_{T}'^*)^k
\widetilde{\mathbf{G}}_{\frac{\var^2}{T^2}}\left(\frac{\var^2}{T^2}\mB_{T}'^*\right)(\mB_{T}'^*)^{m'}s'\right\|_0\right)\|s\|_0.
\end{split}
\end{align}
By \cite[(11.18)]{Bismut1997}, for $m\in \N$, there exist $c_m'>0$ and $c>0$, such that for any $0< \var\leq \delta$, $T\geq 1$,
\begin{align}\label{e68}
\sup_{\lambda\in \Delta}|\lambda|^m\left|\widetilde{\mathbf{G}}_{\frac{\var^2}{T^2}}\left(\frac{\var^2}{T^2}\lambda\right)\right|
\leq c_m'\exp\left(-\frac{cT^2}{\var^2}\right).
\end{align}
From (\ref{e69}) and (\ref{e68}), there exists $c_{m,m'}>0$, such that
\begin{align}\label{e70}
\left\|(\mB_{T}')^{m'}\widetilde{\mathbf{G}}_{\frac{\var^2}{T^2}}\left(\frac{\var^2}{T^2}\mB_{T}'\right)Q\right\|_0\leq
c_{m,m'}\exp\left(-\frac{cT^2}{2\var^2}\right).
\end{align}

Let $P$ be a fiberwise differential operators along $Z$ of order $m'$. Then by (\ref{e67}) and (\ref{e70}),
there exists $c_{m,m'}'>0$, such that for any $0<\var\leq \delta$, $T\geq 1$,
\begin{align}\label{e71}
\begin{split}
\left\|P\widetilde{\mathbf{G}}_{\frac{\var^2}{T^2}}\left(\frac{\var^2}{T^2}\mB_{T}'\right)Q\right\|_0\leq
 \left\|\widetilde{\mathbf{G}}_{\frac{\var^2}{T^2}}
\left(\frac{\var^2}{T^2}\mB_{T}'\right)Q\right\|_{m'}
\leq c_{m,m'}'\exp\left(-\frac{cT^2}{2\var^2}\right).
\end{split}
\end{align}

Following the same process in
(\ref{e04138})-(\ref{e04140}), there exist $C_1,C_2>0$, such that for any $z,z'\in Z_b$, $0<\var\leq \delta$, $T\geq 1$,
\begin{align}
\left|\widetilde{\mathbf{G}}_{\frac{\var^2}{T^2}}\left(\frac{\var^2}{T^2}\mB_{T}'\right)(z,z')\right|\leq
C_1\exp\left(-\frac{C_2T^2}{\varepsilon^2}\right).
\end{align}
Since $\mB_{\frac{\var}{T},T}'=\frac{\var^2}{T^2}\delta_{\frac{\var^2}{T^2}}\mB_T'\delta_{\frac{\var^2}{T^2}}^{-1}$,
we get the proof of Lemma \ref{e06010}.
\end{proof}

Using Lemma \ref{e06010} with $\var=T$ and $T$ replace by $T/\var$, for $T\geq 1$ fixed, we find
\begin{align}\label{e06033}
\begin{split}
&\left|\widetilde{\mathbf{G}}_{\var^2}(\mB_{\var,T/\var}')(z,z')\right|\leq C_1\exp\left(-\frac{C_2}{\varepsilon^2}\right),
\\
&\left|\psi_S\tr_s\left[g\widetilde{\mathbf{G}}_{\var^2}
(\mB_{\var,T/\var}')\right]\right|\leq C_1\exp\left(-\frac{C_2}{\varepsilon^2}\right).
\end{split}
\end{align}

From (\ref{e06009}) and (\ref{e06033}), by the finite propagation speed
for the solution of the hyperbolic equations for $\cos\left(s\sqrt{\mB_{\var, T/\var}'}\right)$
(cf. \cite[\S 7.8]{MR598467} and \cite[\S 4.4]{MR618463}),
it is clear that for $0<\var\leq 1$, $T\geq 1$, $z, z'\in Z_b$, if $d^V(\pi_1z, \pi_1z')\geq \alpha$, then
\begin{align}\label{e06054}
\widetilde{\mathbf{F}}_{\var^2}(\mB_{\var, T/\var}')(z,z')=0,
\end{align}
and moreover, given $z\in Z_b$, $\widetilde{\mathbf{F}}_{\var^2}(\mB_{\var, T/\var}')(z,\cdot)$ only depends on the restriction of
$\mB_{\var, T/\var}'$ to $\pi_1^{-1}(B^Y(\pi_1z, \alpha))$.

Let $\mU_{\alpha_0}(Y_b^g)$ be the set of $y\in Y_b$ such that $d^Y(y,Y_b^g)<\alpha_0/4$.
We identify $\mU_{\alpha_0}(Y_b^g)$ to $\{(y,U)\,:\, y\in Y_b^g, U\in N_{Y^g/Y}, |U|< \alpha_0/4 \}$
by using geodesic coordinates normal to $Y^g$ in $Y$, where $N_{Y^g/Y}$ is the real normal bundle associated to $g\in G$ in $Y$.
%So for $z\in \pi_1^{-1}(\mU_{\alpha_0}(Y_b^g))$, we can identify it to an element $(y,U,x)\in Y_b^g\times B^{N_Y}(0,\alpha_0)\times X_b$.
Let $dv_{Y^g}$ and $dv_{N_Y}$ be the corresponding volume forms on $TY^g$ and $N_{Y}$ induced by $g^{TY}$.
Then there exists the function $k_Y$ on $\mU_{\alpha_0}(Y_b^g)$, such that
\begin{align}\label{e06013}
dv_Z(z)=k_Y(y,U)dv_{Y^g}(y)dv_{N_Y}(U)dv_X(x).
\end{align}
Thus, from (\ref{e06054}),
\begin{align}\label{e06014}
\begin{split}
&\tr_s\left[g\widetilde{\mathbf{F}}_{\var^2}(\mB_{\var,T/\var}')\right]
\\
=&\int_{Z} \tr_s
\left[g\widetilde{\mathbf{F}}_{\var^2}(\mB_{\var,T/\var}')(g^{-1}z,z)\right]dv_Z(z)
\\
=&\int_{Y^g}\int_{U\in N, |U|< \alpha_0/4}\int_X\tr_s
\left[g \widetilde{\mathbf{F}}_{\var^2}(\mB_{\var,T/\var}')(g^{-1}(y,U,x),(y,U,x))\right]
\\
&\cdot k_Y(y,U)dv_{Y^g}(y)dv_{N_Y}(U)dv_X(x).
\end{split}
\end{align}
Therefore, from  (\ref{e06005}), (\ref{e06033}) and (\ref{e06014}),
we see that the proof of Theorem \ref{e03022} ii) is local near $\pi_1^{-1}(V^g)$.

\subsection{Rescaling of the variable $U$ and of the Clifford variables}\label{s0502}

Let $S_{3,T}$ be the tensor defined in (\ref{e01017}) associated to ($T_3^HW, g_T^{TZ}$). We can calculate that
\begin{align}\label{e05001}
\begin{split}
&\la S_{3,T}(Te_i)Te_j, g_{\alpha,3}^H\ra=\la S_{3}(e_i)e_j, g_{\alpha,3}^H\ra,
\\
&\la S_{3,T}(Te_i)f_{p,1}^H, g_{\alpha,3}^H\ra=\frac{1}{T}\la S_{3}(e_i)f_{p,1}^H, g_{\alpha,3}^H\ra,
\\
&\la S_{3,T}(Te_i)g_{\alpha,3}^H, g_{\beta,3}^H\ra
=\frac{1}{T}\la S_{3}(e_i)g_{\alpha,3}^H, g_{\beta,3}^H\ra,
\\
&\la S_{3,T}(f_{p,1}^H)f_{q,1}^H, g_{\alpha,3}^H\ra
=\la S_{3}(f_{p,1}^H)f_{q,1}^H, g_{\alpha,3}^H\ra,
\\
&\la S_{3,T}(f_{p,1}^H)g_{\alpha,3}^H, g_{\beta,3}^H\ra
=\la S_{3}(f_{p,1}^H)g_{\alpha,3}^H, g_{\beta,3}^H\ra.
\end{split}
\end{align}

By (\ref{e01043}), (\ref{e02029}), (\ref{e04071}), (\ref{e06004}) and (\ref{e05001}), after a careful calculation, we have
\begin{align}\label{e06015}
\begin{split}
&\mB_{\var,T/\var}'=-\left(T\,^0\nabla^{\mS_Z\otimes E}_{e_i}+\frac{\var}{2}\la S_1(e_i)e_j, f_{p,1}^H\ra c(e_j)c(f_{p,1}^H)\right.
\\
&\quad+\frac{\var^2}{4T}\la S_1(e_i)f_{p,1}^H, f_{q,1}^H\ra c(f_{p,1}^H)c(f_{q,1}^H)
+\frac{1}{2}\la S_3(e_i)e_j, g_{\alpha,3}^H\ra c(e_j)g^{\alpha}\wedge
\\
&\quad +\frac{\var}{2T}\la S_3(e_i)f_{p,1}^H, g_{\alpha,3}^H\ra c(f_{p,1}^H)g^{\alpha}\wedge
\left.+\frac{1}{4T}\la S_3(e_i)g_{\alpha,3}^H, g_{\beta,3}^H\ra g^{\alpha}\wedge g^{\beta}\wedge\right)^2
\\
&\quad +dT\wedge\left(c(e_i)\,^0\nabla^{\mS_Z\otimes E}_{e_i}-\frac{1}{8T^2}\left(\var^2\la [f_{p,1}^H,f_{q,1}^H], e_i\ra
c(e_i)c(f_{p,1}^H)c(f_{q,1}^H)\right.\right.
\\
&\quad \quad\left.\left.+4\var \la S_1(g_{\alpha,3}^H)e_i, f_{p,1}^H \ra c(e_i)c(f_{p,1}^H) g^{\alpha}_3\wedge
+\la [g_{\alpha,3}^H,g_{\beta,3}^H], e_i\ra c(e_i)g^{\alpha}\wedge g^{\beta}\wedge\right)\right)
\\
&-\var^2\left(\,^0\nabla^{\mS_Z\otimes E}_{f_{p,1}^H}+\frac{\var}{2T}\la S_1(f_{p,1}^H)e_i,f_{q,1}^H\ra c(e_i)c(f_{q,1}^H)
+\frac{1}{2T}\la S_3(f_{p,1}^H)e_i, g_{\alpha,3}^H\ra c(e_i)g^{\alpha}\wedge\right.
\\
&\quad \left.+\frac{1}{2\var}\la S_2(f_{p})f_{q}, g_{\alpha,2}^H\ra c(f_{q,1}^H)g^{\alpha}\wedge
+\frac{1}{4\var^2}\la S_2(f_{p}) g_{\alpha,2}^H, g_{\beta,2}^H\ra g^{\alpha}\wedge g^{\beta}\wedge\right)^2
\\
&+\frac{\var^2}{4}K_{T/\var}^Z+\frac{T^2}{2}(R^{L_Z}/2+R^E)(e_i,e_j)c(e_i)c(e_j)+T\var(R^{L_Z}/2+R^E)(e_i,f_{p,1}^H)c(e_i)c(f_{p,1}^H)
\\
&+\frac{\var^2}{2}(R^{L_Z}/2+R^E)(f_{p,1}^H,f_{q,1}^H)c(f_{p,1}^H)c(f_{q,1}^H)
+\frac{1}{2}(R^{L_Z}/2+R^E)(g_{\alpha,3}^H,g_{\beta,3}^H)g^{\alpha}\wedge g^{\beta}\wedge
\\
&+\var(R^{L_Z}/2+R^E)(f_{p,1}^H,g_{\alpha,3}^H) c(f_{p,1}^H) g^{\alpha}\wedge
+T(R^{L_Z}/2+R^E)(e_i,g_{\alpha,3}^H)c(e_i)g_3^{\alpha,H}\wedge.
\end{split}
\end{align}

Set
\begin{align}\label{e06016}
\begin{split}
\nabla_{f_{p,1}^H}'=\,^0\nabla^{\mS_Z\otimes E}_{f_{p,1}^H}-\frac{1}{2}\la S_1(e_i)e_i, f_{p,1}^H\ra
+\frac{1}{2\var}&\la S_2(f_{p})f_{q}, g_{\alpha,2}^H\ra c(f_{q,1}^H)g^{\alpha}\wedge
\\
&+\frac{1}{4\var^2}\la S_2(f_{p}) g_{\alpha,2}^H, g_{\beta,2}^H\ra g^{\alpha}\wedge g^{\beta}\wedge.
\end{split}
\end{align}
Recall that $\cE_{X,y_0}=\cC^{\infty}(X_{y_0},\mS(TX, L_X)\otimes E)$,
which is naturally equipped with a Hermitian product attached to $g^{TX}$ and $h^{\mS_X\otimes E}$ as in (\ref{e01025}).
By (\ref{e01028}), the connection $\nabla'$ preserves the scalar product (\ref{e02006}) on $\cE_X$.

Take $y_0\in V^g$ and $\pi_2(y_0)=b$.
We identify $B^{Y_b}(y_0, \alpha_0)$ with $B(0,\alpha_0)\subset T_{y_0}Y=\R^{m}$ by using normal coordinates.
Take a vector $U\in \R^{m}$. We identify $TY|_U$ to $TY|_{\{0\}}$
by parallel transport along the curve $t \mapsto tU$ with respect to the connection $\nabla^{TY}$.
We lift horizontally the paths $t\in \R_+^*\mapsto tU$ into paths $t\in \R_+^*\mapsto x_t\in Z_b$ with $x_t\in X_{tU}$,
$dx_t/dt\in T^HZ_b$. If $x_0\in X_{y_0}$, we identify $T_{x_t}X$, $\mS(TZ,L_Z)\otimes E_{x_t}$ to $T_{x_0}X$, $\mS(TZ,L_Z)\otimes E_{x_0}$
by parallel transport along the curve $t \mapsto x_t$ with respect to the connection $\nabla^{TX}$, $\nabla'$.
Then we can define the operator $\mB_{\var,T/\var}'$ to a neighborhood of $\{0\}\times X_{y_0}$ in $T_{y_0}Y\times X_{y_0}$.

Let $\rho:T_{y_0}Y\rightarrow [0,1]$ be a smooth function such that
\begin{align}\label{e06017}
\rho(U)=
\left\{
  \begin{array}{ll}
    1, & \hbox{$|U|\leq \alpha_0/4$;} \\
    0, & \hbox{$|U|\geq \alpha_0/2$.}
  \end{array}
\right.
\end{align}
 Let $\Delta^{TY}$ be the ordinary Laplacian operator on $T_{y_0}Y$.

Recall that $\ker D^X|_{B^Y(y_0,\alpha_0/2)}$ is a  smooth vector
subbundle of $\cE_{X,y_0}$ on $B^Y(y_0,\alpha_0/2)$.
If $\alpha_0>0$ is small enough, there is a vector bundle $K\subset \cE_{X,y_0}$ over
$T_{y_0}Y$, which coincides with $\ker D^X$ on $B(0,\alpha_0/2)$, with $\ker D_{y_0}^X$
on $T_{y_0}Y \backslash B(0,\alpha_0)$, such that if $K^{\bot}$ is the orthogonal bundle to $K$ in $\cE_{X,y_0}$,
then
\begin{align}\label{e06018}
K^{\bot}\cap \ker D_{y_0}^X=\{0\}.
\end{align}
For $U\in T_{y_0}Y$, in the following sections, let $P^K_U $ be the orthogonal projection operator from $\cE_{X,y_0}$ to $K_U$.
Set $P_{U}^{K,\bot}=1-P^K_{U}$.

Set
\begin{align}\label{e06019}
L_{\var,T}^1=(1-\rho^2(U))(-\var^2\Delta^{TY}+T^2P_U^{K,\bot}D_{y_0}^{X,2}P_U^{K,\bot})+\rho^2(U)(\mB_{\var,T/\var}').
\end{align}
Comparing with (\ref{e06033}), %we can get the following lemma.
for any $m\in\N$ and $T\geq 1$ fixed, there exist $C_1,C_2>0$, such that for $|U|$, $|U'|<\alpha_0/4$, $0<\var\leq 1$,
\begin{align}\label{e956}
|\widetilde{\mathbf{G}}_{\var^2}(L_{\var,T}^1)((U,x),(U',x'))|\leq C_1\exp\left(-\frac{C_2}{\var^2}\right).
\end{align}

For $(U,x)\in N_{Y^g/Y, y_0}\times X_{y_0}$, $|U|<\alpha_0/4$, $\var>0$, set
\begin{align}\label{e566}
(S_{\var}s)(U, x)=s\left(U/\var,x\right).
\end{align}
Put
\begin{align}\label{e567}
\begin{split}
L_{\var,T}^2:=S_{\var}^{-1}L_{\var,T}^1S_{\var}&=
(1-\rho^2(\var U))(-S_{\var}^{-1}\var^2\Delta^{TY}S_{\var}+T^2P_{\var U}^{K,\bot}D_{y_0}^{X,2}P_{\var U}^{K,\bot})
\\
&+\rho^2(\var U)S_{\var}^{-1}\mB_{\var, T/\var}'S_{\var}.
\end{split}
\end{align}

Let $\dim T_{y_0}Y^g=l'$ and $\dim N_{Y^g/Y,y_0}=2l''$. Then $l'+2l''=m$.
Let $\{f_1,\cdots,f_{l'}\}$ be an orthonormal basis of $T_{y_0}Y^g$ and let $\{f_{l'+1},\cdots,f_{l'+2l''}\}$ be an orthonormal basis
of $N_{Y^g/Y, y_0}$.
For $\alpha\in\C(f^p\wedge\, i_{f_p})_{1\leq p\leq l'}$, let $[\alpha]^{max}\in\C$ be the coefficient of $f^1\wedge\cdots\wedge f^{l'}$
in the expansion of $\alpha$. Let $R_{\var}$ be a rescaling such that
\begin{align}\label{e568}
\begin{split}
  &R_{\var}(c(e_i))=c(e_i),
  \\
    &R_{\var}(c(f_{p,1}^H))=\frac{f_1^{p,H}\wedge}{\var}-\var\, i_{f_{p,1}^H}, \quad {\rm for}\ 1\leq p\leq l',
    \\
    &R_{\var}(c(f_{p,1}^H))=c(f_{p,1}^H), \quad {\rm for}\ l'+1\leq p\leq l'+2l''.
\end{split}
\end{align}
Then $R_{\var}$ is a Clifford algebra homomorphism. Set
\begin{align}\label{e569}
L_{\var,T}^3=R_{\var}(L_{\var,T}^2).
\end{align}

Let $\exp(-L_{\var,T}^i)((U,x), (U',x'))$, $\widetilde{\mathbf{F}}_{\var^2}(L_{\var,T}^i)((U,x), (U',x'))$
($(U,x), (U',x')\in T_{y_0}Y\times X_{y_0}$, $i=1,2,3$) be the smooth kernels of
$\exp(-L_{\var,T}^i)$, $\widetilde{\mathbf{F}}_{\var^2}(L_{\var,T}^i)$ with respect to the volume form $dv_{T_{y_0}Y}(U')dv_{X_{y_0}}(x')$.
Using finite propagation speed as in (\ref{e06054}), we see that if $(U,x)\in N_{Y^g/Y, y_0}\times X_{y_0}$, $|U|<\alpha_0/4$, then
\begin{align}\label{e06022}
\widetilde{\mathbf{F}}_{\var^2}(\mB_{\var, T/\var}')(g^{-1}(y_0,U,x),(y_0,U,x))k_Y(y_0,U)
=\widetilde{\mathbf{F}}_{\var^2}(L_{\var,T}^1)(g^{-1}(U,x), (U,x)).
\end{align}
By (\ref{e06009}), (\ref{e06033}), (\ref{e956}) and (\ref{e06022}), there exist $C_1,C_2>0$, such that for $|U|<\alpha_0/4$,
$x\in X_{y_0}$,
\begin{align}\label{e565}
\begin{split}
|\exp(-\mB_{\var, T/\var}')&(g^{-1}(y_0,U,x),(y_0,U,x))k_Y(y_0,U)
\\
&-\exp(-L_{\var,T}^1)(g^{-1}(U,x),(U,x))|\leq C_1\exp\left(-\frac{C_2}{\var^2}\right).
\end{split}
\end{align}

Since $T_{y_0}Y_b$ is an Euclidean space, on $T_{y_0}Y_b$,
\begin{align}\label{e442}
\mS(TY, L_Y)_{y_0}=\mS(TY^g)\widehat{\otimes}\mS(N_{Y^g/Y})\otimes L_Y^{1/2},
\end{align}
where $\mS(\cdot)$ is the spinor space.
From (\ref{e568}), we know that $L_{\var, T}^3((U,x), (U',x'))$ lies in
\begin{align}\label{e06055}
\pi_2^*\Lambda (T_b^*S)\widehat{\otimes}
(\End(\Lambda (T^*Y^g))\widehat{\otimes}C(N_{Y^g/Y})\otimes \End(L_Y^{1/2}))_{y_0}\widehat{\otimes}\End(\mS(TX, L_X)\otimes E)
\end{align}
and acts on
\begin{align}\label{e06056}
\pi_2^*\Lambda (T_b^*S)\widehat{\otimes}(\Lambda (T^*Y^g)\widehat{\otimes}\mS(N_{Y^g/Y})\otimes L_Y^{1/2})_{y_0}\widehat{\otimes} \mS(TX, L_X)\otimes E.
\end{align}

Recall that $\widetilde{c}_{TY^g}$ is the trace element defined in (\ref{e01086}).
\begin{lemma}\label{e05011}
For $t>0$, $(U,x)\in N_{Y^g/Y, y_0}\times X_{y_0}$ and $g\in G$, we have
\begin{multline}\label{e570}
\int_{Y^g}\int_{U\in N_{Y^g/Y}, \atop |U|\leq \alpha_0/4}\int_X \tr_s[g\exp(-L_{\var,T}^1)(g^{-1}(U,x),(U,x))]
dv_{Y^g}(y)dv_{N_Y}(U)dv_{X}(x)
\\
=\int_{Y^g}\int_{U\in N_{Y^g/Y}, \atop |U|\leq \alpha_0/4\var}\int_X \widetilde{c}_{TY^g}\tr_s\left[g\exp(-L_{\var,T}^3)
\left(g^{-1}\left(U,x\right),\left(U,x\right)\right)\right]^{max}
\\
\cdot dv_{Y^g}(y)dv_{N_Y}(U)dv_{X}(x).
\end{multline}
\end{lemma}
\begin{proof}
From (\ref{e567}) and the uniqueness of the heat kernel, we have
\begin{align}\label{e93}
\exp(-L_{\var,T}^2)=S_{\var}^{-1}\exp(-L_{\var,T}^1)S_{\var}.
\end{align}

For $U\in T_{y_0}Y$, $x\in X_{y_0}$, ${\rm supp}\,\phi\subset B(0,\alpha_0/2)\times X_{y_0}$, we have
\begin{align}\label{e94}
\begin{split}
&\int_{T_{y_0}Y}\int_X\exp(-L_{\var,T}^2)((U,x),(U',x'))\phi(U',x')dv_{TY}(U')dv_X(x')
\\=&(\exp(-L_{\var,T}^2)\phi)(U,x)
=(S_{\var}^{-1}\exp(-L_{\var,T}^1)S_{\var}\phi)(U,x)=(\exp(-L_{\var,T}^1)S_{\var}\phi)(\var U,x)
\\
=&\int_{T_{y_0}Y}\int_X\exp(-L_{\var,T}^1)((\var U,x),(U',x'))(S_{\var}\phi)(U',x')dv_{TY}(U')dv_X(x')
\\
=&\var^{\dim Y}\cdot\int_{T_{y_0}Y}\int_X\exp(-L_{\var,T}^1)((\var U,x), (\var U',x'))\phi(U',x')dv_{TY}(U')dv_X(x').
\end{split}
\end{align}
Thus,
\begin{align}\label{e95}
\exp(-L_{\var,T}^1)(g^{-1}(U,x),(U,x))=\var^{-\dim Y}\exp(-L_{\var,T}^2)\left(g^{-1}(U/\var,x),(U/\var,x)\right).
\end{align}

By (\ref{e01086}), (\ref{e01154}), (\ref{e06055}), (\ref{e95}) and the definition of $L_{\var,T}^3$, we have
\begin{align}\label{e96}
\begin{split}
&\tr_s\left[g\exp(-L_{\var,T}^3)\left(g^{-1}(U/\var,x),(U/\var,x)\right)\right]^{max}
\\
=&\sum_j\widetilde{c}_{TY^g}^{-1}\,\var^{-\dim
Y^g}\tr_s\left[g\exp(-L_{\var,T}^2)\left(g^{-1}(U/\var,x),(U/\var,x)\right)\right]
\\
=&\widetilde{c}_{TY^g}^{-1}\var^{\dim_{\R} N}\tr_s\left[g\exp(-L_{\var,T}^1)(g^{-1}(U,x),(U,x))\right].
\end{split}
\end{align}

The proof of Lemma \ref{e05011} is complete.
\end{proof}

\subsection{Proof of Theorem \ref{e03022} ii)}

Let $K^X$ be the scalar curvature of the fibers $(TX, g^{TX})$.
Comparing with \cite[(3.15)-(3.17)]{Bismut1985},
for $T\geq 1$, we can compute that
\begin{align}\label{e574}
\lim_{\var\rightarrow 0}\var^2K_{T/\var}^Z=T^2K^X.
\end{align}

Let $\Gamma'$ be the connection form of $\nabla'$, which is defined in (\ref{e06016}). By using \cite[Proposition 3.7]{MR0167985}, we see that
for $U\in TY=\R^{m}$,
\begin{align}\label{e06023}
\Gamma^{'}_{U}=\frac{1}{2}(\nabla')^2(U, \cdot)+O\big(|U|^{2}\big).
\end{align}

\begin{lemma}\label{e06024}
For $U,V\in TY$, the following identity holds.
\begin{multline}
(\nabla')^2(U_1^H,V_1^H)=\frac{1}{4}\la R^{TX}(U_1^H,V_1^H)e_i,e_j\ra c(e_i)c(e_j)+\left(\frac{1}{2}R^{L_Z}+R^E\right)(U_1^H,V_1^H)
\\
+\frac{1}{4}\la R^{TY}(f_{p},f_q)U,V\ra c(f_{p,1}^H)c(f_{q,1}^H)
+\frac{1}{4\var^{2}}\la R^{TY}(g_{\alpha,2}^H, g_{\beta,2}^H)U,V\ra g^{\alpha}\wedge g^{\beta}\wedge
\\
-\frac{1}{2}d(\la S_1(e_i)e_i, \cdot\ra)(U_1^H,V_1^H)
+\frac{1}{2\var}\la R^{TY}(f_p, g_{\alpha,2}^H)U,V\ra c(f_{p,1}^H)g^{\alpha}\wedge .
\end{multline}
\end{lemma}
\begin{proof}

By the fundamental identity of \cite[Theorem 4.14]{Bismut1985} (see also \cite[(7.15)]{MR1942300}), for $Z, W \in TV$,
\begin{multline}\label{e06042}
\la R^{TY}(U,V)P^{TY}Z, P^{TY}W\ra+\la (S_2P^{TY}S_2)(U,V)Z, W\ra
\\
+\la (\nabla^{TY} S_2)(U,V)Z,W \ra=\la R^{TY}(Z,W)U,V\ra.
\end{multline}

Since $S_2$ maps $TY$ to $T_2^HV$, we have
\begin{align}\label{e06043}
(S_2P^{TY}S_2)(U,V)f_p=0, \quad \la (\nabla^{TY} S_2)(U,V)f_p,f_q\ra=0.
\end{align}

Then Lemma \ref{e06024} follows from (\ref{e06016}), (\ref{e06042}) and  (\ref{e06043}).
\end{proof}

\begin{lemma}\label{e593}
When $\var\rightarrow 0$, the limit $L_{0,T}^3=\lim_{\var\rightarrow 0}L_{\var,T}^3$ exists and
\begin{align}\label{e594}
L_{0,T}^3|_{V^g}=-\left(\partial_{p}+\frac{1}{4}\la R^{TY}|_{V^g}U,f_{p,1}^H\ra\right)^2+\frac{1}{2}R^{L_Y}|_{V^g}+\mB_{T^2}''|_{V^g}.
\end{align}
\end{lemma}
\begin{proof}
By (\ref{e06023}) and Lemma \ref{e06024}, we have

\begin{multline}
\lim_{\var\rightarrow 0}R_{\var^2}[\var S_{\var^2}^{-1}\nabla'_{f_p}|_US_{\var^2}]
=
\partial_p+ \lim_{\var\rightarrow 0}R_{\var^2}[\var^2 (S_{\var}^{-1}(\nabla')^2S_{\var})(U, f_p)]
\\
=\partial_p+\frac{1}{4}\sum_{1\leq q, r\leq l'}\la R^{TY}(f_{q},f_r)U, f_p\ra f^q\wedge f^r\wedge
+\frac{1}{4}\la R^{TY}(g_{\alpha,2}^H, g_{\beta,2}^H)U, f_p\ra g^{\alpha}\wedge g^{\beta}\wedge
\\
+\frac{1}{2}\sum_{1\leq q\leq l'}\la R^{TY}(f_q, g_{\alpha,2}^H)U, f_p\ra f^q\wedge g^{\alpha}\wedge.
\end{multline}

Then by (\ref{e06015}), (\ref{e574}) and the definition of $L_{\var,T}^3$, we have
\begin{align}\label{e591}
\begin{split}
\lim_{\var\rightarrow 0}&L_{\var,T}^3=-\left(T\,^0\nabla^{\mS_Z\otimes E}_{e_i}+\frac{1}{2}
\sum_{1\leq p\leq l'}\la S_1(e_i)e_j, f_{p,1}^H\ra c(e_j)f^p\wedge\right.
\\
&\quad+\frac{1}{4T}\sum_{1\leq p,q\leq l'}\la S_1(e_i)f_{p,1}^H, f_{q,1}^H\ra f^p\wedge f^q\wedge
+\frac{1}{2}\la S_3(e_i)e_j, g_{\alpha,3}^H\ra c(e_j)g^{\alpha}\wedge
\\
&\quad +\frac{1}{2T}\sum_{1\leq p\leq l'}\la S_3(e_i)f_{p,1}^H, g_{\alpha,3}^H\ra f^p\wedge g^{\alpha}\wedge
\left.+\frac{1}{4T}\la S_3(e_i)g_{\alpha,3}^H, g_{\beta,3}^H\ra g^{\alpha}\wedge g^{\beta}\wedge\right)^2
\\
&\quad +dT\wedge\left(D^X-\frac{1}{8T^2}\left(\sum_{1\leq p,q\leq l'}
\la [f_{p,1}^H,f_{q,1}^H], e_i\ra c(e_i)f^p\wedge f^q\wedge\right.\right.
\\
&\quad \quad\left.\left.+4\sum_{1\leq p\leq l'}\la S_1(g_{\alpha,3}^H)e_i, f_{p,1}^H\ra c(e_i)f^p\wedge g^{\alpha}\wedge
+\la [g_{\alpha,3}^H,g_{\beta,3}^H], e_i\ra c(e_i)g^{\alpha}\wedge g^{\beta}\wedge\right)\right)
\\
&-\left(\partial_p+\frac{1}{4}\sum_{1\leq q, r\leq l'}\la R^{TY}(U, f_p)f_{q},f_r\ra f^q\wedge f^r\wedge\right.
\\
&
+\left.\frac{1}{4}\la R^{TY}(U, f_p)g_{\alpha,2}^H, g_{\beta,2}^H\ra g^{\alpha}\wedge g^{\beta}\wedge
+\frac{1}{2}\sum_{1\leq q\leq l'}\la R^{TY}(U, f_p)f_q, g_{\alpha,2}^H\ra f^q\wedge g^{\alpha}\wedge\right)^2
\\
&+\frac{T^2}{4}K^X+\frac{T^2}{2}(R^{L_Z}/2+R^E)(e_i,e_j)c(e_i)c(e_j)+T \sum_{1\leq p\leq l'}(R^{L_Z}/2+R^E)(e_i,f_{p,1}^H)c(e_i) f^p\wedge
\\
&+\frac{1}{2}\sum_{1\leq p,q\leq l'}(R^{L_Z}/2+R^E)(f_{p,1}^H,f_{q,1}^H)f^p\wedge f^q\wedge
+\frac{1}{2}(R^{L_Z}/2+R^E)(g_{\alpha,3}^H,g_{\beta,3}^H)g^{\alpha}\wedge g^{\beta}\wedge
\\
&+\sum_{1\leq p\leq l'} (R^{L_Z}/2+R^E)(f_{p,1}^H,g_{\alpha,3}^H) f^p\wedge g^{\alpha}\wedge
+T(R^{L_Z}/2+R^E)(e_i,g_{\alpha,3}^H)c(e_i)g_3^{\alpha,H}\wedge.
\end{split}
\end{align}
By (\ref{e01043}) and (\ref{e06050}), we have
\begin{multline}\label{e05012}
(B_{1, T^2}|_{V^g})^2=-\left(T\,^0\nabla^{\mS_Z\otimes E}_{e_i}+\frac{1}{2}\sum_{1\leq p\leq l'}\la S_1(e_i)e_j, f_{p,1}^H\ra c(e_j)f^p\wedge\right.
\\
\quad+\frac{1}{4T}\sum_{1\leq p, q\leq l'}\la S_1(e_i)f_{p,1}^H, f_{q,1}^H\ra f^p\wedge f^q\wedge
+\frac{1}{2}\la S_3(e_i)e_j, g_{\alpha,3}^H\ra c(e_j)g^{\alpha}\wedge
\\
\quad +\frac{1}{2T}\sum_{1\leq p\leq l'}\la S_3(e_i)f_{p,1}^H, g_{\alpha,3}^H\ra f^p\wedge g^{\alpha}\wedge
\left.+\frac{1}{4T}\la S_3(e_i)g_{\alpha,3}^H, g_{\beta,3}^H\ra g^{\alpha}\wedge g^{\beta}\wedge\right)^2
\\
+\frac{T^2}{4}K^X+\frac{T^2}{2}(R^{L_Z}/2+R^E)(e_i,e_j)c(e_i)c(e_j)+T \sum_{1\leq p\leq l'}(R^{L_Z}/2+R^E)(e_i,f_{p,1}^H)c(e_i) f^p\wedge
\\
+\frac{1}{2}\sum_{1\leq p,q\leq l'}(R^{L_Z}/2+R^E)(f_{p,1}^H,f_{q,1}^H)f^p\wedge f^q\wedge
+\frac{1}{2}(R^{L_Z}/2+R^E)(g_{\alpha,3}^H,g_{\beta,3}^H)g^{\alpha}\wedge g^{\beta}\wedge
\\
+\sum_{1\leq p\leq l'} (R^{L_Z}/2+R^E)(f_{p,1}^H,g_{\alpha,3}^H) f^p\wedge g^{\alpha}\wedge
+T(R^{L_Z}/2+R^E)(e_i,g_{\alpha,3}^H)c(e_i)g_3^{\alpha,H}\wedge.
\end{multline}

So
\begin{align}
\lim_{\var\rightarrow 0}L_{\var,T}^3=-\left(\partial_{p}+\frac{1}{4}\la R^{TY}|_{V^g}U,f_{p,1}^H\ra\right)^2+
\frac{1}{2}R^{L_Y}|_{V^g}+\mB_{T^2}''|_{V^g}.
\end{align}
The proof of Lemma \ref{e593} is complete.
\end{proof}

\begin{thm}\label{e107}
i) For $T\geq 1$ fixed and $k\in \N$, there exist $c>0, C>0, r\in\N$
such that for any $(U,x), (U',x')\in T_{y_0}Y\times X_{y_0}$, $\var\in (0,1]$,
\begin{multline}\label{e108}
\sup_{|\alpha|,|\alpha'|\leq k}\left|\frac{\partial^{|\alpha|+|\alpha'|}}{\partial U^{\alpha}\partial U'^{\alpha'}}
\exp(-L_{\var,T}^3)((U,x),(U',x'))\right|
\\
\leq c(1+|U|+|U'|)^r\exp(-C|U-U'|^2).
\end{multline}

ii) For $T\geq 1$ fixed,  there exist $c>0, C>0, r\in\N$, $\gamma>0$,
such that for any $(U,x), (U',x')\in T_{y_0}Y\times X_{y_0}$, $\var\in (0,1]$,
\begin{multline}\label{e109}
|(\exp(-L_{\var,T}^3)-\exp(-L_{0,T}^3))((U,x), (U',x'))|
\\
\leq c\var^{\gamma}(1+|U|+|U'|)^r\exp(-C|U-U'|^2).
\end{multline}
\end{thm}
The proof of Theorem \ref{e107} is left to the next subsection.

On the vector space $N_{Y^g/Y, y_0}$, there exists $c>0$, such that for any $U\in N_{Y^g/Y, y_0}$,
\begin{align}\label{e110}
|g^{-1}U-U|\geq c|U|.
\end{align}
Then by (\ref{e565}), Lemma \ref{e05011}, \ref{e593}, Theorem \ref{e107} and the dominated convergence theorem, we have
\begin{align}\label{e111}
\begin{split}
&\lim_{\var\rightarrow 0}\psi_S\tr_s[g\exp(-\mB_{\var, T/\var}')]
\\
=&\int_{Y^g}\int_{N_{Y^g/Y}}\int_X\widetilde{c}_{TY^g}\,\psi_S\tr_s
\left[g\exp(-L_{0,T}^3)(g^{-1}(U,x),(U,x))\right]dv_N(U)dv_X(x).
\end{split}
\end{align}

By Mehler's formula (cf. \cite[(1.33)]{MR1756105}) and (\ref{e01054}),
\begin{align}\label{e112}
\begin{split}
&\int_X\tr_s\left[g\exp(-L_{0,T}^3)(g^{-1}(U,x),(U,x))\right]dv_X(x)
\\
=&(4\pi)^{-\frac{1}{2}\dim Y}\mathrm{det}^{\frac{1}{2}} \left(\frac{R^{TY}/2}{\sinh (R^{TY}/2)}\right)
\exp\left\{-\frac{1}{4}\left\la\frac{R^{TY}/2}{\tanh (R^{TY}/2)}U,U\right\ra\right.
\\
-&\frac{1}{4}\left\la \frac{R^{TY}/2}{\tanh (R^{TY}/2)}g^{-1}U,g^{-1}U\right\ra
\quad\left.+\frac{1}{2}\left\la\frac{R^{TY}/2}{\sinh (R^{TY}/2)}
\exp(R^{TY}/2)U,g^{-1}U\right\ra\right\}
\\
&\cdot \tr_s[g|_{\mS(N)}]\wedge\tr\left[g\exp\left(-\frac{1}{2}R^{L_Y}|_{V^g}\right)\right]\wedge\tr_s[g\exp(\mB_{T^2}''|_{V^g})].
\end{split}
\end{align}

Following the same computations in \cite[(1.33)-(1.38)]{MR1756105},
 by (\ref{e01050}), (\ref{e01051}), (\ref{e01138}) and (\ref{e111}), we have
\begin{multline}\label{e06057}
\lim_{\var\rightarrow 0}\psi_S\tr_s[g\exp(-\mB_{\var, T/\var}')]
\\
=\psi_S\int_{Y^g}\widetilde{c}_{TY^g}(4\pi)^{-\frac{\dim Y^g}{2}}\psi_{V^g}^{-1}\left(\widehat{\mathrm{A}}_g(TY,\nabla^{TY})\wedge \ch_g(L_Y^{1/2},\nabla^{L_Y^{1/2}})
\wedge \psi_{V^g}\tr_s[g\exp(\mB_{T^2}'')|_{V^g}]\right).
\end{multline}

Using (\ref{e01059}), (\ref{e06005}) and (\ref{e06051}), we get Theorem \ref{e03022} ii) when $\dim Z$ and $\dim Y$ are all even.

\subsection{General case}\label{s0604}

When $\dim Y$ is odd and $\dim Z$ is even, by (\ref{e01154}), %we have
following the same process in this section,
we can get an analogue of (\ref{e06057}):
\begin{multline}\label{e06058}
\lim_{\var\rightarrow 0}\psi_S\tr^{\mathrm{odd}}[g\exp(-\mB_{\var, T/\var}')]
\\
=\psi_S\int_{Y^g}\widetilde{c}_{TY^g}(4\pi)^{-\frac{\dim Y^g}{2}}\psi_{V^g}^{-1}\left(\widehat{\mathrm{A}}_g(TY,\nabla^{TY})\wedge
\ch_g(L_Y^{1/2},\nabla^{L_Y^{1/2}})
\wedge \psi_{V^g}\tr_s[g\exp(\mB_{T^2}'')|_{V^g}]\right).
\end{multline}
Then Theorem \ref{e03022} ii) in this case follows from (\ref{e01059}), (\ref{e06005}), (\ref{e06051}) and (\ref{e06058}).

When $\dim Y$ is even and $\dim Z$ is odd, it is the same as the case above.

When $\dim Y$ and $\dim Z$ are all odd,  by (\ref{e01154}), as in (\ref{e06058}), we have
\begin{multline}\label{e06061}
\lim_{\var\rightarrow 0}\psi_S\tr_s[g\exp(-\mB_{\var, T/\var}')]
=2\sqrt{-1}\psi_S\int_{Y^g}\widetilde{c}_{TY^g}(4\pi)^{-\frac{\dim Y^g}{2}}
\\
\cdot\psi_{V^g}^{-1}\left(\widehat{\mathrm{A}}_g(TY,\nabla^{TY})\wedge
\ch_g(L_Y^{1/2},\nabla^{L_Y^{1/2}})
\wedge \psi_{V^g}\tr[g\exp(\mB_{T^2}''|_{V^g})]\right).
\end{multline}
Since the left hand side of (\ref{e06061}) takes value in even forms and $\dim Y^g$ is odd,
by (\ref{e01090}) and (\ref{e01059}),  we have
\begin{multline}\label{e06062}
\lim_{\var\rightarrow 0}\psi_S\tr_s[g\exp(-\mB_{\var, T/\var}')]
\\
=\int_{Y^g}\widehat{\mathrm{A}}_g(TY,\nabla^{TY})\wedge \ch_g(L_Y^{1/2},\nabla^{L_Y^{1/2}})
\wedge \psi_{V^g}\tr^{\mathrm{odd}}[g\exp(\mB_{T^2}''|_{V^g})].
\end{multline}

The proof of Theorem \ref{e03022} ii) is complete.

\subsection{Proof of Theorem \ref{e107}}

We prove Theorem \ref{e107} by following the process of \cite[Section 11]{MR1188532} and \cite[Section 11]{MR1316553}.

Let $I^0$ be the vector space of square integrable sections of $\pi_2^*\Lambda (T^{*}S)\widehat{\otimes}\Lambda (TY^g)
\widehat{\otimes} \mS(N_{Y^g/Y})\otimes L_Y^{1/2} \widehat{\otimes}\mS(TX,L_X)\otimes E$ over $T_{y_0}Y_b\times X_{y_0}$.
For $0\leq q\leq \dim Y^g$, let $I^0_q$ be the vector space of square integrable sections of
$\pi_2^*\Lambda (T^{*}S)\widehat{\otimes}\Lambda^q (TY^g)
\widehat{\otimes} \mS(N_{Y^g/Y})\otimes L_Y^{1/2} \widehat{\otimes}\mS(TX,L_X)\otimes E$. Then $I^0=\oplus_{q=0}^{l'}I_q^0$.
Similarly, if $p\in \R$, $I^p$ and $I_q^p$ denote the corresponding $p$-th Sobolev spaces.

For $U\in T_{y_0}Y^g$, set
\begin{align}\label{e06025}
g_{\var}(U)=1+(1+|U|^2)^{\frac{1}{2}}\rho\left(\frac{\var U}{2}\right).
\end{align}
If $s\in I_q^0$, set
\begin{align}\label{e444}
|s|_{\var,0}^2=\int_{T_{y_0}Y_b\times X_{y_0}}|s(U,x)|^2g_{\var}(U)^{2(l'-q)}dv_{TY}(U)dv_X(x).
\end{align}
Let $\la\cdot,\cdot\ra_{\var,0}$ be the Hermitian product attached to $|\cdot|_{\var,0}$.

So, for $1\leq p\leq l'$, $s\in I_p$, we can get
\begin{align}\label{e446}
\begin{split}
\left|1_{\var|U|\leq\alpha_0/2}|U|\right.&\left.(f^p\wedge-\var^2i_{f_p})s\right|_{\var,0}^2=
\left|1_{\var|U|\leq\alpha_0/2}|U|f^p\wedge s\right|_{\var,0}^2
+\left|1_{\var|U|\leq\alpha_0/2}|U|\var^2i_{f_p}s\right|_{\var,0}^2
\\
&=\int_{|U|\leq\frac{\alpha_0}{2\var}}|s|^2|U|^2(1+(1+|U|^2)^{\frac{1}{2}})^{2(l-p-1)}dv_{TY}(U)
\\
&\quad+\int_{|U|\leq\frac{\alpha_0}{2\var}}\var^4|s|^2|U|^2(1+(1+|U|^2)^{\frac{1}{2}})^{2(l-p+1)}dv_{TY}(U).
\end{split}
\end{align}
Since there exists $C>0$, such that
\begin{align}\label{e447}
\begin{split}
\frac{|U|}{1+(1+|U|^2)^{\frac{1}{2}}}\leq 1, \quad
\var^4|U|^2(1+(1+|U|^2)^{\frac{1}{2}})^2 \leq C,
\end{split}
\end{align}
we have the following lemma.
\begin{lemma}\label{e448}
The operators
$1_{\var |U|\leq\alpha_0/2}(f^p\wedge-\var^2i_{f_p})$ and
  $1_{\var |U|\leq\alpha_0/2}|U|(f^p\wedge-\var^2i_{f_p})$
are uniformly bounded with respect to the norm $|\cdot|_{\var,0}$.
\end{lemma}

Set $\mD_H=\{\partial_p, \nabla^{\mS_X\otimes E}_{e_i}\}$.
Set
\begin{align}\label{e974}
|s|_{\var,k}^2=\sum_{l=0}^k\sum_{
	Q_i\in \mD_H}
|Q_1\cdots Q_ls|_{\var,0}^2.
\end{align}

\begin{lemma}(cf. \cite[Theorem 11.26]{MR1188532})\label{e06027}
For $T\geq 1$ fixed, there exist $c_1,c_2,c_3,c_4>0$, such that for any $\var\in (0,1]$, $s\in I^1$,
\begin{align}\label{e973}
\begin{split}
&\Re\la L_{\var,T}^3s,s\ra_{\var,0}\geq c_1|s|_{\var,1}^2-c_2|s|_{\var,0}^2,
\\
&|\Im\la L_{\var,T}^3s,s\ra_{\var,0}|\leq c_3|s|_{\var,1}|s|_{\var,0},
\\
&|\la L_{\var,T}^3s,s'\ra_{\var,0}|\leq c_4|s|_{\var,1}|s'|_{\var,1}.
\end{split}
\end{align}
\end{lemma}
\begin{proof}
let $\nabla$ denote the gradient in the variable $U$.
Since $\rho$ has compact support, there exists $C>0$, such that
\begin{align}\label{e06071}
\left|\nabla\left(g_{\var}(U)\right)\right|\leq C.
\end{align}
From Lemma \ref{e448} and the definition of $L_{\var,T}^3$, we can get Lemma \ref{e06027}.
\end{proof}

As in (\ref{e04022}), set
\begin{align}\label{e460}
|s|_{\var,-1}:=\sup_{0\neq s'\in I^1}\frac{\la s,s'\ra_{\var,0}}{|s'|_{\var,1}}.
\end{align}

\begin{lemma}\label{e459}
There exist $c, C>0$ such that if
\begin{align}\label{e461}
\lambda\in U=\left\{\lambda\in \C:\Re(\lambda)\leq\frac{\Im(\lambda)^2}{4c^2}-c^2\right\},
\end{align}
the resolvent $(\lambda-L_{\var,T}^3)^{-1}$ exists, and moreover for any $\var\in (0,1]$, $s\in I^1$,
\begin{align}\label{e462}
\begin{split}
&|(\lambda-L_{\var,T}^3)^{-1}s|_{\var,0}\leq C|s|_{\var,0},
\\
&|(\lambda-L_{\var,T}^3)^{-1}s|_{\var,1}\leq C(1+|\lambda|)^2|s|_{\var,-1}.
\end{split}
\end{align}
\end{lemma}
\begin{proof}
Take $c_2$ in Lemma \ref{e06027}.
If $\lambda\in\R$, $\lambda\leq -2c_2$, for $s\in I^1$, we have
\begin{align}\label{e500}
\Re\la(L_{\var,T}^3-\lambda)s,s\ra_{\var,0}\geq c_1|s|_{\var,0}^2.
\end{align}
So
\begin{align}\label{e501}
|s|_{\var,0}\leq c_1^{-1}|(L_{\var,T}^3-\lambda)s|_{\var,0}.
\end{align}
Since $|s|_{\var,0}\leq c(\var)|s|_0$ for $c(\var)>0$,
\begin{align}\label{e502}
|s|_0\leq |s|_{\var,0}\leq c_1^{-1}|(L_{\var,T}^3-\lambda)s|_{\var,0}\leq c(\var)c_1^{-1}|(L_{\var,T}^3-\lambda)s|_0.
\end{align}
So $(L_{\var,T}^3-\lambda)^{-1}$ exists for $\lambda\in\R$, $\lambda\leq -2c_2$.

Set $\lambda=a+ib\in\C$. Then by Lemma \ref{e06027},
\begin{align}\label{e503}
\begin{split}
|\la(L_{\var,T}^3-\lambda)s,s\ra_{\var,0}|&\geq \max\{\Re\la L_{\var,T}^3s,s\ra_{\var,0}-a|s|_{\var,0}^2, |\Im\la
L_{\var,T}^3s,s\ra_{\var,0}-b|s|_{\var,0}^2|\}
\\
&\geq \max\{c_1|s|_{\var,1}^2-(c_2+a)|s|_{\var,0}^2,-c_3|s|_{\var,1}|s|_{\var,0}+|b||s|_{\var,0}^2\}.
\end{split}
\end{align}

Set
\begin{align}\label{e504}
C(\lambda)=\inf_{t\in\R,t\geq 1}\max\{c_1t^2-(c_2+a),-c_3t+|b|\}.
\end{align}
If $c>0$ is small enough, we can get
\begin{align}\label{e505}
c_0=\inf_{\lambda\in U}C(\lambda)>0.
\end{align}
Since $|s|_{\var,0}\leq |s|_{\var,1}$, if the resolvent $(\lambda-L_{\var,T}^3)^{-1}$ exists, then
\begin{align}\label{e506}
|(\lambda-L_{\var,T}^3)^{-1}s|_{\var,0}\leq c_0^{-1}|s|_{\var,0}.
\end{align}
From (\ref{e506}), if $\lambda'\in U$, $|\lambda'-\lambda|\leq c_0/2$, then the resolvent $(\lambda'-L_{\var,T}^3)^{-1}$ exists.
By (\ref{e502}), we get the first inequality of (\ref{e462}).

For $\lambda_0\in\R$, $\lambda_0\leq -2c_2$ and $s\in I^1$,
by Lemma \ref{e06027}, we have
\begin{align}\label{e507}
|\la (\lambda_0-L_{\var,T}^3)s,s\ra_{\var,0}|\geq c_1|s|_{\var,1}^2.
\end{align}
Following the same process in (\ref{e04064})-(\ref{e04068}), we get the second estimate of (\ref{e462}).

The proof of Lemma \ref{e459} is complete.
\end{proof}

As in Lemma \ref{e04107}, since $[Q, L_{\var,T}^3]$ has the same structure as $L_{\var,T}^3$ for $Q\in \mD_H$,
for any $k\in \mathbb{N}$ fixed, there exists $C_k>0$ such that for $\var\in (0,1]$, $Q_1,\cdots, Q_k\in\mD_H$
and $s,s'\in I^2$, we have
\begin{align}\label{e04222}%{e05014}
|\langle [Q_1, [Q_2,\cdots[Q_k, L_{\var,T}^3],\cdots]]s,s'\ra_{\var,0}|\leq C_k|s|_{\var,1}|s'|_{\var,1}.
\end{align}
Then using the proof of Lemma \ref{e04126}, we can get the Lemma as follows.

\begin{lemma}\label{e464}
For any $\var\in (0,1]$, $\lambda$ satisfies (\ref{e461}) and $m\in\N$, there exist $C_m>0$ and $p_m\in\N$, such that
\begin{align}\label{e466}
|(\lambda-L_{\var,T}^3)^{-1}s|_{\var,m+1}\leq C_m(1+|\lambda|)^{p_m}|s|_{\var,m}.
\end{align}
\end{lemma}

Set
\begin{align}\label{e474}
\Gamma=\partial U=\left\{\lambda\in\C: \Re(\lambda)=\frac{\Im(\lambda)^2}{4c^2}-c^2\right\},
\end{align}
and
\begin{align}
\Gamma'=\{\lambda\in\C: |\Im \lambda|\leq c\}.
\end{align}
Then the map $\lambda\mapsto \lambda^2$ sends $\Gamma'$ to $\Gamma$.
Let $\Delta=-\Delta^{TY}+D_{y_0}^{X,2}$.
For $\lambda\in\Gamma$, $k,m\in \N$ and $k\leq m$, from Lemma \ref{e464}, there exist $C_k>0$ and $p_m'>0$ such that
\begin{align}\label{e475}
\begin{split}
&|\Delta^{k}(\lambda-L_{\var,T}^3)^{-m}s|_{\var,0}\leq |(\lambda-L_{\var,T}^3)^{-m}s|_{\var,k}
\\
\leq &C_k(1+|\lambda|)^{p_m'}|(\lambda-L_{\var,T}^3)^{-m+k}s|_{\var,0}
\leq C_k(1+|\lambda|)^{p_m'}|s|_{\var,0}.
\end{split}
\end{align}

Denote by $L_{\var,T}^{3,*}$ the formal adjoint of $L_{\var,T}^{3}$ with respect to the usual Hermitian product in $I^0$.
Then $L_{\var,T}^{3,*}$ has the same structure as $L_{\var,T}^{3}$
except that
we replace the operators $f^{p}\wedge$, $i_{f_{p}}$ by $i_{f_{p}}$
and $f^{p}\wedge$.
For $s\in I_q^0$,
set
\begin{align}\label{e06044}
|s|_{\var,0}'^2=\int_{T_{y_0}Y_b\times X_{y_0}}|s(U,x)|^2g_{\var}(U)^{2(q-l')}dv_{TY}(U)dv_X(x).
\end{align}
From the above analysis associated to $|\cdot|_{\var,0}'$, we obtain (\ref{e475}) for
$L_{\var,T}^{3,*}$ and $|\cdot|_{\var,0}'$.
Taking adjoint with respect to the usual Hermitian product in $I^0$, we have
\begin{align}\label{e477}
|(\lambda-L_{\var,T}^{3})^{-m}\Delta^{k}s|_{\varepsilon,0}
\leq C_k(1+|\lambda|)^{p_m'}|s|_{\varepsilon,0}.
\end{align}
So for $k, k', m\in \N$ and $m\geq k+k'$, there exists $C_{k,k'}>0$, such that
\begin{align}\label{e478}
\begin{split}
|\Delta^{k}\exp(-L_{\var,T}^3)\Delta^{k'}s|_{\varepsilon,0}
=&\left|\frac{(-1)^{m-1}(m-1)!}{2\pi i}\int_{\Gamma}e^{-\lambda}\Delta^{k}(\lambda-L_{\var,T}^{3})^{-m}\Delta^{k'}s\right|_{\varepsilon,0}
\\
\leq &C_{k,k'}\left(\int_{\Gamma}e^{-\lambda}(1+|\lambda|)^{p_m'}d\lambda\right)|s|_{\varepsilon,0}
\\
= &C_{k,k'}\left(\int_{\Gamma'}e^{-\lambda^2}(1+|\lambda^2|)^{p_m'}d\lambda\right)|s|_{\varepsilon,0}\leq C|s|_{\varepsilon,0}.
\end{split}
\end{align}

Take $p\in \N$. Let $J_{p,y_0}^0$ be the set of square integrable sections of
$\Lambda (TV^g)
\widehat{\otimes} \mS(N_{Y^g/Y})\otimes L_Y^{1/2} \widehat{\otimes}\mS(TX,L_X)$ over
\begin{align}\label{e06032}
\left\{(U,x)\in T_{y_0}Y\times X_{y_0}; x\in X_{y_0}, |U|\leq p+\frac{1}{2}\right\}.
\end{align}
We equip $J_{p,y_0}^0$ with the Hermitian product for $s\in J_{p,y_0}^0$,

\begin{align}\label{e479}
|s|_{(p),0}^2=\int_{|U|\leq p+\frac{1}{2}}\int_{X_{y_0}}|s(U,x)|^2 dv_{T_{y_0}Y}dv_X.
\end{align}

Obviously, there exists $C>0$ such that for any $p\in \N$, $s\in J_{p,y_0}^0$,
\begin{align}\label{e06028}
|s|_{(p),0}\leq |s|_{\var,0}\leq C(1+p)^{l'}|s|_{(p),0}.
\end{align}

By (\ref{e478}) and (\ref{e06028}), we find for any $k\leq m$, $k'\leq m'$, there exists $C'>0$
such that for $\var\in (0,1]$, $p\in \N$, $s\in J_{p,y_0}^0$,
\begin{align}\label{e481}
\begin{split}
|\Delta^{k}\exp(-L_{\var,T}^3)\Delta^{k'}s|_{(p),0}\leq
\left|\Delta^{k}\exp(-L_{\var,T}^3)\Delta^{k'}s\right|_{\var,0}
%\leq c_{\alpha,\alpha'}|s|_{\var,0}
\leq C'(1+p)^{l'}|s|_{(p),0}.
\end{split}
\end{align}
Thus, following the same process in (\ref{e04138}) - (\ref{e04140}), for $k, k'\in \N$ there exists $C>0$, $r>0$
such that for $\var\in (0,1]$, $p\in \N$,
\begin{align}\label{e482}
\sup_{|U|, |U'|\leq p+1/4}|\Delta^{k}_{(U,x)}\Delta^{k'}_{(U',x')}\exp(-L_{\var,T}^3)((U,x), (U',x'))|\leq C(1+p)^r.
\end{align}
So we get the bounds in (\ref{e108}) with $C=0$.

By (\ref{e06007}) and (\ref{e06008}), we write
\begin{multline}\label{e485}
\widetilde{\mathbf{G}}_u(L_{\var,T}^3)((U,x),(U',x'))
\\
=\int_{-\infty}^{+\infty}\cos\left(\sqrt{2}v\sqrt{L_{\var,T}^3}\right)((U,x),(U',x'))e^{-\frac{v^2}{2}}(1-f(\sqrt{u}v))\frac{dv}{\sqrt{2\pi}}.
\end{multline}
\begin{lemma}\label{e898}
There exist $C_1, C_2>0$, $r>0$,
such that for $\var\in (0,1]$, $m, m'\in \N$,
\begin{align}\label{e487}
\sup_{|\beta|\leq m, |\beta'|\leq m'}|\partial^{\beta}_U\partial^{\beta'}_{U'}\widetilde{\mathbf{G}}_u(L_{\var,T}^3)((U',x'),(U,x))|
\leq C_1(1+|U|+|U'|)^r\exp\left(-\frac{C_2}{u}\right).
\end{align}
\end{lemma}
\begin{proof}
After replacing $\exp(-L_{\var,T}^3)$ to $\widetilde{\mathbf{G}}_u(L_{\var,T}^3)$ in (\ref{e478})-(\ref{e482})
and using (\ref{e68}), we get Lemma \ref{e898}.
\end{proof}

If $|\sqrt{u}v|\leq \alpha/2$, then $f(\sqrt{u}v)=0$.
Using finite propagation speed of the hyperbolic equation
for the solution of hyperbolic equations for $\cos (s\sqrt{L_{\var,T}^3})$ (cf. \cite[\S 7.8]{MR598467}, \cite[\S 4.4]{MR618463}),
there exists a constant $C_0'>0$, such that
\begin{align}\label{e488}
\widetilde{\mathbf{G}}_{u}(L_{\var,T}^3)((U,x),(U',x'))=\exp(L_{\var,T}^3)((U,x),(U',x')),
\end{align}
if $|U-U'|\geq C_0'/\sqrt{u}$.

Then by (\ref{e488}) and Lemma \ref{e898}, For $m, m'\in \N$, there exists $C_1, C_2>0$, $r>0$,
such that for $\var\in (0,1]$,
\begin{align}\label{e489}
\begin{split}
\sup_{|\beta|\leq m,|\beta'|\leq m'}&\left|\partial^{\beta}_U\partial^{\beta'}_{U'}\exp(-L_{\var,T}^3)((U,x),(U',x'))\right|
\\
&\leq C_1(1+|U|+|U'|)^r\exp\left(-\frac{C_2|U-U'|^2}{C_0'^2}\right).
\end{split}
\end{align}

So we get the bounds in (\ref{e108}).

For $U\in T_{y_0}Y$, set $U=U_pf_p$.
Let $|\cdot |_{0, k}$ be the limit norm of $|\cdot|_{\var, k}$
as $\var\rightarrow 0$ for $k\in \{-1, 0, 1\}$.
Note that all the estimates in this subsection work for $\var=0$.
For $k\in \{-1, 0, 1\}$ and $k'\in\N $, set
\begin{align*}
I^{k, k'}_{0}=\big\{s\in I^k:\
U^{\alpha}s\in I^k\
\textup{for}\ |\alpha|\leq k'\big\}.
\end{align*}
For $s\in I^{k, k'}_{0}$, set
\begin{align}\label{wlu5.140}
\big|s\big|_{0, (k, k')}^{2}
=\sum_{|\alpha|\leqslant k'}\big|U^{\alpha}s\big|^{2}_{0, k}.
\end{align}

\begin{lemma}\label{e06045}
There exist $C>0$, $k, k'\in \N $ such that
for $s\in I$,
\begin{align}\label{e06046}
\Big|\big[\big(\lambda-L^{3}_{\var,T}\big)^{-1}
-\big(\lambda-L^{3}_{0,T}\big)^{-1}\big]s\Big|_{\var, 0}
\leqslant C \var\big(1+|\lambda|\big)^{k}\big|s\big|_{0, (0, k')}.
\end{align}
\end{lemma}
\begin{proof} Clearly,
\begin{align}\label{e06047}
\big(\lambda-L^{3}_{\var,T}\big)^{-1}-\big(\lambda-L^{3}_{0,T}\big)^{-1}
=\big(\lambda-L^{3}_{\var,T})^{-1}
\big(L^{3}_{\var,T}-L^{3}_{0,T}\big)\big(\lambda-L^{3}_{0,T}\big)^{-1}.
\end{align}
Since $|\cdot|_{\var, 0}\leqslant |\cdot|_{0, 0}$, then by (\ref{e06023}),
\begin{align}\label{e06048}
\Big|\big\langle(L^{3}_{\var,T}-L^{3}_{0,T})s, s'\big\rangle_{\var, 0}\Big|
\leqslant C\var\big|s\big|_{0, (1, 4)}\big|s'\big|_{\var, 1},
\end{align}
which implies that
\begin{align}\label{e06049}
\big|(L^{3}_{\var,T}-L^{3}_{0,T})s\big|_{\var, -1}
\leqslant C\var\big|s\big|_{0, (1, 4)}.
\end{align}
On the other hand, we have
\begin{align}\label{e06052}
\Big|\la\big[U_{i_{1}}, [\cdots [U_{i_{p}},
\ L^{3}_{0,T}]\cdots] s,s'\ra\big]\Big|_{0,0}\leqslant C_{p}|s|_{0,1}|s'|_{0,1}.
\end{align}
From (\ref{e06052}) and the argument as in the proof of Theorem \ref{e04126}, we obtain
\begin{align}\label{e06059}
\big|(\lambda-L^{3}_{0,T})^{-1}s\big|_{0, (1, k)}
\leqslant C\big(1+|\lambda|\big)^{k}\big|s\big|_{0, (0, k)}.
\end{align}
This completes the proof of Lemma \ref{e06045}.
\end{proof}

By (\ref{e06028}) and Lemma \ref{e06045}, there exists $r\in \N$ for $s\in J_{p,y_0}^0$,
\begin{align}\label{e493}
|((\lambda-L_{\var,T}^3)^{-1}-(\lambda-L_{0,T}^3)^{-1})s|_{(p),0}\leq c\var(1+|\lambda|)^{2}(1+p)^{r}|s|_{(p),0}.
\end{align}
So  there exists $C>0$, $r\in \N$,
such that for $\var\in (0,1]$, $p\in \N$,
\begin{align}\label{e494}
|(\exp(-L_{\var,T}^3)-\exp(-L_{0,T}^3))s|_{(p),0}\leq C\var(1+p)^{r}|s|_{(p),0}.
\end{align}

By the same process in (\ref{e04145})-(\ref{e04148}),
there exist $c>0, C>0, r\in\N$,
such that for any $(U,x), (U',x')\in T_{y_0}Y\times X_{y_0}$, $\var\in (0,1]$,
\begin{multline}
|(\exp(-L_{\var,T}^3)-\exp(-L_{0,T}^3))((U,x), (U',x'))|
\\
\leq c\var^{(\dim Y+1)^{-1}}(1+|U|+|U'|)^r\exp(-C|U-U'|^2).
\end{multline}

Then the proof of Theorem \ref{e107} is complete.

\section{Proof of Theorem \ref{e03022} iii)}

In this Section, we use the notations and assumptions in Section 2.2 and 6.

\subsection{Localization of the problem near $\pi_1^{-1}(V^g)$}

We replace $T$ by $u$ and $T/\var$ by $T'$.

By Lemma \ref{e06010}, there exist $C_1,C_2>0$, such that for any $z,z'\in Z_b$ and $u\in (0,1]$, $T'\geq 1$,
\begin{align}\label{e06030}
\left|\widetilde{\mathbf{G}}_{u^2/T'^2}\left(\frac{u^2}{T'^2}\mB_{T'}'\right)(z,z')\right|\leq C_1\exp\left(-\frac{C_2T'^2}{u^2}\right),
\end{align}
and
\begin{align}\label{e06031}
\left|\psi_S\widetilde{\tr}\left[g\widetilde{\mathbf{G}}_{{u^2/T'^2}}
\left(\mB_{u/T',T'}'\right)\right]\right|\leq C_1\exp\left(-\frac{C_2T'^2}{u^2}\right).
\end{align}

We trivialize the bundle $\pi_3^*\Lambda (T^{*}S)\widehat{\otimes}\mS(TZ, L_Z)$ as in Section \ref{s0502}.
By (\ref{e06019}),
we can get
\begin{align}\label{e684}
L_{u/T',u}^1=u^2\delta_{u^2}L_{1/T',1}^1\delta_{u^2}^{-1}.
\end{align}

Comparing with (\ref{e565}),
there exists $C>0$, such that for $|U|<\alpha_0/4$,
\begin{align}\label{e683}
\begin{split}
&\left|\exp\left(-u^2\mB_{1/T'}'^2\right)(g^{-1}(U,x),(U,x))k_Y(y_0,U)\right.
\\
&\quad\quad \left.-\exp(-u^2L_{1/T',1}^1)(g^{-1}(U,x),(U,x))\right|\leq C\exp\left(-\frac{C_2T'^2}{u^2}\right).
\end{split}
\end{align}
Then we can replace the fiber $Z$ by $T_{y_0}Y\times X_{y_0}$ for $y_0\in V^g$.

\subsection{Proof of Theorem \ref{e03022} iii)}

We will use the notation of Section \ref{s0502} with $\var$ replaced by $1/T'$, and $T$ by $1$.
By Lemma \ref{e593}, we see that as $T'\rightarrow +\infty$
\begin{align}\label{e608}
L_{1/T',1}^3\rightarrow L_{0,1}^3.
\end{align}

Let $\exp(-u^2L_{\var,T}^i)((U,x),(U',x'))\ ((U,x), (U'x')\in T_{y_0}Y\times X_{y_0})\ (i=1,2,3)$
be the smooth kernel associated to the operator $\exp(-u^2L_{\var,T}^i)$ with respect to
$dv_{T_{y_0}Y}(U')dv_{X_{y_0}}(x')$. %For $U\in N_{Y^g/Y,y_0}, x\in X_{y_0}$, set
Then by (\ref{e570}),
\begin{multline}\label{e610}
\psi_S\int_{Y^g}\int_{U\in N,\atop |U|\leq \alpha_0/4}\int_X\delta_{u^2}\widetilde{\tr}\left[g\exp\left(-u^2L_{1/T',1}^1\right)
\left(g^{-1}(U,x),\left(U,x\right)\right)\right]
\\
\cdot dv_{Y^g}dv_N(U)dv_X(x)
\\
=\psi_S\int_{Y^g}\int_{U\in N,\atop |U|\leq T'\alpha_0/4}\int_X\tilde{c}_{TY^g}\delta_{u^2}\widetilde{\tr}\left[g\exp(-u^2L_{1/T',1}^3)
\left(g^{-1}(U,x),\left(U,x\right)\right)\right]^{max}
\\
\cdot dv_{Y^g}dv_N(U)dv_X(x).
\end{multline}
By (\ref{e610}) and the argument of Section \ref{s0502}, to calculate the asymptotic of the left hand side of (\ref{e610})
as $u\rightarrow 0$ uniformly in $T\geq 1$, we have to find the asymptotic as $u\rightarrow 0$ of
\begin{align}\label{e611}
\psi_S\int_{U\in N}\int_X\tilde{c}_{TY^g}\delta_{u^2}\widetilde{\tr}\left[g\exp(-u^2L_{1/T,1}^3)
\left(g^{-1}(U,x),\left(U,x\right)\right)\right]^{max}dv_N(U)dv_X(x).
\end{align}

The following lemma is a modification of Lemma \ref{e107}.
\begin{lemma}\label{e613}
There exist $C_1,C_2>0, p,r\in\N$ such that for any $(U,x),(U',x')\in T_{y_0}Y\times X_{y_0}$,
$\var\in [0,1]$, $u\in (0,1]$,
\begin{align}\label{e614}
\begin{split}
|u^p\exp(-u^2&L_{\var,1}^3)((U,x),(U',x'))|
\\
&\leq C_1(1+|U|+|U'|)^r\cdot \exp\left(-C_2\frac{|U-U'|^2+d^X(x,x')^2}{u^2}\right).
\end{split}
\end{align}
\end{lemma}
\begin{proof}
By (\ref{e478}),
\begin{align}\label{e691}
\begin{split}
&|\Delta^{k}\exp(-u^2L_{\var,1}^3)\Delta^{k'}s|_{\var,0}
\leq C\left(\int_{\Gamma}e^{-u^2\lambda}(1+|\lambda|)^{p_m}d\lambda\right)|s|_{\var,0}
\\
\leq &Cu^{-2p_m-2}\left(\int_{u^2\Gamma}e^{-\lambda}(1+|\lambda|)^{p_m'}d\lambda\right)|s|_{\var,0}\leq Cu^{-2p_m-2}|s|_{\var,0}.
\end{split}
\end{align}
So, there exists $p\in \N$, such that
\begin{align}\label{e692}
|u^p\Delta^{k}\exp(-u^2L_{\var,1}^3)\Delta^{k'}s|_{\var,0}\leq C|s|_{\var,0}.
\end{align}
Following the process in (\ref{e06032})-(\ref{e482}), we have
\begin{align}\label{e694}
|u^p\exp(-u^2L_{\var,1}^3)((U,x),(U',x'))|\leq C(1+|U|+|U'|)^r.
\end{align}

Following the process in (\ref{e485})-(\ref{e489}),
We get Lemma \ref{e613}.
\end{proof}

Let $N_{X^g/X}$ be the normal bundle to $X^g$ in $X$. We identify $N_{X^g/X}$ to the orthogonal bundle to $TX^g$ in $TX$.
Let $g^{N_{X}}$ be the metric on $N_{X^g/X}$ induced by $g^{TX}$. Let $dv_{N_X}$ be the Riemannian volume form on $(N_{X^g/X},g^{N_X})$.

For $U\in T_{y_0}Y$, $x\in X^g$, $V\in N_{X^g/X}$, $|U|, |V|\leq \alpha_0/4$, let
$k_X(U,x,V)$ be defined by
\begin{align}\label{e616}
dv_X(U,x,V)=k_X(U,x,V)dv_{N_{X^g/X}}(V)dv_{X^g}(x).
\end{align}
Set $n'=\dim Z^g$.
By standard results on heat kernel (cf. \cite[Theorem 6.11]{MR2273508}),
there exist smooth functions $a_{T,-n'}'(x),\cdots,$$ a_{T,0}'(x)\, (x\in W^g)$
such that as $u\rightarrow 0$, for $x\in X_{y_0}^g$,
\begin{align}\label{e617}
\begin{split}
\int_{\scriptstyle V\in N_{X}, U\in N_{Y}, \atop \scriptstyle |U|,|V|\leq \alpha_0/4}
&\delta_{u^2}\widetilde{\tr}\left[g\exp(-u^2L_{1/T',1}^3)
\left(g^{-1}(U,x,V),\left(U,x,V\right)\right)\right]^{max}
\\
&\cdot k_X(U,x,V)dv_{N_X}dv_{N_Y}
=\sum_{j=-n'}^{0} a_{T',j}'(x)u^{j}+O(u),
\end{split}
\end{align}
where the $a_{T',j}'(x)$ only depend on the operator $L_{1/T',1}^3$ and its higher derivatives on $x$.
By (\ref{e608}), $a_{T',j}'(x)$ is continuous on $T'\in [1,+\infty]$.

By (\ref{e06014}), (\ref{e608})-(\ref{e614}) and (\ref{e617}), there exist $a_{T',j}$ depending continuously on $T'\in [1,+\infty]$
such that for any $u\in (0,1]$, $T'\in [1,+\infty]$,
\begin{align}\label{e618}
\left|\psi_S\widetilde{\tr}\left[g\exp\left(-\mB_{u/T',T'}'\right)\right]
-\sum_{j=-n'}^{0} a_{T',j}u^{j}\right|\leq Cu.
\end{align}

Since $\var= u/T'$, (\ref{e618}) is reformulated by
\begin{align}\label{e06034}
\left|\psi_S\widetilde{\tr}\left[g\exp\left(-\mB_{\var,T'}'\right)\right]
-\sum_{j=-n'}^{0} a_{T',j}(\var T')^{j}\right|\leq C\var T'.
\end{align}
Following the process in (\ref{e05050})-(\ref{e05043}), we have
\begin{align}\label{e06063}
\left|\left\{\psi_S\widetilde{\tr}\left[g\exp\left(-\mB_{\var,T'}'\right)\right]\right\}^{dT'}
-\sum_{j=-n'}^{0} [a_{T',j}]^{dT'}(\var T')^{j}\right|\leq C\var .
\end{align}

For $T'\geq 1$ fixed,
by Theorem \ref{e01061} and (\ref{e03008}), we have
\begin{align}\label{e07003}
\begin{split}
\lim_{\var\rightarrow 0}\left\{\psi_S\widetilde{\tr}\left[g\exp\left(-\mB_{\var,T'}'\right)\right]\right\}^{dT'}
=-\int_{Z^g}\gamma_{\mA}(T')\wedge \ch_g(L_Z^{1/2}, \nabla^{L_Z^{1/2}})\wedge \ch_g(E, \nabla^E).
\end{split}
\end{align}
From (\ref{e06034}) and (\ref{e07003}),
\begin{align}\label{e620}
\begin{split}
[a_{T',j}]^{dT'}=0 \quad \text{if}\quad j<-1, \quad [a_{T',0}]^{dT'}=-\int_{Z^g}\gamma_{\mA}(T')\wedge  \ch_g(L_Z^{1/2}, \nabla^{L_Z^{1/2}})\wedge \ch_g(E, \nabla^E).
\end{split}
\end{align}

Since $T'=\var T$,
\begin{align}\label{e06035}
[a_{T',j}]^{dT}=\var^{-1}[a_{T',j}]^{dT'}.
\end{align}
From  (\ref{e620}) and (\ref{e06035}), comparing the coefficients of $dT$ in (\ref{e06034}),
we have
\begin{align}\label{e06036}
\left|\left\{\psi_S\widetilde{\tr}\left[g\exp\left(-\mB_{\var,T/\var}'\right)\right]\right\}^{dT}
+\var^{-1} \int_{Z^g}\gamma_{\mA}(T/\var)\wedge \ch_g(L_Z^{1/2}, \nabla^{L_Z^{1/2}})\wedge \ch_g(E, \nabla^E)\right|\leq C.
\end{align}

By (\ref{e06005}) and (\ref{e06036}), we get Theorem \ref{e03022} iii).

\section{Proof of Theorem \ref{e03022} iv)}

In this section, we prove Theorem \ref{e03022} iv) by following the process of \cite[Section IX]{MR1305280} and \cite[Section 9]{MR1800127}.
In Section 8.1, as in Section \ref{s0601}, we reduce the problem to a local problem near $\pi_1^{-1}(V^g)$.
In Section 8.2, we study the matrix structure of $L_{\var, T}^3$ as in Section 4.2. In Section 8.3, we prove Theorem \ref{e03022} iv).

We use the same notation as in Section 4, 6 and the assumptions in Section 2.2.

\subsection{Finite propagation speed and localization}

\begin{prop}\label{e621}
There exist $C>0$, $C'>0$, $\delta>0$, $T_0\geq 1$, such that for $0<\var\leq 1$, $T\geq T_0$,
\begin{align}\label{e622}
\left|\left\{\psi_S\widetilde{\tr}\left[g\widetilde{\textbf{G}}_{\var^2}
(\mB_{\var, T/\var}')\right]\right\}^{dT}\right|\leq \frac{C}{T^{1+\delta}}.
\end{align}
\end{prop}
\begin{proof}

As we noted in Section \ref{s05}, if we replace $\mB_T$ by $\mB_{T/\var}'$ and $\mB_2$ to $B_2$, everything
in Section \ref{s04} works well. So there exist $C>0$, $\delta>0$, $T_0\geq 1$, such that for $0<\var\leq 1$, $T\geq T_0$,
\begin{align}
\left|\psi_S\widetilde{\tr}\left[g\widetilde{\textbf{G}}_{\var^2}
(\var^2\mB_{T/\var}')\right]-\psi_S\widetilde{\tr}\left[g\widetilde{\textbf{G}}_{\var^2}
(\var^2B_{2})\right]\right|\leq \frac{C}{T^{\delta}}.
\end{align}
Since the second term above does not involve $dT$ part, by (\ref{e06004}) and
following the argument in (\ref{e04149})-(\ref{e05043}), we get Proposition \ref{e621}.
\end{proof}

By Proposition \ref{e621}, to establish Theorem \ref{e03022} iv), we only need to prove
the following result.

\begin{thm}\label{e628}
There exist $C>0$, $C'>0$, $\delta>0$, and $T_0\geq 1$
such that for $0<\var\leq 1$, $T\geq T_0$
\begin{align}\label{e629}
\left|\left\{\psi_S\widetilde{\tr}\left[g\widetilde{\textbf{F}}_{\var^2}
(\mB_{\var,T/\var}')\right]\right\}^{dT}\right|\leq \frac{C}{T^{1+\delta}}.
\end{align}
\end{thm}

By the finite propagation speed as in (\ref{e06054}), if $x\in W$, $\widetilde{\textbf{F}}_{\var^2}
(\mB_{\var,T/\var}')(x,\cdot)$ only depends on the restriction of $\mB_{\var, T/\var}'$
to $\pi^{-1}(B^{Y}(\pi_1x,\alpha))$.

Now we can use the same argument as discussed in (\ref{e06054})-(\ref{e06014}) to know
 the proof of Theorem \ref{e628} is local near $\pi_1^{-1}(V^g)$.

\subsection{The matrix structure of the operator $L_{\var,T}^3$ as $T\rightarrow +\infty$}

We use the same trivialization and notations as in Section \ref{s0601}.

By (\ref{e570}),
\begin{align}\label{e630}
\begin{split}
&\int_{Y^g}\int_{U\in N_Y, \atop |U|\leq \alpha_0/4}\widetilde{\tr}[g
\widetilde{\textbf{F}}_{\var^2}(L_{\var,T}^1)(g^{-1}(U,x),(U,x))]dv_{N_Y}dv_{Y^g}
\\
=&\int_{Y^g}\int_{U\in N_Y, \atop |U|\leq \alpha_0/4\var}\tilde{c}_{TY^g}\widetilde{\tr}\left[g\widetilde{\textbf{F}}_{\var^2}(L_{\var,T}^3)
\left(g^{-1}\left(U,x\right),\left(U,x\right)\right)\right]dv_{N_Y}.
\end{split}
\end{align}

Recall that the vector bundle $K$ was defined in the argument before (\ref{e06018}) and  the operator $S_{\var}$ was defined in (\ref{e566}). Let
$\mathbb{F}_{\var}^0$ be the vector space of square integrable sections of
$\Lambda(T^*V^g)\widehat{\otimes}\mS(N_{Y^g/Y})\widehat{\otimes}S_{\var}^{-1*}K\otimes L_Y^{1/2}$
over $T_{y_0}Y$. Then $\mathbb{F}_{\var}^0$ is a Hilbert subspace of $I^0$.
Let $\mathbb{F}_{\var}^{0,\bot}$ be its orthogonal complement in $I^0$.
Let $p_{\var}$ be the orthogonal projection operator from
$I^0$ on $\mathbb{F}^0_{\var}$. Set $p_{\var}^{\bot}=1-p_{\var}$. Then if $s\in I^0$,
\begin{align}\label{e632}
p_{\var}s(U)=P_{\var U}^Ks(U,\cdot)\quad U\in T_{y_0}Y.
\end{align}
Put
\begin{align}\label{e633}
\begin{split}
&E_{\var,T}=p_{\var}L_{\var,T}^3p_{\var},\ \ F_{\var,T}=p_{\var}L_{\var,T}^3p_{\var}^{\bot},\ \
\\
&G_{\var,T}=p_{\var}^{\bot}L_{\var,T}^3p_{\var},\ \ H_{\var,T}=p_{\var}^{\bot}L_{\var,T}^3p_{\var}^{\bot}.
\end{split}
\end{align}
Then we write $L_{\var, T}^3$ in matrix form with respect to the splitting
$I^0=\mathbb{F}_{\var}^0\oplus\mathbb{F}_{\var}^{0,\bot}$,
\begin{align}\label{e635}
\begin{split}
L_{\var,T}^3=\left(
  \begin{array}{cc}
    E_{\var,T} & F_{\var,T} \\
    G_{\var,T} & H_{\var,T} \\
  \end{array}
\right).
 \end{split}
\end{align}

The following lemma is an analogue of Proposition \ref{e04023}.
\begin{lemma}\label{e698}
There exist operators $E_\var$, $F_{\var}$, $G_{\var}$, $H_\var$
such that as $T\rightarrow \infty$,
\begin{align}\label{e699}
\begin{split}
&E_{\var,T}=E_{\var}+O(1/T),\quad \quad F_{\var,T}=TF_{\var}+O(1),\\
&G_{\var,T}=TG_{\var}+O(1),\ \ \quad \quad H_{\var,T}=T^2H_{\var}+O(T).
\end{split}
\end{align}

Set
\begin{align}\label{e701}
\begin{split}
Q_{\var}:=\rho^2(\var U)R_{\var}S_{\var}^{-1}\left[D^X, \var D^H+\,^0\nabla^{\cE_Z, u}\right]S_{\var}.
\end{split}
\end{align}
Then $Q_{\var}$ maps $\mathbb{F}_{\var}^0$ into $\mathbb{F}_{\var}^{0,\bot}$. Moreover,
\begin{align}\label{e702}
\begin{split}
&F_{\var}=p_{\var}Q_{\var}p_{\var}^{\bot},
\\
&G_{\var}=p_{\var}^{\bot}Q_{\var}p_{\var},
\\
&H_{\var}=p_{\var}^{\bot}(\rho^2(\var|U|)D_{\var U}^{X,2}+(1-\rho^2(\var U))D_{y_0}^{X,2})p_{\var}^{\bot}.
\end{split}
\end{align}
\end{lemma}
\begin{proof}
From (\ref{e06004}), (\ref{e06001}), (\ref{e567}) and (\ref{e569}), we find the coefficient of $T^2$ in the expansion
of $L_{\var,T}^3$ is given by
\begin{align}\label{e704}
\begin{split}
H_{\var}
=(1-\rho^2(\var|U|))P_{\var U}^{K,\bot}D_{y_0}^{X,2}P_{\var U}^{K,\bot}+\rho^2(\var |U|)D_{\var U}^{X,2}.
\end{split}
\end{align}
When $\rho(\var|U|)\neq 0$, $K_{\var U}=\ker D_{\var U}^{X,2}$. So
\begin{align}\label{e705}
\begin{split}
H_{\var}=P_{\var U}^{K,\bot}\left((1-\rho^2(\var|U|))D_{y_0}^{X,2}+\rho^2(\var |U|)D_{\var U}^{X,2}\right)P_{\var U}^{K,\bot}.
\end{split}
\end{align}
Using (\ref{e632}), we see that (\ref{e705}) fits with the last formula in (\ref{e702}).

By (\ref{e06004}), (\ref{e06001}), (\ref{e567}) and (\ref{e569}),
we find that the coefficient of $T$ in the expansion of $L_{\var, T}^3$ is the operator $Q_{\var}$.

Using (\ref{e701}), it is clear that $Q_{\var}$
maps
$\mathbb{F}_{\var}^0$ into $\mathbb{F}_{\var}^{0,\bot}$. Also (\ref{e699}) and the remaining equations in (\ref{e702}) follow.

The proof of Theorem \ref{e698} is complete.
\end{proof}

Clearly, for $U\in T_{y_0}Y$, $H_{\var U}$, the operator $H_{\var}$ at $U$, is an elliptic operator acting along $X_{y_0}$.
\begin{prop}\label{e707}
For any $\var>0$,
\begin{align}\label{e708}
\ker H_{\var U}=\Lambda(T^*V^g)\widehat{\otimes}\mS(N_{Y^g/Y})\widehat{\otimes}K_{\var U}\otimes L_Y^{1/2}.
\end{align}
\end{prop}
\begin{proof}
By (\ref{e702}), if $s\in \Lambda(T^*V^g)\widehat{\otimes}\mS(N_{Y^g/Y})\widehat{\otimes}K_{\var U}\otimes L_Y^{1/2}$, then
\begin{align}\label{e709}
H_{\var}s=0.
\end{align}

The operator $H_{\var U}$ is self-adjoint and nonnegative. Therefore if $H_{\var}s=0$, then
\begin{align}\label{e710}
\begin{split}
&P_{\var U}^{K,\bot}\rho^2(\var|U|)D_{\var U}^{X,2}P_{\var U}^{K,\bot}s=0,
\\
&P_{\var U}^{K,\bot}(1-\rho^2(\var U))D_{y_0}^{X,2})P_{\var U}^{K,\bot}s=0.
\end{split}
\end{align}

If $\rho^2(\var|U|)\neq 0$, we deduce from the first identity in (\ref{e710}) that
$P_{\var U}^{K,\bot}s=0$, i.e. $s\in \Lambda(T^*V^g)\widehat{\otimes}\mS(N_{Y^g/Y})\widehat{\otimes}K_{\var U}\otimes L_Y^{1/2}$.
If $\rho^2(\var|U|)=0$, by the second identity in (\ref{e710}), $P_{\var U}^{K,\bot}s\in \ker D_{y_0}^X$. Using (\ref{e06018}),
we deduce that $P_{\var U}^{K,\bot}s=0$, i.e., $s\in \Lambda(T^*V^g)\widehat{\otimes}\mS(N_{Y^g/Y})\widehat{\otimes}K_{\var U}\otimes L_Y^{1/2}$.

The proof of proposition \ref{e707} is complete.
\end{proof}

\subsection{Proof of Theorem \ref{e628}}
%Let $\nabla$ denote the gradient in the variable $U$.
For $s\in I$,
put
\begin{align}\label{e789}
|s|_{\var,T,1}^{2}:=|P_{\var U}^Ks|_{\var,0}^{2}+T^2|P_{\var U}^{K,\bot} s|_{\var,0}^{2}
+\sum_p|\nabla_{f_p}s|_{\var,0}^{2}+T^2\sum_i|^0\nabla^{\mS_Z\otimes E}_{e_i}P_{\var U}^{K,\bot} s|_{\var,0}^{2}.
\end{align}

\begin{lemma}\label{e790}
There exist $c_1,c_2,c_3, c_4>0$, $T_0\geq 1$, such that for any $s,s'\in I$ with compact support, $\var\in (0,1]$, $T\geq T_0$, we have
\begin{align}\label{e791}
\begin{split}
\Re\la L_{\var,T}^3s,s\ra_{\var,0}^{}&\geq c_1|s|_{\var,T,1}^{2}-c_2|s|_{\var,0}^{2},
\\
|\Im\la L_{\var,T}^3s,s\ra_{\var,0}^{}|&\leq c_3|s|_{\var,T,1}^{}|s|_{\var,0}^{},
\\
|\la L_{\var,T}^3s,s'\ra_{\var,0}^{}|&\leq c_4|s|_{\var,T,1}^{}|s'|_{\var,T,1}^{}.
\end{split}
\end{align}
\end{lemma}
\begin{proof}
By (\ref{e06004}), (\ref{e06001}), (\ref{e567}) and (\ref{e569}), the 2-order term of the differential operator $L_{\var, T}^3$ is a
fiberwise elliptic operator
\begin{align}\label{e987}
T^2H_{\var}+\Delta^{TY}.
\end{align}
From (\ref{e701}), since $K$ is a vector bundle over $T_{y_0}Y\times S$, for $s\in I$ with compact support, there exists $C_1>0$, such that
\begin{align}\label{e989}
\la H_{\var}P_{\var U}^{K,\bot}s, P_{\var U}^{K,\bot}s \ra_{\var,0}^{}\geq C_1 |P_{\var U}^{K,\bot}s|_{\var,0}^{2}.
\end{align}
Since $H_{\var}$ is a fiberwise selfadjoint elliptic operator along the fibers $X$,
from the elliptic estimates, there exist $C_2, C_3>0$, such that
\begin{align}\label{e990}
\la H_{\var}P_{\var U}^{K,\bot} s,P_{\var U}^{K,\bot} s\ra_{\var,0}^{}\geq C_2\sum_{i}|\,^0\nabla^{\mS_Z\otimes E}_{e_i} P_{\var U}^{K,\bot}
 s|_{\var,0}^{2}
-C_3 |P_{\var U}^{K,\bot}s|_{\var,0}^{2}.
\end{align}
From (\ref{e989}) and (\ref{e990}), there exists $C_4>0$, such that
\begin{align}\label{e06026}
\la H_{\var}P_{\var U}^{K,\bot} s,P_{\var U}^{K,\bot} s\ra_{\var,0}^{}\geq C_4\left(\sum_{i}|\,^0\nabla^{\mS_Z\otimes E}_{e_i} P_{\var U}^{K,\bot}
 s|_{\var,0}^{2}
+ |P_{\var U}^{K,\bot}s|_{\var,0}^{2}\right).
\end{align}

By (\ref{e06071}),
there exist $C_5, C_6>0$, such that
\begin{align}\label{e991}
\la \Delta^{TY} s, s\ra_{\var,0}^{}\geq C_5\sum_p|\nabla_{f_p}s|_{\var,0}^{2}-C_6|s|_{\var,0}^2.
\end{align}
Then
there exist $C_1', C_2'>0$, such that
\begin{align}\label{e992}
\begin{split}
\la (T^2H_{\var}+\Delta^{TY})s, s\ra_{\var,0}&\geq C_1'|s|_{\var,T,1}^2-C_2'|s|^2_{\var,0}.
\end{split}
\end{align}

By Lemma \ref{e448} and (\ref{e701}), there exist $C>0$, such that
\begin{align}\label{e798}
|\la TQ_{\var} s,s\ra_{\var,0}|\leq C|s|_{\var,T,1}|s|_{\var,0}.
\end{align}

Then  Lemma \ref{e790} follows from (\ref{e06071}), (\ref{e992})  and  (\ref{e798}).
\end{proof}

Set $\mD_{\var}=\{P_{\var U}^K\partial_p P_{\var U}^K+P_{\var U}^{K,\bot}\partial_p P_{\var U}^{K,\bot},
\ P_{\var U}^{K,\bot}\,\nabla_{e_i}^{\mS_X\otimes E} P_{\var U}^{K,\bot}\}$.

Let $\Xi_{\var}$ be the operator from $\mathbb{F}_{\var}$ to itself,
\begin{align}\label{e08003}
\Xi_{\var}=E_{\var}-F_{\var}H_{\var}^{-1}G_{\var}.
\end{align}

Following the same argument in (\ref{e04099})-(\ref{e04127}), we can get an analogue of Theorem \ref{e04223}.
\begin{thm}\label{e08006}
There exist $C>0$, $\delta>0$, and $T_0\geq 1$
such that for $0<\var\leq 1$, $T\geq T_0$,
\begin{align}\label{e08007}
\left|\psi_S\widetilde{\tr}\left[g\widetilde{\textbf{F}}_{\var^2}
(L_{\var,T}^3)\right]-\psi_S\widetilde{\tr}\left[g\widetilde{\textbf{F}}_{\var^2}
(\Xi_{\var})\right]\right|\leq \frac{C}{T^{\delta}}.
\end{align}
\end{thm}

Since there is no $dT$ part in the second term above,
as in (\ref{e04149})-(\ref{e05043}), we get Theorem \ref{e628}.

\vspace{3mm}\textbf{Acknowledgements}\ \
The author would like to thank Prof. Xiaonan Ma for many discussions.
He would like to thank the Mathematical Institute of
K\"oln University, especially Prof. George Marinescu for financial support and hospitality.
The author was supported by the DFG funded Project MA 2469/2-2
and by the Key Profile Area "Quantum Matter and Materials" in the framework of
the Institutional Strategy of the University of Cologne within the German Excellence Initiative.

\end{document}